\documentclass[11pt,a4paper,reqno,intlimits]{amsart}

\usepackage[margin=1.00in]{geometry}
\usepackage{amsthm,amssymb,amsmath,amsfonts}
\usepackage[utf8]{inputenc}
\usepackage[dvips]{epsfig}
\usepackage{graphicx}
\usepackage[english]{babel}
\usepackage{soul}
\usepackage{hyperref}
\hypersetup{colorlinks=true,linkcolor=blue,citecolor=red,linktoc=page}

\usepackage{mathptmx}
\usepackage{cite}
\usepackage{esint}

\usepackage[nameinlink]{cleveref}
\usepackage[hyperpageref]{backref}

\allowdisplaybreaks
\usepackage{color}
\usepackage[dvipsnames]{xcolor}

\usepackage{tikz}

\def\sideremark#1{\ifvmode\leavevmode\fi\vadjust{\vbox to0pt{\vss
 \hbox to 0pt{\hskip\hsize\hskip1em
 \vbox{\hsize2.1cm\tiny\raggedright\pretolerance10000
  \noindent #1\hfill}\hss}\vbox to15pt{\vfil}\vss}}}%

\let\oldsqrt\sqrt
\def\sqrt{\mathpalette\DHLhksqrt}
\def\DHLhksqrt#1#2{%
\setbox0=\hbox{$#1\oldsqrt{#2\,}$}\dimen0=\ht0
\advance\dimen0-0.2\ht0
\setbox2=\hbox{\vrule height\ht0 depth -\dimen0}%
{\box0\lower0.4pt\box2}}

\newcommand{\R}{\mathbb{R}} 
\newcommand{\N}{\mathbb{N}} 
\newcommand{\Z}{\mathbb{Z}} 

\newcommand{\inn}{\textnormal{int}} 
\newcommand{\dist}{\textnormal{dist}} 
\newcommand{\diam}{\textnormal{diam}} 
\newcommand{\supp}{\textnormal{supp}} 
\newcommand{\essinf}{\textnormal{essinf}} 

\renewcommand{\phi}{\varphi}
\renewcommand{\div}{\textnormal{div}}

\newcommand{\rz}{\mathbb{R}}

\newcommand{\cD}{{\mathcal D}}
\newcommand{\cE}{{\mathcal E}}
\newcommand{\cF}{{\mathcal F}}
\newcommand{\cG}{{\mathcal G}}

\newcommand{\cL}{{\mathcal L}}

\newcommand{\cW}{{\mathcal W}}

\newcommand{\eps}{\varepsilon}
\newcommand*\re{\mathbb{R}}
\newcommand*\rn{\mathbb{R}^N}

\theoremstyle{definition}
\newtheorem{defi}{Definition}[section]
\newtheorem{remark}[defi]{Remark}

\theoremstyle{plain} 
\newtheorem{thm}[defi]{Theorem}
\newtheorem{prop}[defi]{Proposition}
\newtheorem{lemma}[defi]{Lemma}
\newtheorem{cor}[defi]{Corollary}
\newtheorem{definition}[defi]{Definition}
\theoremstyle{definition}

\numberwithin{equation}{section}

\title[logarithmic $p$-Laplacian]{The Dirichlet problem for the Logarithmic $p$-Laplacian}

\author[Dyda]{Bart{\l}omiej Dyda}
\address[B. Dyda]{Faculty of Pure and Applied Mathematics\\ Wroc{\l}aw University 
	of Science and Technology\\
	Wybrze\.ze Wyspia\'nskiego 27,
	50-370 Wroc{\l}aw, Poland
}
\email{bdyda@pwr.edu.pl\quad dyda@math.uni-bielefeld.de}

\author[Jarohs]
{Sven Jarohs}
\address[S. Jarohs]{Institut f\"ur Mathematik,
Goethe-Universit\"at Frankfurt.
Robert-Mayer-Str. 10
D-60629 Frankfurt am Main, Germany}
\email{jarohs@math.uni-frankfurt.de}

\author[Sk]
{Firoj Sk}
\address[F. Sk]{Carl von Ossietzky Universit\"{a}t Oldenburg,
Fakult\"{a}t V,
Institut f\"{u}r Mathematik,
Ammerländer Heerstraße 114-118,
26129 Oldenburg,
Germany}
\email{firoj.sk@uni-oldenburg.de\quad firojmaciitk7@gmail.com}

\keywords{Fractional $p$-Laplacian, fractional Sobolev spaces, maximum principle, comparison principle, Hardy inequality, eigenvalue problem.}

\subjclass{Primary: 35D30; 35B50; 35B51; 35P30; 35R11, Secondary: 46E35; 35A23.}

\begin{document}

\begin{abstract}
We introduce and study the logarithmic $p$-Laplacian $L_{\Delta_p}$, which emerges from the formal derivative of the fractional $p$-Laplacian $(-\Delta_p)^s$ at $s=0$. This operator is nonlocal, has logarithmic order, and is the nonlinear version of the newly developed logarithmic Laplacian operator \cite{CW19}. We present a variational framework to study the Dirichlet problems involving the $L_{\Delta_p}$ in bounded domains. 

This allows us to investigate the connection between the first Dirichlet eigenvalue and eigenfunction of the fractional $p$-Laplacian and the logarithmic $p$-Laplacian. As a consequence, we deduce a Faber-Krahn inequality for the first Dirichlet eigenvalue of $L_{\Delta_p}$. We discuss maximum and comparison principles for $L_{\Delta_p}$ in bounded domains and demonstrate that the validity of these depends on the sign of the first Dirichlet eigenvalue of $L_{\Delta_p}$. In addition, we prove that the first Dirichlet eigenfunction of $L_{\Delta_p}$ is bounded.
Furthermore, we establish a boundary Hardy-type inequality for the spaces associated with the weak formulation of the logarithmic $p$-Laplacian.
\end{abstract}

\maketitle

\tableofcontents

\section{Introduction}
In recent decades there has been a growing interest in the understanding of boundary value problems involving nonlocal integro-differential operators --- the most prominent example is given by the \textit{fractional Laplace} operator. Fueled by this interest, also \textit{nonlinear nonlocal} interactions have been studied extensively \cite{L14,L16,IMS16,BP16,PQ16,MPSY16,CKP16,ILPS16,BLS18,IMP23} with a prominent example being given by the \textit{fractional $p$-Laplace} operator. This operator is given by
\begin{align*}
(-\Delta_p)^su(x)&=C_{N,s,p}\,p.v.\int_{\R^N}\frac{|u(x)-u(y)|^{p-2}(u(x)-u(y))}{|x-y|^{N+sp}}\,dy,
\end{align*}
where $p\in(1,\infty)$, $s\in(0,1)$, p.v.\ stands for the Cauchy principal value, and $u$ is suitably regular at $x\in\R^N$ and integrable at infinity with respect to the kernel $z\mapsto |z|^{-N-sp}$. The constant $C_{N,s,p}$ here is chosen in the particular case $p=2$ such that $(-\Delta_2)^s$ has the Fourier symbol $|\cdot|^{2s}$. With this, $(-\Delta_2)^s$ can easily be seen as an intermediate operator between the identity operator and the Laplacian $-\Delta=-\sum_{k=1}^{N}\partial_{kk}$. Similarly, for $p\neq 2$, the fractional $p$-Laplace operator can be seen as an intermediate operator between the identity and the $p$-Laplacian $-\Delta_p$ given by $-\Delta_pu=-\div(|\nabla u|^{p-2}\nabla u)$. However, in this general case an appropriate constant is not clear, see the discussions in \cite{TGV21,F23}. Here, we choose $C_{N,s,p}$ such that the limits
$$
\lim_{s\to0^+}(-\Delta_p)^su(x)=|u(x)|^{p-2}u(x)\quad\text{and}\quad \lim_{s\to1^-}(-\Delta_p)^su(x)=-\Delta_pu(x),
$$
hold for smooth compactly supported functions, see Section \ref{constant}. The goal of this work is to find a suitable operator $L_{\Delta_p}$ to improve the understanding at the limit $s\to 0^+$. To be precise, we define an operator $L_{\Delta_p}$, which we call the \textit{logarithmic $p$-Laplace} operator, such that the expansion
\begin{equation}\label{simple expansion}
(-\Delta_p)^su(x)=|u(x)|^{p-2}u(x)+sL_{\Delta_p}u(x)+o(s)\quad\text{for $s\to 0^+$},
\end{equation}
holds for a suitable class of functions $u$.

In the linear case that is $p=2$, the study of the expansion \eqref{simple expansion} has been started in \cite{CW19}, where the authors show that it holds
$$
L_{\Delta_2}u(x)=C_N\int_{B_1(x)}\frac{u(x)-u(y)}{|x-y|^N}\,dy-C_N\int_{\R^N\setminus B_1(x)}\frac{u(y)}{|x-y|^{N}}\,dy+\rho_Nu(x),
$$
where $C_N=\frac{\Gamma(N/2)}{\pi^{N/2}}$ and $\rho_N=2\ln 2-\gamma+\psi(N/2)$. Here and below, $\psi=\Gamma'/\Gamma$ denotes the digamma function, $\Gamma$ is the gamma function, and $\gamma=-\Gamma'(1)$ is the Euler-Mascheroni constant. Moreover, it is shown in \cite{CW19} that $L_{\Delta_2}$ indeed has the Fourier symbol $2\ln|\cdot|$.

\noindent Recent studies have focused on the behavior of small order limits $s\to0^+$
and their connection to $L_{\Delta_2}$ within the context of $s$-dependent nonlinear Dirichlet problems, see \cite{AnSa, SanSa}. The investigation of fractional problems in this regime is especially relevant to optimization problems where the optimal order $s$ is small. Such small order limits are particularly important in applications like image processing and population dynamics, as discussed in references \cite{AnBa, PeVe, SpVa}.
Having the expansion at zero, there have been several works on the study of the expansion of eigenvalues, eigenfunctions, and certain solutions of the problem involving $L_{\Delta_2}$, see \cite{ChVe1, ChVe2, LaWe, JSW20,HS21,FJW22}. In the spirit of the Caﬀarelli--Silvestre extension problem for the fractional Laplacian, a characterization of the logarithmic Laplacian through a local extension problem is addressed in \cite{ChHaWe}. This logarithmic Laplacian operator appears naturally in the expansion at $s=1$, see \cite{JSW23}, and also arises in the geometric context of the 0-fractional perimeter, see \cite{LuNoPo}.

Inspired by the aforementioned works, we aim to generalize these results to the nonlinear case. In particular, we give an explicit representation of the operator $L_{\Delta_p}$ in \eqref{simple expansion}, which, unlike its linear counterpart, is both nonlinear and of logarithmic order. This representation is important for addressing problems involving $L_{\Delta_p}$ where the standard techniques are insufficient due to the nonlinear nature of the operator. In addition, the combination of nonlinearity and the weak singularity of the kernel in the representation of $L_{\Delta_p}$ introduces several challenges and we develop new techniques.
\subsection{Main results}
Our first main result deals with the expansion in \eqref{simple expansion}.
\begin{thm}\label{derivative}
Let $0<s<1$ and $1<p<\infty$. Suppose $u\in C_c^{\alpha}(\rz^{N})$ for some $\alpha>0$. Then, for $x\in\rz^N$ 
\begin{equation}\label{integral representation}
    \begin{split}
        L_{\Delta_p}u(x)
        &:=\frac{d}{ds}\Big|_{s=0}(-\Delta_p)^s\,u(x)\\
        &=C_{N,p}\int_{B_1(x)}\frac{|u(x)-u(y)|^{p-2}(u(x)-u(y))}{|x-y|^N}dy\\
        &\quad +C_{N,p}\int_{\rz^N\setminus B_1(x)}\frac{|u(x)-u(y)|^{p-2}(u(x)-u(y))-|u(x)|^{p-2}u(x)}{|x-y|^N}dy+\rho_N\,|u(x)|^{p-2}u(x),
    \end{split}
\end{equation}
where
$$
C_{N,p}:=\frac{p\,\Gamma\left(\frac{N}{2}\right)}{2\,\pi^{\frac{N}{2}}}\quad\text{and}\quad \rho_N:=\rho_N(p):=2\ln(2)-\gamma+\frac{p}{2}\psi\left(\frac{N}{2}\right).
$$ 
Moreover, for any $1<q\leq\infty$, we have $L_{\Delta_p}u\in L^q(\rz^N)\cap C(\rz^N)$ and
\begin{equation*}
    \frac{(-\Delta_p)^s\,u-|u|^{p-2}u}{s}\xrightarrow{s\to0^+} L_{\Delta_p}u\,\,\text{ in }L^q(\rz^N).
\end{equation*}
\end{thm}

In contrast
to the linear case, the \textit{convolution type} integral in $\R^N\setminus B_1(x)$ in the representation \eqref{integral representation} of $L_{\Delta_p}$ cannot be studied simply with usual convolution inequalities, since on the one hand, $z\mapsto |z|^{-N}$ is non-integrable at infinity and on the other hand, due to the appearing nonlinearity in the numerator the term does not immediately compensate the singular behavior at infinity. Similar to the case $p=2$, we also have a more localized representation to $L_{\Delta_p}$, see Lemma \ref{some properties}(3) and Lemma \ref{alternative2} below. Given $\Omega\subset \R^N$ open and $u\in C^{\alpha}_c(\R^N)$, $0<\alpha<1$, it also holds
\begin{align*}
        L_{\Delta_p}u(x)&=C_{N,p} \int_{\Omega}\frac{|u(x)-u(y)|^{p-2}(u(x)-u(y))}{|x-y|^{N}}dy+\Big(\rho_N(p)+h_{\Omega}(x)\Big)|u(x)|^{p-2}u(x)\\
        &\qquad\qquad +C_{N,p}\int_{\R^N\setminus \Omega}\frac{|u(x)-u(y)|^{p-2}(u(x)-u(y))-|u(x)|^{p-2}u(x)}{|x-y|^N}\,dy,
    \end{align*}
where
\begin{equation}\label{h-omega function}
    h_{\Omega}(x):=C_{N,p}\int_{B_1(x)\setminus \Omega}|x-y|^{-N}\,dy-C_{N,p}\int_{\Omega\setminus B_1(x)}|x-y|^{-N}\,dy.
 \end{equation}
Note that the function $h_{\Omega}$ (up to the multiplicative constant $\frac{p}{2}$) coincides with the function introduced in \cite[Corollary 1.9]{CW19}, and several bounds and various properties of this function can be found in \cite[Section 4]{CW19} and \cite{JSW20,JSW23}. 
\smallskip

Our next result deals with the expansion at $s=0$ of the first eigenvalue of the fractional $p$-Laplacian. Recall the fractional Sobolev space for $\Omega\subset \R^N$ open and bounded set 
$$
\cW^{s,p}_0(\Omega)=\{u\in W^{s,p}(\R^N)\;:\; u1_{\R^N\setminus \Omega}\equiv 0\},
$$
and the first (Dirichlet) eigenvalue of $(-\Delta_p)^s$ in $\Omega$ given by
\begin{equation}\label{intro:lambda-s-p}
\lambda_{s,p}^1(\Omega)=\inf_{\substack{u\in\cW^{s,p}_0(\Omega)\\\|u\|_{L^p(\Omega)}=1}}\frac{C_{N,s,p}}{2}\iint_{\R^N\times\R^N}\frac{|u(x)-u(y)|^{p}}{|x-y|^{N+sp}}\,dxdy.
\end{equation}
It is well known that the first eigenvalue $\lambda_{s,p}^1(\Omega)$ is positive and that there is an associated minimizer $\phi_{s}$, which is unique up to sign and can be chosen to be positive in $\Omega$. Similarly, we can and do set up a weak framework of $L_{\Delta_p}$. To this end, let
\begin{equation}\label{eq:first ev of frac p}
X^p_0(\Omega):=\left\{u\in L^p(\R^N)\;:\; u1_{\R^N\setminus \Omega}\equiv 0\quad\text{and}\quad \int_{\R^N}\int_{B_1(x)}\frac{|u(x)-u(y)|^p}{|x-y|^N}\,dydx<\infty\right\},
\end{equation}
which we discuss in more detail in Section \ref{variational}. Such spaces have been recently investigated in \cite{F23} and the kernel $z\mapsto 1_{B_1}(z)|z|^{-N}$ can be seen as the kernel of a $p$-L\'evy operator as introduced in \cite{F23}. Thus, the operator $L_{\Delta_p}$ is a $p$-L\'evy operator perturbed by two lower order terms. In \cite{F23}, several important statements for the analysis of solutions such as compact embeddings into $L^p(\Omega)$ and a Poincar\'e inequality have been shown, which hold in particular for $X^p_0(\Omega)$. We recall the statements adjusted to our setting in Section \ref{variational}.

The fractional Hardy inequality\cite{Dyda, DyKi, DyVa} states that for any bounded Lipschitz domain $\Omega\subset\rn$ and $0<s<1,\,p>0$ with $sp\neq 1$ there exists a constant $C=C(N,s,p,\Omega)>0$ such that 
\begin{equation}\label{Dyda-Hardy inq}
\int_{\Omega}\frac{|u(x)|^p}{\delta_x^{sp}}dx\leq C\iint_{\Omega\times\Omega}\frac{|u(x)-u(y)|^p}{|x-y|^{N+sp}}\,dx\,dy +C \int_\Omega |u(x)|^p\,dx \quad\text{ for all } u\in C_c^\infty(\Omega),
\end{equation}
and this inequality fails to hold in the so called \emph{critical case} $sp=1$. Recently, in \cite{AdJaRo, AdRoSa, AdRoSa2}, the authors addressed the critical case of the Hardy inequality, by inserting a logarithmic weight in the denominator on the left side of \eqref{Dyda-Hardy inq}.

We are interested in an inequality \eqref{Dyda-Hardy inq} for $s=0$ and with a~logarithmic weight in the numerator of the left hand side.
In our next result, we show the validity of this inequality for any $1\leq p<\infty$. However, in Section \ref{sec:hardy}, we prove a much more general result (Theorem \ref{hardy}). This inequality plays an important role in studying the solution space of the logarithmic $p$-Laplacian.

\begin{thm}[Logarithmic boundary Hardy inequality]\label{intro:hardy}
  Let $\Omega\subset \R^N$ be an open bounded Lipschitz set and $1\leq p<\infty$. Then there is $c>0$, depending on $\Omega$, $N$, and $p$, such that for every $u \in L^p(\Omega)$
  \[
  \int_\Omega |u(x)|^p \ln^+\left(\frac{1}{\delta_x}\right) \,dx
  \leq c \left(\  \iint_{\substack{\Omega\times \Omega\\ |x-y|<1}}\frac{|u(x)-u(y)|^p}{|x-y|^N}\,dy\,dx
  + \int_\Omega |u(x)|^p\,dx \right),
  \]
  where $\ln^+(t)=\max\{\ln(t),0\}$ for $t>0$.
\end{thm}

As mentioned above, Theorem \ref{intro:hardy} is a direct consequence of Corollary \ref{cor:hardyplump} below, which even holds for any $p>0$. Moreover, it follows that we can identify $X^p_0(\Omega)$ with those $L^p(\Omega)$-functions which satisfy
$$
\iint_{\substack{\Omega\times \Omega\\ |x-y|<1}}\frac{|u(x)-u(y)|^p}{|x-y|^N}\,dy\,dx<\infty,
$$
extended by zero in $\R^N\setminus \Omega$. Such a characterization is known\footnote{Note that the spaces $W^{s,p}(\Omega)$ and $W^{s,p}_0(\Omega)$ coincide for $sp\leq 1$.} for $W^{s,p}(\Omega)$ and $\cW^{s,p}_0(\Omega)$ if $sp<1$, and also for the corresponding spaces of logarithmic order in the case $p=2$, see \cite{CW19}. 

We note that Theorem~\ref{intro:hardy} was obtained for bounded Lipschitz set $\Omega\subset\rn$ and $p=2$, see \cite[Proposition A.1]{CW19}. Indeed, in this case such an inequality was proven with Fourier methods, which cannot be extended to the case $p\neq 2$. 
Additionally, for $\Omega=\R^N\setminus\{0\}$, the inequality of Theorem~\ref{intro:hardy} was proved
in \cite[Lemma~2.2]{LWXY}.
Very recently, in \cite[Theorem~3.4]{GL24}, a~version of the Hardy inequality in $\R^N$ with kernels more general than logarithmic was obtained. 

\smallskip

In the case of the fractional $p$-Laplacian, for an open bounded set $\Omega\subset \R^N$, there is a unique (up to sign and normalization) first eigenfunction $u_1\in X^p_0(\Omega)$ of $L_{\Delta_p}$ in $\Omega$ (see Sections \ref{variational} and \ref{eigenfunction}) corresponding to the first eigenvalue
\begin{align*}
    \lambda_{L,p}^1(\Omega)&:=\inf_{\substack{u\in X^p_0(\Omega)\\ \|u\|_{L^p(\Omega)}=1}}\Bigg(\  \frac{C_{N,p}}{2}\int_{\R^N}\int_{B_1(x)}\frac{|u(x)-u(y)|^p}{|x-y|^N}\,dydx+\rho_N(p)\\
    &\qquad\qquad \qquad \qquad +C_{N,p}\int_{\R^N}\int_{\R^N\setminus B_1(x)}\frac{|u(x)-u(y)|^p-|u(x)|^p}{|x-y|^N}\,dydx\Bigg)\\
    &=\inf_{\substack{u\in X^p_0(\Omega)\\ \|u\|_{L^p(\Omega)}=1}} \Bigg(\ \frac{C_{N,p}}{2}\iint_{\Omega\times \Omega}\frac{|u(x)-u(y)|^p}{|x-y|^N}\,dxdy+\int_{\Omega}(h_{\Omega}(x)+\rho_{N}(p))|u(x)|^p\,dx\Bigg).
\end{align*}
Contrary to the fractional $p$-Laplacian case, however, $\lambda_{L,p}^1(\Omega)$ is in general not positive. This follows from the logarithmic scaling behavior, see Proposition \ref{prop of lambda1}(ii) below, 
\begin{equation}\label{eigen scaling}
\lambda_{L,p}^1(r\Omega)=\lambda_{L,p}^1(\Omega)-p\ln(r)\quad\text{for $r>0$, $1<p<\infty$, and $\Omega\subset \R^N$ open.}
\end{equation}
We state our next result which can be seen as the nonlinear version of  \cite[Theorem 1.5]{CW19}.
\begin{thm}\label{relation between evs of log p Lap and frac p Lap}
    Let $\Omega$ be an open bounded Lipschitz subset of $\rn$ and $p\in(1,\infty)$. Then 
    \begin{equation*}
        \lambda^1_{L,p}(\Omega)=\frac{d}{ds}\bigg|_{s=0} \lambda^1_{s,p}(\Omega).
    \end{equation*}
Moreover, if we define $\phi_s$ as the $L^p$-normalized unique positive extremal for $\lambda^1_{s,p}(\Omega)$, then we have, as $s\to0^+$
\begin{equation}\label{eigenfunction conv}
    \phi_s\to u_1 \text{ in }L^p(\Omega),
\end{equation}
where $u_1$ is the $L^p$-normalized unique positive extremal for $\lambda^1_{L,p}(\Omega)$. 
\end{thm}

As a consequence of Theorem \ref{relation between evs of log p Lap and frac p Lap}, we deduce the Faber-Krahn inequality for the logarithmic $p$-Laplacian operator $L_{\Delta_p}$.

\begin{cor}[Faber--Krahn inequality for $L_{\Delta_p}$]\label{faber krahn}
Let $\Omega\subset\rn$ be an open bounded Lipschitz set with $|\Omega|=m\in(0,\infty)$, $p\in(1,\infty)$ and let $B^{(m)}\subset\rn$ be any ball with volume $m$. Then
$$
\lambda^1_{L,p}(B^{(m)})\leq\lambda^1_{L,p}(\Omega).
$$
\end{cor}
\begin{remark}
It remains an intriguing open question whether the inequality in Corollary \ref{faber krahn} is strict when $\Omega$ is different from a ball. This problem persists as unsolved, even in the linear case $p=2$, see \cite{CW19}.
\end{remark}
Noteworthy in Theorem \ref{relation between evs of log p Lap and frac p Lap} is the positivity of the extremal $u_1$. Since $\lambda_{L,p}^1(\Omega)$ may be negative for large $\Omega$ due to \eqref{eigen scaling}, the validity of a maximum principle is not clear and indeed will be false. It is worth mentioning that maximum principles and, more interestingly in the nonlinear setting, comparison principles are in general quite delicate, see the discussion in Remark \ref{discussion on scp}. To state our main results on the maximum and the comparison principles, we need to introduce some further notation. For $u,v\in X^p_0(\Omega)$, let
\begin{align*}
\cE_{L,p}(u,v)&=\frac{C_{N,p}}{2}\int_{\R^N}\int_{B_1(x)}\frac{|u(x)-u(y)|^{p-2}(u(x)-u(y))(v(x)-v(y))}{|x-y|^N}\,dydx\\
&\qquad +\int_{\R^N}\rho_{N}(p)|u(x)|^{p-2}u(x)v(x)\,dx\\
&\qquad+ \frac{C_{N,p}}{2}\int_{\R^N}\int_{\R^N\setminus B_1(x)}\frac{1}{|x-y|^N}\Bigg(|u(x)-u(y)|^{p-2}(u(x)-u(y))(v(x)-v(y))\\
&\qquad\qquad\qquad\qquad\qquad\qquad\qquad\qquad\qquad -|u(x)|^{p-2}u(x)v(x)-|u(y)|^{p-2}u(y)v(y)\Bigg)\,dydx.
\end{align*}

\begin{defi}
For $\Omega\subset \R^N$ open bounded set and $f\in L^{\frac{p}{p-1}}(\Omega)$, we say that a measurable function $u:\R^N\to\R$ satisfies weakly
$$
L_{\Delta_p}u\geq f\quad\text{in $\Omega$,}\qquad u=0\quad\text{in $\R^N\setminus \Omega$,}
$$
if $u\in X^p_0(\Omega)$ and for all nonnegative $\phi\in X^p_0(\Omega)$ it holds
$$
\cE_{L,p}(u,\phi)\geq \int_{\Omega}f\phi\,dx.
$$
Similarly, we define $L_{\Delta_p}u= f$ or $L_{\Delta_p}u\leq f$ in $\Omega$.
\end{defi}

We emphasize that we generalize this definition of \textit{supersolution} in Section \ref{dirichlet problem} below. Now, we are in position to state our results concerning strong maximum and comparison principles involving the logarithmic $p$-Laplacian.

\begin{thm}[Strong maximum principle]\label{intro:maximum principle}
Let $\Omega\subset \R^N$ be an open bounded set. If and only if $\lambda_{L,p}^1(\Omega)>0$ holds, the following is true: For any $u\in X^p_0(\Omega)$ satisfying weakly $L_{\Delta_p}u\geq 0$ in $\Omega$, $u=0$ in $\R^N\setminus \Omega$, it follows that either $u\equiv 0$ in $\Omega$ or  $u> 0$ in $\Omega$ in the sense that
$$
\essinf_{K}u>0\quad\text{for all compact sets $K\subset \Omega$}.
$$
\end{thm}

\begin{thm}[Strong comparison principle]\label{intro:comparison principle}
Let $\Omega\subset \R^N$ be an open bounded set, $c\in L^{\infty}(\Omega)$ with
$$
c(x)\leq \rho_{N}(p)+h_{\Omega}(x)\quad\text{for a.e. $x\in \Omega$}.
$$
Suppose $u,v\in X^p_0(\Omega)$ are such that either $u\in L^{\infty}(\Omega)$ or $v\in L^{\infty}(\Omega)$ and it holds weakly
$$
L_{\Delta_p}u-c(x)|u|^{p-2}u\geq L_{\Delta_p}v-c(x)|v|^{p-2}v\quad\text{in $\Omega$},\qquad u=0=v\quad\text{in $\R^N\setminus \Omega$}.
$$
Then, either $u\equiv v$ in $\R^N$ or  $u> v$ in $\Omega$ in the sense that
$$
\essinf_{K}(u-v)>0\quad\text{for all compact sets $K\subset \Omega$}.
$$
\end{thm}
We emphasize that both above Theorems on the strong maximum and strong comparison principles are special cases for more general types of solutions, which we introduce in Section \ref{dirichlet problem}. Let us also mention that the validity of a comparison principle is usually linked to the first eigenvalue, though in the case $p\neq 2$ this is not trivial. Here, this can be seen through Lemma \ref{how to ensure lambda1 positive}, which states that we have
$$
\lambda_{L,p}^1(\Omega)>0\quad\text{if}\quad \rho_{N}(p)+h_{\Omega}(x)\geq0.
$$

As a final main result, we show the boundedness of solutions to certain equations, which include inhomogeneous problems and eigenvalue type-problems involving $L_{\Delta_p}$.

\begin{thm}\label{bounded2}
Let $\Omega\subset \R^N$ be an open and bounded set, $f,c\in L^{\infty}(\Omega)$, and assume $u\in X^p_0(\Omega)$ satisfies weakly
$$
L_{\Delta_p}u=c(x)|u|^{p-2}u+f\quad\text{in $\Omega$},\qquad u=0 \quad\text{in $\R^N\setminus \Omega$}.
$$
Then $u\in L^{\infty}(\R^N)$.
\end{thm}

\subsection{Plan of the paper}
We begin to collect some preliminaries in Section \ref{preliminaries}, which include a priori estimates, the definitions of classical function spaces we need, and simple limiting behaviors of $(-\Delta_p)^s$ and its constant. In Section \ref{proof derivative} we give the Proof of Theorem \ref{derivative} and present pointwise properties of $L_{\Delta_p}$. Section \ref{variational} deals with the weak formulation of the logarithmic $p$-Laplace and the properties of the space $X^p_0(\Omega)$. The proof of the logarithmic boundary Hardy inequality is done in Section \ref{sec:hardy}. In Section \ref{dirichlet problem} we formulate a framework for weak supersolutions, which do not vanish outside of $\Omega$ and we give the proofs of Theorem \ref{bounded2}, Theorem \ref{intro:maximum principle}, and Theorem \ref{intro:comparison principle} alongside more general statements. Finally, in Section \ref{eigenfunction}, we give the proof of Theorem \ref{relation between evs of log p Lap and frac p Lap} and Corollary \ref{faber krahn}. We emphasize, moreover, the properties on $\lambda_{L,p}^1(\Omega)$ and $h_{\Omega}$ listed in Subsection \ref{sec:maximum principle} and Section \ref{eigenfunction}, which in particular lead to small volume type maximum principles such as Corollary \ref{small volume}. 

\section{Preliminaries and known results}\label{preliminaries}
\subsection{Notation} We use the following notation. For $U\subset \R^N$, $r>0$, let $B_r(U):=\{x\in \R^N\;:\; \dist(x,U)<r\}$, where $\dist(\cdot,U)$ denotes the distance of $x$ to $U$. If $U=\{x\}$ for some $x\in \R^N$, we also write $B_r(x)$ in place of $B_r(\{x\})$ to denote the ball of radius $r$ centered at $x$. Moreover, we put $B_r:=B_r(0)$. Throughout, we set $U^c:=\R^N\setminus U$. If $U$ is measurable, $|U|$ denotes the $N$-dimensional Lebesgue measure of $U$ and we put
$$
\omega_N:=\frac{2\pi^{\frac{N}{2}}}{\Gamma(\frac{N}{2})}
$$
for the $(N-1)$-dimensional volume of $\partial B_1$. $\Omega$ denotes throughout this work an open nonempty subset of $\R^N$, which may have further properties as stated. We let $\delta_x:=\delta(x):=\dist(x,\partial \Omega)$.\\
For a function $u:\R^N\to\R$, we denote $u^+=\max\{u,0\}$ for the positive part and $u^-:=\max\{-u,0\}$ for the negative part of $u$ so that $u=u^+-u^-$.\\
Finally, for $p> 1$, we set $g(a):=g_p(a):=|a|^{p-2}a$ for $a\in \R$. 

\subsection{Function spaces}\label{function spaces}

We use several different definitions of functions spaces ---classical and new. For the readers convenience we give here a list of the known spaces with their respective short definitions. We remark that our definitions might vary slightly, since we always consider functions to be defined on the whole~$\R^N$.

\noindent Let $U\subset \R^N$ be open. For $\alpha=k+\sigma>0$ with $k\in \N_0$ and $\sigma\in(0,1)$ let $C^{\alpha}(U)$ denote the space of functions $u:\R^N\to\R$, which in $U$ are $k$-times continuously differentiable and the derivatives up to order $k$ are $\sigma$-H\"older continuous in $U$. Moreover, we set $C^{k+1}(U):=C^{k,1}(U)$ as the space of functions $u:\R^N\to\R$, which in $U$ are $k$-times continuously differentiable and the derivatives up to order $k$ are Lipschitz continuous in $U$. Here, for $\sigma\in(0,1]$ and an arbitrary nonempty set $K\subset \R^N$ a function $u:\R^N\to\R$ is called $\sigma$-H\"older continuous (resp. Lipschitz continuous if $\sigma=1$) in $K$, if
$$
\sup_{x,y\in K}\frac{|u(x)-u(y)|}{|x-y|^{\sigma}}<\infty.
$$
As usual $C^{\infty}(U)=\bigcap_{\alpha>0} C^{\alpha}(U)$. We set, for an arbitrary $\alpha\in(0,\infty]$,
\begin{align*}
   C^{\alpha}(\overline{U})&:=\Big\{u\in C^{\alpha}(U)\;:\; \text{all derivatives of $u$ up to order $\lfloor \alpha\rfloor$ }\\
   &\qquad\qquad\qquad\qquad\qquad\qquad  \qquad\qquad  \text{ have a continuous extension to $\overline{U}$}\Big\},\\
   C^{\alpha}_{c}(U)&:=\Big\{u\in C^{\alpha}(U)\;:\; \text{$\supp\,u$ is a compact subset of $U$}\Big\},\\
   C^{\alpha}_{loc}(U)&:=\Big\{u:\R^N\to\R\;:\; u|_{K}\in C^{\alpha}(K)\ \text{for all nonempty compact subsets $K\subset U$}\Big\}.
\end{align*}
We use the above notation also for the space of continuous function $C(U)$.
Given $q>0$ and a measurable function $u:\R^N\to\R$, we let 
$$
\|u\|_{L^q(U)}:=\left(\,\int_{U}|u(x)|^q\,dx\right)^{\frac{1}{q}}
$$
and
\begin{align*}
L^q(U)&:=\Big\{u:\R^N\to \R\;:\; u1_{\R^N\setminus U}\equiv 0\ \text{and}\ \|u\|_{L^q(U)}<\infty\Big\},\\
L_{loc}^q(U)&:=\Big\{u:\R^N\to\R\;:\; 1_Au\in L^q(A)\ \text{for all measurable sets $A\subset U$ with $\overline{A}\subset U$}\Big\}.
\end{align*}
We also use analogous definitions for $L^{\infty}(U)$ and $L^{\infty}_{loc}(U)$. Moreover, for $t\in \R$, we let $L^q_t$ be  the space of functions $u\in L^q_{loc}(\R^N)$ such that
\begin{equation}\label{tail space}
\|u\|_{L^q_t}:=\Bigg(\,\,\int_{\R^N}\frac{|u|^{q}}{(1+|x|)^{N+t(q+1)}}\ dx\Bigg)^{\frac{1}{q}}<\infty.
\end{equation}
Note that $\|\cdot\|_{L^q(U)}$ and $\|\cdot\|_{L^q_t}$ are norms  only if $q\geq 1$, but the extension to $q\in(0,1)$ is convenient. 
Given $s\in(0,1)$ and $p\in[1,\infty)$, we let
$$
W^{s,p}(U)=\left\{u\in L^p(U)\;:\; \iint_{U\times U}\frac{|u(x)-u(y)|^p}{|x-y|^{N+sp}}\,dxdy<\infty\right\}
$$
to denote the usual fractional Sobolev space, see e.g. \cite{NPV12} for an introduction to such spaces. $W^{s,p}(U)$ is a Banach space with the norm
$$
\|u\|_{s,p,U}:=\Big(\|u\|_{L^p(U)}^p+[u]_{W^{s,p}(U)}^p\Big)^{\frac{1}{p}},
$$
where
$$ 
[u]_{W^{s,p}(U)}^p :=\frac{C_{N,s,p}}{2}\iint\limits_{U\times U}\frac{|u(x)-u(y)|^p}{|x-y|^{N+sp}}\,dxdy
$$
is called the Gagliardo seminorm, and $C_{N,s,p}$ denotes the \textit{normalization constant} for the fractional $p$-Laplace, see Section \ref{constant} below. Moreover, we let
$$
\cW^{s,p}_0(U):=\Big\{u\in W^{s,p}(\R^N)\;:\; u1_{\R^N\setminus U}\equiv 0\Big\},
$$
which is also a Banach space with the norm $\|\cdot\|_{s,p,\R^N}$, and 
$$
W^{s,p}_0(U):=\overline{C^{\infty}_c(U)}^{\|\cdot\|_{s,p,U}}.
$$
Finally, the main spaces we study in this work are $X^p_0(\Omega)$ as defined in \eqref{eq:first ev of frac p}, see Section \ref{variational}, and the space $V(\Omega,\R^N)$ in Section \ref{dirichlet problem} for supersolutions.

\subsection{Some useful inequalities }
We recall some elementary inequalities that are useful in proving our results.

\begin{lemma}[Lemmas 2 and 3, \cite{L16}]\label{estimate-general-1}
For all $a,b\in \R$, the following estimates hold:\\
If $p\in(1,2]$, then
$$
|g(a+b)-g(a)|\leq (3^{p-1}+2^{p-1})|b|^{p-1},
$$
and if $p\geq 2$, then
$$
|g(a+b)-g(a)|\leq (p-1)|b|\big(|a|+|b|\big)^{p-2}.
$$
\end{lemma}

\begin{lemma}[Section 2.2, \cite{IMS16}]\label{estimate-general-2}
Let $b>0$. Then
$$
    g(a+b)\leq \max\{1,2^{p-2}\}(a^{p-1}+b^{p-1})\quad\text{for all $a\geq 0$.}
$$
If, in addition, $p\geq 2$, then
$$
    g(a+b)-g(a)\geq 2^{2-p}b^{p-1}\quad \text{for all $a\in \R$.}
$$
\end{lemma}

\begin{lemma}[Section 2.1, \cite{J18}]\label{estimate-general-3}
Let $M>0$ and $p>1$. Then there is $C_1,C_2>0$ such that for all $a\in[-M,M]$, $b\geq 0$:
\begin{align*}
 g(a)-g(a-b)&\leq C_1\max\{b,b^{p-1}\},\\
    g(a+b)-g(a)&\geq C_2\min\{b,b^{p-1}\}.
\end{align*}
\end{lemma}

\subsection{On the normalization constant of the fractional \texorpdfstring{$p$}{p}-Laplacian}\label{constant}

Let $N\in \N$, $p>1$ and $s\in(0,1)$. Recall the definition of the fractional $p$-Laplacian $(-\Delta_p)^su$ for $u$ sufficiently regular in the introduction, where the normalization constant $C_{N,s,p}$ is given by
$$
C_{N,s,p}=\left\{\begin{aligned}&\frac{sp\,2^{2s-2}\,\Gamma\left(\frac{N+sp}{2}\right)}{\pi^{\frac{N-1}{2}}\Gamma(1-s)\Gamma\left(\frac{p+1}{2}\right)} && \text{ if } s>\frac{1}{2},\\
& \frac{sp\,2^{2s-1}\,\Gamma\left(\frac{N+sp}{2}\right)}{\pi^{\frac{N}{2}}\Gamma(1-s) } && \text{ if }s\leq \frac{1}{2}.
\end{aligned}\right.
$$
The following is well-known, however, we include its proof for the reader's convenience and completeness.

\begin{lemma}\label{limit s to 0}
Let $p> 1$. If $s\in\left(0,\frac{p-1}{p}\right)$, $\alpha\in\left(\frac{sp}{p-1},1\right]$, and $u\in C^{\alpha}_{c}(\R^N)$. Then $(-\Delta_p)^su(x)$ is well-defined for any $x\in \R^N$ and it holds
$$
\lim_{s\to0^+}(-\Delta_p)^su(x)=g(u(x))\quad\text{for all $x\in \R^N$.}
$$
\end{lemma}
\begin{proof}
Note that, for $r>0$
$$
\int_{B_r}\frac{g(|y|^\alpha)}{|y|^{N+sp}}\ dy=\int_{B_r}|y|^{-N+(\alpha-s)p-\alpha}\ dy=\frac{\omega_N}{(\alpha-s)p-\alpha}r^{(\alpha-s)p-\alpha}=\frac{\omega_N}{\alpha(p-1)-sp}r^{\alpha(p-1)-sp}<\infty,
$$
and
$$
\int_{B_r^c}\frac{1}{|y|^{N+sp}}\ dy=\frac{\omega_N}{sp}r^{-sp}<\infty.
$$
Hence, $(-\Delta_p)^su$ is well-defined. Moreover, after fixing $x\in \R^N$, we have, with $R>0$ such that $\overline{\supp\,u}\subset B_R(x)$,
\begin{align*}
(-\Delta_p)^su(x)&=C_{N,s,p}\int_{B_R(x)}\frac{g(u(x)-u(y))}{|x-y|^{N+sp}}\,dy+ C_{N,s,p}\int_{B_R^c(x)}\frac{g(u(x)-u(y))}{|x-y|^{N+sp}}\,dy\\
&=C_{N,s,p}\int_{B_R(x)}\frac{g(u(x)-u(y))}{|x-y|^{N+sp}}\,dy+ g(u(x))C_{N,s,p}\int_{B_R^c(x)}|x-y|^{-N-sp}\,dy,
\end{align*}
where
$$
\Bigg|\int_{B_R(x)}\frac{g(u(x)-u(y))}{|x-y|^{N+sp}}\,dy\Bigg|\leq g(\|u\|_{C^{\alpha}(\R^N)})\frac{\omega_N}{(\alpha-s)p-\alpha}R^{(\alpha-s)p-\alpha},
$$
and
$$
\int_{B_R^c(x)}|x-y|^{-N-sp}\,dy=\int_{B_R^c}|y|^{-N-sp}\,dy=\frac{\omega_N}{sp}R^{-sp}.
$$
Thus by definition of $C_{N,s,p}$, it follows that
$$
\lim_{s\to 0^+}(-\Delta_p)^su(x)=g(u(x))\quad\text{for $x\in \R^N$.}
$$
as claimed.
\end{proof}

\begin{remark}
Let us add some further remarks concerning $(-\Delta_p)^s$.
\begin{enumerate}
\item Note that $(-\Delta_p)^s$ is also pointwisely well-defined for any $s\in(0,1)$ if $u$ is sufficiently regular, see e.g. \cite{IMS16,F23}. To be precise, if $u\in C^{2+\alpha}_c(\R^N)$ for some $\alpha>0$, then $(-\Delta_p)^su(x)$ is well-defined for all $x\in \R^N$ and $s\in(0,1)$.
\item The choice of the constant $C_{N,p,s}$ in our setting is rather artificial as there is no Fourier transform to justify a normalization constant as in the case $p=2$ (see e.g. \cite{NPV12}). With our choice, the constant agrees with the case $p=2$, the limit $s\to 0^+$ gives a kind of identity, and it holds, see for instance the discussions in \cite{TGV21,F23},
$$
\lim_{s\to1^-}(-\Delta_p)^su(x)=-\Delta_pu(x)=-\div(|\nabla u(x)|^{p-2}\nabla u(x))\quad\text{for all $x\in \R^N$.}
$$
However, the choices of $C_{N,s,p}$ for $s\geq\epsilon$ for some $\epsilon>0$ is indeed not relevant to our analysis.\\
In particular, any other choice of $C_{N,s,p}$, such that $\partial_sC_{N,s,p}$ exists at $s=0$ can be used and only changes the zero order part of the logarithmic $p$-Laplacian.
\end{enumerate}
\end{remark}

\section{Derivation of the logarithmic \texorpdfstring{$p$}{p}-Laplacian and some properties}\label{proof derivative}

The goal of this section is to prove Theorem \ref{derivative} and, in addition, to give several properties of the integral representation of the operator. We begin with the introduction and study of the integral operator
\begin{equation*}\label{integral representation2}
\begin{split}
   \cL_{\Delta_p} u(x)&:= C_{N,p}\int_{B_1(x)}\frac{|u(x)-u(y)|^{p-2}(u(x)-u(y))}{|x-y|^N}dy\\
        &\quad +C_{N,p}\int_{\rz^N\setminus B_1(x)}\frac{|u(x)-u(y)|^{p-2}(u(x)-u(y))-|u(x)|^{p-2}u(x)}{|x-y|^N}dy+\rho_N\,|u(x)|^{p-2}u(x),
\end{split}
\end{equation*}
for suitable $u:\R^N\to\R$, $x\in \R^N$, and with the constants $C_{N,p}$, $\rho_N$ as in Theorem \ref{derivative}.

In the first lemma, we collect some basic properties of $\cL_{\Delta_p}$ and provide an alternative integral representation of $\cL_{\Delta_p}$, which involves the function $h_{\Omega}$ defined in \eqref{h-omega function}.

\begin{lemma}\label{some properties}
Let $1<p<\infty$ and $u\in C^{\alpha}_{c}(\R^N)$ for some $\alpha \in (0,1)$. Then $\cL_{\Delta_p}u$ is well-defined and continuous. Moreover, the following holds.
\begin{enumerate}
    \item $\cL_{\Delta_p}$ is translation and rotation invariant in the sense that
    $$
    \cL_{\Delta_p}v(x)=\cL_{\Delta_p}u(Ox+x_0),
    $$
    where $v(x)=u(Ox+x_0)$ with $x_0\in \R^N$ and a  rotation $O$.
    \item $\cL_{\Delta_p}$ satisfies the following scaling property for $r>0$:
    $$
    \cL_{\Delta_p}v(x)=\cL_{\Delta_p}u(rx)+C_{N,p}\omega_N\ln(r)u(rx),
    $$
    where $v(x)=u(rx)$.
    \item If $\Omega\subset \R^N$ is an open bounded set and $x\in\Omega$, then
    \begin{align}
        \cL_{\Delta_p}u(x)&=C_{N,p} \int_{\Omega}\frac{g(u(x)-u(y))}{|x-y|^{N}}dy+C_{N,p}\int_{\R^N\setminus \Omega}\frac{g(u(x)-u(y))-g(u(x))}{|x-y|^N}\,dy\nonumber\\
        &\qquad\qquad  +\Big(\rho_N(p)+h_{\Omega}(x)\Big)g(u(x)), \label{LwithOmega}
    \end{align}
    with $h_{\Omega}$ defined as in \eqref{h-omega function}, that is, 
    $$
    h_{\Omega}(x)=C_{N,p}\int_{B_1(x)\setminus \Omega}|x-y|^{-N}\,dy-C_{N,p}\int_{\Omega\setminus B_1(x)}|x-y|^{-N}\,dy.
    $$
\end{enumerate}
\end{lemma}
\begin{proof}
Let $x\in \R^N$ and fix $R>0$ such that $\supp\,u\subset B_R(x)$. Moreover, let $c>0$ such that
$$
|u(x)-u(y)|\leq c\min\{1,|x-y|^\alpha\}\quad \text{and}\quad |u(x)|\leq c\quad\text{for all $x,y\in \R^N$.} 
$$
Then
$$
\Bigg|\int_{B_1(x)}\frac{g(u(x)-u(y))}{|x-y|^N}\ dy\Bigg|\leq c^{p-1}\omega_N\int_0^1t^{-1+\alpha (p-1)}\ dt<\infty,
$$
and since $g(u(y)-u(x))-g(u(x))=0$ for $y\in B_R(x)^c$ 
\begin{align*}
\Bigg|\int_{B_1^c(x)}\frac{g(u(x)-u(y))-g(u(x))}{|x-y|^N}\ dy\Bigg|\leq c^{p-1}\int_{B_R(x)\setminus B_1(x)}\frac1{|x-y|^{N}}\, dy=c^{p-1}\omega_N\int_1^{R} t^{-1}\ dt<\infty.
\end{align*}
Hence, $\mathcal{L}_{\Delta_p}u$ is well-defined in $\R^N$. Next, let $f_1,f_2:\R^N\to\R$ be defined by
$$
f_1(x)=\int_{B_1(x)}\frac{g(u(x)-u(y))}{|x-y|^N}\,dy\quad\text{and}\quad f_2(x)=\int_{B_1^c(x)}\frac{g(u(x)-u(y))-g(u(x))}{|x-y|^N}\,dy+\rho_ng(u(x)).
$$
Then the continuity of $f_2$ follows analogously to the proof of $f_2$ being well-defined. To prove the continuity of $f_1$, we separate the cases $p\geq 2$ and $p\in(1,2)$.
\smallskip

\noindent\textbf{Case 1:} $p\geq 2$. Recall that, by Lemma \ref{estimate-general-1},
\begin{equation}\label{p greater 2 estimate}
\Big|g(a+b)-g(a)\Big|\leq (p-1)|b|(|a|+|b|)^{p-2}\quad \text{for all $a,b\in\R$.}
\end{equation}
Thus, for $x,z\in \R^N$ we have
\begin{align*}
|&g(u(x)-u(x+y))-g(u(z)-u(z+y))|\\
&\leq (p-1)|u(x)-u(z)+u(z+y)-u(x+y)|(|u(x)-u(z)+u(z+y)-u(x+y)|+|u(z)-u(z+y)|)^{p-2}\\
&\leq (6c)^{p-2}(p-1) 
\Big(|u(x)-u(x+y)|+|u(z+y)-u(z)|\Big)^{1/2} \Big(|u(x)-u(z)|+|u(z+y)-u(x+y)|\Big)^{1/2}\\
&\leq C |y|^{\alpha/2} |x-z|^{\alpha/2},
\end{align*}
with $C=2\cdot 6^{p-2} c^{p-1}(p-1)$.
Thus, for $x,z\in \R^N$ we have
\begin{align*}
\Big|f_1(x)-f_1(z)\Big|&\leq \int_{B_1}\frac{|g(u(x)-u(x+y))-g(u(z)-u(z+y))|}{|y|^N}\,dy\\
&\leq C|x-z|^{\alpha/2} \int_{B_1}|y|^{\alpha/2-N}\ dy
= \frac{2C\omega_N}{\alpha}|x-z|^{\alpha/2}.
\end{align*}
This gives that $f_1$ is continuous, and this implies the continuity of $\cL_{\Delta_p}$.
\smallskip

\noindent\textbf{Case 2:} $p\in(1,2)$. In this case it holds
Lemma \ref{estimate-general-1},
\begin{equation}\label{p less 2 estimate}
\Big|g(a+b)-g(a)\Big|\leq (3^{p-1}+2^{p-1})|b|^{p-1}\quad \text{for all $a,b\in\R$.}
\end{equation}
As in \textbf{Case 1}, we conclude that for $x,z\in \R^N$ we have 
\begin{align*}
|&g(u(x)-u(x+y))-g(u(z)-u(z+y))|\\
&\leq \min\{2c|y|^{(p-1)\alpha},\tilde{C}|x-z|^{(p-1)\alpha} \}
 \leq (2c\tilde{C})^{1/2} |y|^{(p-1)\alpha/2} |x-z|^{(p-1)\alpha/2},
\end{align*}
for some constant $\tilde{C}>0$ independent of $y$. The continuity of $\cL_{\Delta_p}$ now follows similarly to \textbf{Case 1}.\\
We continue using the notation $f_1=f_1(u)$ and $f_2=f_2(u)$. To see (1), let $Tx=Ox+x_0$ and note that
\begin{align*}
f_1(v)(x)&=\int_{B_1(x)}\frac{g(u(Tx)-u(Ty))}{|x-y|^N}\,dy=\int_{B_1(x)}\frac{g(u(Tx)-u(Ty))}{|Tx-Ty|^N}\,dy=\int_{T(B_1(x))}\frac{g(u(Tx)-u(y))}{|Tx-y|^N}\ dy\\
&=\int_{B_1(Tx))}\frac{g(u(Tx)-u(y))}{|Tx-y|^N}\ dy=f_1(u)(Tx),
\end{align*}
and similarly $f_2(v)(x)=f_2(u)(Tx)$.\\
To see (2), let first $r>1$. Then
\begin{align*}
\cL_{\Delta_p}v(x)&=C_{N,p}\int_{B_1}\frac{g(u(rx)-u(r(x+y)))}{|y|^N}\ dy+C_{N,p}\int_{B_1^c}\frac{g(u(rx)-u(rx+ry))-g(u(rx))}{|y|^N}\, dy+\rho_Ng(u(rx))\\
&=C_{N,p}\int_{B_r}\frac{g(u(rx)-u(rx+z)))}{|z|^N}\ dz+C_{N,p}\int_{B_r^c}\frac{g(u(rx)-u(rx+z))-g(u(rx))}{|z|^N}\, dz+\rho_Ng(u(rx))\\
&=C_{N,p}\int_{B_1}\frac{g(u(rx)-u(rx+z)))}{|z|^N}\ dz+C_{N,p}\int_{B_1^c}\frac{g(u(rx)-u(rx+z))-g(u(rx))}{|z|^N}\, dz\\
&\qquad +C_{N,p}g(u(rx))\int_{B_r\setminus B_1}|y|^{-N}\ dy+\rho_Ng(u(rx))\\
&=\cL_{\Delta_p}u(rx)+C_{N,p}\omega_N\ln(r)g(u(rx)).
\end{align*}
For $r<1$, we may take $\rho=1/r>1$ and $u(x)=v(\rho x)$, which allows us to use already proven
(2) with $\rho$ instead of $r$ and functions $u$ and $v$ interchanged.
This shows property (2).

To see the last statement, we let
\[
 F_1 = \frac{g(u(x)-u(y))}{|x-y|^{N}}, \qquad
 F_2 = \frac{g(u(x)-u(y))-g(u(x))}{|x-y|^N},
\]
and observe that the integrals in (3) are absolutely convergent. For the first one,
it follows from the fact that $|F_1| \leq c\min\{1, |x-y|^{\alpha(p-1)}\}$,
and for the second, from $|F_2|\leq 2c$ and $F_2=0$ if $y\not \in \supp\;u$.
These integrals are also absolutely convergent when $\Omega=B_1(x)$, because $B_1(x)$ satisfies all assumptions imposed on $\Omega$.
Therefore,
\begin{align*}
    \cL_{\Delta_p}u(x)&= C_{N,p}\int_{B_1(x)} F_1 \,dy +C_{N,p}\int_{B_1^c(x)} F_2 \,dy
    +\rho_N(p)g(u(x)) \\
    &=C_{N,p} \left( \int_\Omega - \int_{\Omega\setminus B_1(x)} + \int_{B_1(x) \setminus\Omega}
    \right) F_1 \,dy 
    +C_{N,p} \left(\int_{\Omega^c} - \int_{\Omega^c\setminus B_1^c(x)}+ \int_{B_1^c(x)\setminus \Omega^c}\right) F_2 \,dy + \rho_N(p)g(u(x)) \\
    &=C_{N,p}
        \left(\int_\Omega F_1 \,dy + \int_{\Omega^c} F_2 \,dy\right)+ \rho_N(p)g(u(x))
    +  C_{N,p}
    \left(\int_{\Omega\setminus B_1(x)}\!\!\!(F_2-F_1)\,dy +
    \int_{B_1(x)\setminus\Omega} \!\!\!(F_1-F_2)\,dy \right)\\
    &=C_{N,p}
    \left(\int_\Omega F_1 \,dy + \int_{\Omega^c} F_2 \,dy\right)+ 
    g(u(x))\left(\rho_N(p) - C_{N,p}\int_{\Omega\setminus B_1(x)} \frac{dy}{|x-y|^N}
    + C_{N,p}\int_{B_1(x) \setminus \Omega} \frac{dy}{|x-y|^N} \right),
\end{align*}
which proves (3).
\end{proof}

\begin{remark}\label{rem:Lomega}
In fact, in the proof of (3), we have used the fact that $u\in C_c^\alpha(\Omega)$ only
to show the convergence of the integrals. Therefore, the following alternative version
of (3) holds:
If $\Omega\subset \R^N$ is an open set, $x\in \Omega$ is such that $\cL_{\Delta_p}u(x)$
exists and is finite, and the integrals appearing in \eqref{LwithOmega} and in $h_{\Omega}$ are convergent,
then \eqref{LwithOmega} holds. For instance, this is also the case if $u$ is suitably Dini-continuous (see Lemma \ref{some properties 2} below) and $\Omega$ is a strip.
\end{remark}

In the following, we aim to extend the class of functions on which $\cL_{\Delta_p}$ can be applied. For this, we use the \textit{tail spaces} by $L^q_t$ with $q\in(0,\infty)$ and $t\in \R$ as introduced in \eqref{tail space} in Section \ref{function spaces}. By definition we have
\begin{equation*}\label{embed-lps simple}
\|u\|_{L^q_s}\leq \|u\|_{L^q_t}\quad\text{for all $u:\R^N\to\R$, $q\in(0,\infty)$, and $s\geq t$.}
\end{equation*}
\begin{lemma}\label{embed-lps}
Let $0\leq t<s$, $0< r\leq q<\infty$. Then, there is $c>0$ such that
$$
\|u\|_{L^r_s}\leq c \|u\|_{L^q_t}\quad\text{for all $u:\R^N\to\R$.}
$$
In particular, $L^q_t\subset L^r_s$.
\end{lemma}
\begin{proof}
This follows immediately from H\"older's inequality noting that since $t>s$ and with $r=\frac{q}{q-r}$, we have
\begin{align*}
\int_{\R^N}\frac{|u|^r}{(1+|x|)^{N+s(r+1)}}\ dx&\leq \left(\,\,\,\int_{\R^N}\frac{|u|^q}{(1+|x|)^{N+t(q+1)}}\ dx
\right)^{\frac{r}{q}} \left(\,\,\,\int_{\R^N}\frac{1}{(1+|x|)^{N+(s-t)r}}\ dx\right)^{\frac{1}{r}}\\
&\leq c\|u\|_{L^q_t}^{r},
\end{align*}
for a constant $c=c(N,s,t,r,q)>0$.
\end{proof}

We next show the following modification of \cite[Lemma 2.1]{CW19}.

\begin{lemma}\label{simple property}
Let $u\in L^{q}_0$ for some $0<q<\infty$ and let $v:\R^N\to\R$ such that there is $c>0$ with $|v(x)|\leq c(1+|x|)^{-N}$. Then, for any $x\in \R^N$ we have
$$
\lim_{y\to x}\int_{\R^N}|u(x+z)-u(y+x)|^{q}v(z)\,dz=0.
$$
\end{lemma}
\begin{proof}
First, note that by Lemma \ref{estimate-general-2} it holds
$$
|u(x+z)-u(y+x)|^{q}\leq \max\{1,2^{q-1}\}\big(|u(x+z)|^{q}+|u(y+z)|^q\big),
$$
and $1+|z|\leq 1+|z-x|+|x|\leq (1+|x|)(1+|z-x|)$. Therefore, by assumption, for any $x\in \R^N$
\begin{align*}
\int_{\R^N}|u(x+z)|^{q}v(z)\,dz&\leq c\int_{\R^N}\frac{|u(z)|^q}{(1+|{z-x}|)^{N}}\,dz
\leq
c(1+|x|)^{{N}} \int_{\R^N}\frac{|u(z)|^q}{(1+|z|)^N}\,dz < \infty.
\end{align*}
Thus 
$$
\int_{\R^N}|u(x+z)-u(y+x)|^{q}v(z)\,dz
$$
is finite, for any $x,y\in \R^N$. Moreover, the claim follows immediately if $u\in C_c(\R^N)$ by the dominated convergence theorem. Finally, the space $C_c(\R^N)$ is also dense in $L^q_0$ and thus, the statement holds by approximation.
\end{proof}

\begin{lemma}\label{included}
Let $0<q<p<\infty$ and $u\in L^q_0\cap L^p_0$. Then, $u\in L^t_0$\,\,\,for any $t\in[q,p]$.
\end{lemma}
\begin{proof}
Noting that, there is $\lambda\in(0,1)$ such that $t=(1-\lambda) p+\lambda q$, the statement follows in a standard way by H\"older's inequality with exponent $\frac{1}{\lambda}$ and its conjugate exponent $\frac{1}{1-\lambda}$.
\end{proof}

\noindent Next, let $\Omega\subset \R^N$ be a measurable set and let $u:\Omega\to\R$ be a~measurable function, the module of continuity of $u$ at a point $x\in \Omega$ is given by
$$
\omega_{u,x,\Omega}:(0,\infty)\to [0,\infty),\quad \omega_{u,x,\Omega}(r):=\sup_{y\in B_r(x)\cap\Omega}|u(x)-u(y)|.
$$
For $q>0$, a function $u:\Omega\to \R$ is called \textit{$q$-Dini-continuous at $x\in \Omega$}, if 
$\int_0^1 \omega_{u,x,\Omega}(r)^{q} r^{-1}\,dr<\infty$ and the function is called \textit{uniformly $q$-Dini-continuous in $\Omega$} if
$$
\int_0^1\frac{\omega_{u,\Omega}(r)^q}{r}\ dr<\infty,\quad\text{where}\quad \omega_{u,\Omega}(r)=\sup_{x\in \Omega}\omega_{u,x,\Omega}(r).
$$

\begin{remark}\label{remark dini}
If $u\in L^{\infty}_{loc}(\R^N)$ is (uniformly) $q$-Dini continuous for some $q>0$, then it is also (uniformly) $r$-Dini continuous for any $r<q$. 
\end{remark}

\begin{lemma}\label{some properties 2}
Let $1<p<\infty$, $\Omega\subset \R^N$ open, and let $u\in L^{p-1}_0$. Assume, additionally, $u\in L^1_0\cap L^{\infty}(\R^N)$ for $p>2$.
\begin{enumerate}
    \item If $u$ is $(p-1)$-Dini-continuous at $x\in \Omega$, then $\cL_{\Delta_p}[u](x)$ is well-defined. 
    \item If $u$ is uniformly $(p-1)$-Dini-continuous in $\Omega$, then $\cL_{\Delta_p}u$ is continuous in $\Omega$.
\end{enumerate}
Here, $u$ is any fixed representative with the property of being (uniformly) $(p-1)$-Dini-continuous at $x\in \Omega$ (in $\Omega$).
\end{lemma}
\begin{proof}
For (1), first note that
\begin{align*}
\int_{B_1(x)}&\frac{|g(u(x)-u(y))|}{|x-y|^N}\,dy\leq \omega_N\int_0^1\frac{\omega_{u,x,\Omega}^{p-1}(r)}{r}\,dr<\infty.
\end{align*}
Next, for $p\in(1,2]$, we have as in \eqref{p less 2 estimate}
$$
\Big|g(a+b)-g(a)\Big|\leq (3^{p-1}+2^{p-1})|b|^{p-1}\quad\text{for all $a,b\in\R$.}
$$
Thus
\begin{align*}
\int_{B_1(x)^c}&\frac{|g(u(x)-u(y))-g(u(x))|}{|x-y|^N}dy\leq (3^{p-1}+2^{p-1})\int_{B_1(x)^c}\frac{|u(y)|^{p-1}}{|x-y|^{N}}\,dy\leq c\int_{\rz^N}\frac{|u(y)|^{p-1}}{1+|y|^{N}}\,dy < \infty
\end{align*}
for some $c=c(N,p,x)>0$. For $p>2$, we have by \eqref{p greater 2 estimate}
$$
\Big|g(a+b)-g(a)\Big|\leq (p-1)|b|(|a|+|b|)^{p-2}\quad \text{for all $a,b\in\R$.}
$$
Thus
\begin{align*}
\int_{B_1(x)^c}&\frac{|g(u(x)-u(y))-g(u(x))|}{|x-y|^N}dy\leq (p-1)\int_{B_1(x)^c}\frac{|u(y)|\big(|u(x)|+|u(y)|\big)^{p-2}}{|x-y|^N}dy\\
&\leq 2^{p-1}(p-1)\Bigg(\ \int_{\substack{B_1(x)^c\\ \{|u(x)|\leq |u(y)|\}}} \frac{|u(y)|^{p-1}}{|x-y|^N}\,dy+\|u\|^{p-2}_{L^{\infty}(\R^N)}\int_{\substack{B_1(x)^c\\ \{|u(x)|> |u(y)|\}}} \frac{|u(y)|}{|x-y|^N}\, dy\Bigg) < \infty.
\end{align*}
From here, (1) follows. Note that, the last computation above remains true, if $B_1(x)^c$ is replaced by $B_{\epsilon}(x)^c$ for any $\epsilon>0$.\\
For (2), let us fix $x_0\in \Omega$ and $0 < \varepsilon < \min\{1, \delta(x_0)\}$.
From the proof of (1), it follows that for every $x$ such that $|x-x_0|<\varepsilon$, integrals 
\[
\int_{B_{1}(x)}\frac{|g(u(x)-u(y))|}{|x-y|^N}\,dy
\qquad\textrm{and}\qquad
\int_{B_\varepsilon(x)^c}\frac{|g(u(x)-u(y))-g(u(x))|}{|x-y|^N}dy
\]
are convergent.
This allows us to use Remark~\ref{rem:Lomega} to obtain
\[
\cL_{\Delta_p}u(x)=C_{N,p}f_1(x)+f_2(x)
\]
with
\begin{align*}
f_1,f_2:\R^N\to\R,\quad &f_1(x)=\int_{B_\epsilon(x)}\frac{g(u(x)-u(y))}{|x-y|^N}\,dy\quad\text{and}\\
&f_2(x)=C_{N,p}\int_{\R^N\setminus B_\epsilon(x)}\frac{g(u(x)-u(y))-g(u(x))}{|x-y|^N}\,dy+\big(\rho_{N}(p)+h_{B_{\varepsilon}(x)}(x)\big)g(u(x)).
\end{align*}
Noting that, $h_{B_{\varepsilon}(x)}(x)=-p\ln(\varepsilon)$ with Lemma \ref{simple property} it follows that $f_2$ is continuous at $x_0$.
Furthermore, for $x\in \Omega$ with $|x-x_0| < \epsilon$, we have
\begin{align*}
|f_1(x)-f_1(x_0)|&\leq\Bigg|\,\int_{B_\epsilon}\frac{g(u(x)-u(x+z))-g(u(x_0)-u(x_0+z))}{|z|^N}\,dz\Bigg|\\
&\leq \int_{B_{\epsilon}}\frac{|g(u(x)-u(x+z))|+|g(u(x_0)-u(x_0+z)|}{|z|^N}\,dz\\
&\leq 2\int_{B_{\epsilon}}\frac{\omega_{u,\Omega}^{p-1}(|z|)}{|z|^N}\,dz=2\omega_N\int_0^{\epsilon}\frac{\omega_{u,\Omega}^{p-1}(r)}{r}\,dr.
\end{align*}
Therefore,
\begin{align*}
\limsup_{x\to x_0} \left| \cL_{\Delta_p}u(x) - \cL_{\Delta_p}u(x_0) \right|
&\leq C_{N,p} \limsup_{x\to x_0} |f_1(x)-f_1(x_0)| 
+ \limsup_{x\to x_0} |f_2(x)-f_2(x_0)| \\
&\leq 2C_{N,p}\omega_N\int_0^{\epsilon}\frac{\omega_{u,\Omega}^{p-1}(r)}{r}\,dr.
\end{align*}
Since $\epsilon\in(0,\delta(x_0))$ is arbitrary and the latter right-hand side converges to $0$ for $\eps\to 0$, the claim of (2) follows.
\end{proof}

\begin{remark}
 If $u$ is a function as in Lemma \ref{some properties 2}(1) or (2), then also the statements of Lemma \ref{some properties}(1)--(3) hold. This follows immediately from the respective proofs. 
\end{remark}

We close this section by showing that $L_{\Delta_p}u=\cL_{\Delta_p}u$ for $u\in C^{\alpha}_c(\R^N)$.
\begin{proof}[Proof of Theorem \ref{derivative}]
Let $u\in C_c^{\alpha}(\rz^{N})$. As, we consider the limit $s\to 0^+$, we may assume 
$$
0<s<\begin{cases}
    \text{min}\left\{\frac{1}{p},\,\frac{\alpha}{p}\right\}\,\, \text{ if }p\geq2,\\
    \text{min}\left\{\frac{p-1}{p},\,\frac{\alpha(p-1)}{p}\right\}\,\, \text{ if }1<p<2.
\end{cases}
$$
Note that with this, we have $s<\frac{1}{2}$ and $s<\frac{\alpha(p-1)}{p}$. In particular, we are in the setting of Lemma \ref{limit s to 0}. Next, choose $r>4$ such that $\text{supp}\,u\subset B_{\frac{r}{4}}(0)$. Then for $x\in\rz^N$, we have
\begin{equation*}
\begin{split}
(-\Delta_p)^s\,u
&=C_{N,s,p}\int_{B_r(x)}\frac{g(u(x)-u(y))}{|x-y|^{N+sp}}dy+C_{N,s,p}\int_{B^c_r(x)}\frac{g(u(x)-u(y))}{|x-y|^{N+sp}}dy
\\
&=A_r(s,p,x)+D_r(s,p,x),
\end{split}
\end{equation*}
where
$$
A_r(s,p,x):=C_{N,s,p}\int_{B_r(x)}\frac{g(u(x)-u(y))}{|x-y|^{N+sp}}dy-C_{N,s,p}\int_{B^c_r(x)}\frac{|u(x)-u(y)|^{p-2}u(y)}{|x-y|^{N+sp}}dy
$$
and
\begin{equation*}\label{D r s p basis}
D_r(s,p,x):=C_{N,s,p}\int_{B_r^c(x)}\frac{|u(x)-u(y)|^{p-2}u(x)}{|x-y|^{N+sp}} dy.   
\end{equation*}
If $x\in B_{\frac{r}{2}}$, $y\in B_r^c(x)$, then $|y|\geq |x-y|-|x|>\frac{r}{2}$, so that $y\notin \supp\,u$. It thus follows that
$$
D_r(s,p,x)=\left\{\begin{aligned}
& D_r(s,p) g(u(x)) && \text{if $x\in B_{\frac{r}{2}}$,}\\
& 0 && \text{if $x\in B_{\frac{r}{2}}^c$,}
\end{aligned}
\right.
$$
where
\begin{equation}\label{D r s p}
D_r(s,p):=C_{N,s,p}\int_{B_r^c(x)}\frac{dy}{|x-y|^{N+sp}}=C_{N,s,p}\frac{\omega_N}{sp}\,r^{-sp}. 
\end{equation}
Now, we estimate $A_r(s,p,x)$ in the following. 

\smallskip

\noindent\textbf{\textit{Case 1:}} If $|x|\geq r/2$ then $u(x)=0$ and $|x-y|\geq\frac{|x|}{2}>1$ provided $y\in\text{supp u}$. Then, we obtain
\begin{equation}\label{Ar est}
    |A_r(s,p,x)|\leq C_{N,s,p}\int_{\rz^N}\frac{|u(y)|^{p-1}}{|x-y|^{N+sp}}dy
    \leq2^{N}\,C_{N,s,p}\|u\|^{p-1}_{L^{p-1}(\rz^N)}|x|^{-N}.
\end{equation}
Hence, for any $1<q<\infty$, from \eqref{Ar est} we obtain
\begin{equation}\label{Lq norm est of Ar}
\int_{B_\frac{r}{2}^c}|A_r(s,p,x)|^q\,dx\leq\frac{2^{2qN-N}\,C_{N,s,p}^q\,\omega_N}{N-qN}\|u\|^{q(p-1)}_{L^{p-1}(\rz^N)}r^{N-qN},
\end{equation}
and 
\begin{equation*}\label{Linfty norm est of Ar}
\|A_r(s,p,\cdot)\|_{L^{\infty}(B_\frac{r}{2}^c)}\leq4^{N}\,C_{N,s,p}\|u\|^{p-1}_{L^{p-1}(\rz^N)}r^{-N}.
\end{equation*}
Recall the definition of the constant $C_{N,s,p}$
\begin{equation*}
C_{N,s,p}=\frac{sp\,2^{2s-1}\,\Gamma\left(\frac{N+sp}{2}\right)}{\pi^{\frac{N}{2}}\Gamma(1-s)}=:s\,d_{N,p}(s).
\end{equation*}
This entails
$$
d_{N,p}(0)=\frac{p\,\Gamma\left(\frac{N}{2}\right)}{2\,\pi^{\frac{N}{2}}}=C_{N,p}\quad\text{and}\quad d^{\prime}_{N,p}(0)=C_{N,p}\,\rho_N.
$$
It follows from \eqref{Lq norm est of Ar}, for any $q\in(1,\infty)$ we have
\begin{equation}\label{final est of Ar}
    \|A_{r}(s,p,\cdot)\|_{L^{q}(B^c_{\frac{r}{2}})}\leq sm_{p,q}\,r^{\frac{N}{q}-N},
\end{equation}
where $m_{p,q}$ is a positive constant depending only on $u$.
\smallskip

\noindent\textbf{\textit{Case 2:}} If $|x|< r/2$ then $|x-y|<r$ provided $y\in\text{supp } u$ and thus the second integral in the definition of $A_r(s,p,x)$ is zero. Since $u\in C_c^{\alpha}(\rz^N)$ and by dominated convergence theorem, we then obtain
\begin{equation}\label{limit of Ar divide s as s to 0}
\begin{split}
   \lim_{s\to0^+}\frac{A_r(s,p,x)}{s}
   &=\lim_{s\to0^+}d_{N,p}(s)\int_{B_r(x)}\frac{g(u(x)-u(y))}{|x-y|^{N+sp}}dy\\
   &=\tilde{A}_{r}(p,x)
   :=C_{N,p}\int_{B_r(x)}\frac{g(u(x)-u(y))}{|x-y|^{N}}dy, 
\end{split}
\end{equation}
and the above convergence is uniform when $|x|<r/2$. Now, from \eqref{D r s p} we obtain 
$$
\lim_{s\to0^+}D_r(s,p)=\frac{C_{N,p}\omega_N}{p}=1\,\,\text{ and } D^{\prime}_r(0,p)=\rho_{N}-p\ln r=:k_{r}(p).
$$
Thus using the fact $u\in C_c^{\alpha}(\rz^N)$, we obtain for $1<q<\infty$
\begin{equation}\label{final Lq limit of Dr s p}
\lim_{s\to0^+}\left|\left|\frac{D_r(s,p)g(u)-g(u)}{s}-k_r(p)g(u)\right|\right|_{L^q(\rz^N)}=0.
\end{equation}
Note that
$$\ln r=\frac{1}{\omega_N}\int_{B_r(x)\setminus B_1(x)}\frac{dy}{|x-y|^N}.$$
Therefore, using this and from the above definitions of $\tilde{A}_{r}(p,\cdot)$ and $k_r(p)$ we obtain, for $x\in\rz^N$ 
\begin{equation}\label{sum of tilde Ar and kr}
\begin{split}
&\tilde{A_r}(p,x)+k_r(p)g(u(x))\\
&=C_{N,p}\left(\int_{B_r(x)}\frac{g(u(x)-u(y))}{|x-y|^{N}}dy-g(u(x))\int_{B_r(x)\setminus B_1(x)}\frac{dy}{|x-y|^N}\right)+\rho_Ng(u(x))
\\
&=L_{\Delta_p}u(x)-F(x), \,\,\text{ where }F(x):=C_{N,p}\int_{B_r^c(x)}\frac{g(u(x)-u(y))-g(u(x))}{|x-y|^N}dy.
\end{split}
\end{equation}
Now for $x\in B_{r/2}(0)$,  we have $F(x)=0$ as $u(y)=0$, and on the other hand, for $|x|\geq\frac r2$, we then proceed as in \textbf{Case 1} for $F$ to obtain 
$$
\|F\|_{L^q(B^c_{\frac r2})}\leq M_{p,q}r^{\frac Nq-N}\,\,\,\text{ for all }\,q\in(1,\infty).
$$
Finally, combining \eqref{final est of Ar}, \eqref{limit of Ar divide s as s to 0}, \eqref{final Lq limit of Dr s p}, and \eqref{sum of tilde Ar and kr}, we obtain
\begin{equation*}
  \limsup_{s\to0^+}\left\|\frac{(-\Delta_p)^s\,u-g(u)}{s}-\cL_{\Delta_p}u\right\|_{L^q(\rz^N)}\leq(m_{p,q}+M_{p,q})r^{\frac{N}{q}-N}
\end{equation*}
for all $r>0$,\, $1<q<\infty$, where $\cL_{\Delta_p}u$ is well-defined, bounded, and continuous by Lemma~\ref{some properties}. Consequently this yields 
$$
\lim_{s\to0^+}\left\|\frac{(-\Delta_p)^s\,u-g(u)}{s}-\cL_{\Delta_p}u\right\|_{L^q(\rz^N)}=0.
$$
In particular, it follows $\cL_{\Delta_p}u=L_{\Delta_p}u$ in $L^q(\R^N)$ for any $q\in(1,\infty)$. Note that following the above proof for the case $q=\infty$, since $(-\Delta_p)^su,g(u),\cL_{\Delta_p}u\in C(\R^N)$ (for $s$ small enough), we also have 
$$
L_{\Delta_p}u(x)=\lim_{s\to0^+}\frac{(-\Delta_p)^su(x)-g(u(x))}{s}=\cL_{\Delta_p}u(x)\quad\text{for every $x\in \R^N$},
$$
and the claim follows.
\end{proof}

\begin{remark}
\begin{enumerate}
\item Due to Theorem \ref{derivative}, we may replace $\cL_{\Delta_p}$ with $L_{\Delta_p}$ in Lemmas \ref{some properties} and \ref{some properties 2}.
\item Note that, the proof of Theorem \ref{derivative} actually gives the local uniform convergence of the difference quotient. That is, for every compact $K\subset \R^N$, $u\in C^{\alpha}_c(\R^N)$ for some $\alpha>0$, we have
$$
\sup_{x\in K}\Bigg|\frac{(-\Delta_p)^su(x)-g(u(x))}{s} -\cL_{\Delta_p}u(x)\Bigg|\to 0\quad\text{for $s\to 0^+$.}
$$
\end{enumerate}
\end{remark}

\section{A variational framework}\label{variational}
In this section, we give the detail of a weak formulation of problems involving the logarithmic $p$-Laplacian. For this, we introduce here a suitable functional space and summarize known properties of it. Let
\begin{equation*}\label{kernel k}
    \textbf{k}:\rz^N\setminus\{0\}\to\rz \text{ defined by }\mathbf{k}(z)=C_{N,p}1_{B_1}(z)|z|^{-N},
\end{equation*}
and 
\begin{equation*}\label{kernel j}
    \textbf{j}:\rz^N\to\rz \text{ defined by }\mathbf{j}(z)=C_{N,p}1_{\R^N\setminus B_1}(z)|z|^{-N}.
\end{equation*}
Then, we can write the integral representation of $L_{\Delta_p}$ given by \eqref{integral representation} with the above kernel functions as follows
\begin{equation}\label{integral representation with kernels}
L_{\Delta_p}u(x)=\int_{\rz^N}g(u(x)-u(y))\mathbf{k}(x-y)\,dy+\int_{\rn}\left(g(u(x)-u(y))-g(u(x))\right)\mathbf{j}(x-y)\,dy+\rho_{N}(p)g(u)(x).
\end{equation}
Let $\Omega\subset\rz^N$ be an open set and $p\in[1,\infty)$. Recall, the space $X^p_0(\Omega)$ defined  by 
$$
X^p_0(\Omega):=\{u\in L^p(\Omega): 1_{\R^N\setminus \Omega}u\equiv 0\text{ and },\,[u]_{X^p_0(\Omega)}<\infty\}
$$
endowed with the norm 
$$
\|u\|_{X^p_0(\Omega)}=\left(\|u\|_{L^p(\Omega)}^p+[u]^p_{X^p_0(\Omega)}\right)^{1/p},
$$
where 
$$
[u]_{X^p_0(\Omega)}^p=\iint\limits_{\rz^N\times\rz^N}|u(x)-u(y)|^p\mathbf{k}(x-y)\,dxdy.
$$
The space $X^p_0(\Omega)$ is a reflexive Banach space with respect to the norm $\|\cdot\|_{X^p_0(\Omega)}$ for $1<p<\infty$, see \cite[Section 3]{F23}. If $\Omega$ is an open set with finite measure or it is bounded in one direction, then $[\cdot]_{X^p_0(\Omega)}$ gives an equivalent norm for $X^p_0(\Omega)$,  which  follows from the following fractional Poincar\'e type result.

\begin{prop}[\protect{see\cite[Theorems 7.2 and 7.4]{F23}}]\label{poincare}
Let $p\in[1,\infty)$ and $\Omega\subset\R^N$ be an open set such that one of the following is true:
\begin{enumerate}
    \item $|\Omega|<\infty$;
    \item $\Omega$ is bounded in one direction, that is, there is an affine function $T:\R^N\to\R^N$, $Tv=Ov+c$ with a rotation $O$ and $c\in \R^N$ such that $T(\Omega)\subset (-a,a)\times \R^{N-1}$ for some $a>0$.
\end{enumerate}
Then, there exists a constant $C=C(N, p,\Omega)>0$ such that
$$\int_{\Omega}|u(x)|^p\,dx\leq C\iint\limits_{\rz^N\times\rz^N}|u(x)-u(y)|^p\mathbf{k}(x-y)\,dxdy\quad\text{ for all }u\in X^p_0(\Omega).
$$
\end{prop}

In order to study some variational problems related to $L_{\Delta_p}$ in some open sets $\Omega\subset \R^N$, the following compactness statement will play a pivotal role.

\begin{prop}[\protect{see\cite[Corollary 6.3]{F23}}]
Let $p\in(1,\infty)$ and $\Omega\subset\R^N$ open with $|\Omega|<\infty$. Then, the embedding $X^p_0(\Omega)\hookrightarrow L^p(\Omega)$ is compact.
\end{prop}

As shown, in \cite[Theorem 3.66]{F20} it holds that $C^{\infty}_c(\R^N)$ is dense in $X^p_0(\R^N)$. And by \cite[Theorem 3.76]{F20}, if $\Omega$ is a bounded subset with continuous boundary then $C^{\infty}_c(\Omega)$ is dense in $X^p_0(\Omega)$. Similar to \cite[Theorem 3.6]{FJ23}, we have the following if $\partial \Omega$ is of Lipschitz class.

\begin{prop}\label{dense}
    Let $\Omega$ be a bounded Lipschitz domain in $\rz^N$. Then, the space $C_c^\infty(\Omega)$ is a dense subset of the space $X^p_0(\Omega)$. Moreover, if $u\in X^p_0(\Omega)$ is a non-negative function then 
    \begin{itemize}
        \item[i)] there exists a non-decreasing sequence $\{u_n\}\subset X^p_0(\Omega)\cap L^{\infty}(\Omega)$ of functions in $X^p_0(\Omega)$ such that $u_n\geq0$ for all $n$ and $[u_n-u]_{X^p_0(\Omega)}\to0$ as $n\to\infty$;
        \smallskip
        \item[ii)] there exists a sequence $\{u_n\}\subset C_c^\infty(\Omega)$ such that $u_n\geq0$ for all $n$ and $u_n\to u$ in $X^p_0(\Omega)$ as $n\to\infty.$
    \end{itemize}
\end{prop}
\begin{proof}
The proof is analogous to the proof of \cite[Theorem 3.6]{FJ23}.
\end{proof}

To set up the weak formulation, we first observe the following.

\begin{lemma}\label{weak to strong}
  Let $1<p<\infty$ and $u,v\in C^{\infty}_c(\R^N)$. Then
  $$
  \cE_{L,p}(u,v)=\int_{\R^N}L_{\Delta_p}u(x)v(x)\,dx,
  $$
  where
  \begin{equation}\label{defi-bilinear}
  \cE_{L,p}(u,v):=\cE_{p}(u,v)+\cF_p(u,v)+\rho_{N}(p)\int_{\R^N}g(u(x))v(x)\,dx
  \end{equation}
  with
  \begin{align*}
  \cE_p(u,v)&:=\frac12\iint_{\R^N\times\R^N}g(u(x)-u(y))(v(x)-v(y))\mathbf{k}(x-y)\,dxdy,\\
  \cF_p(u,v)&:=\frac12\iint_{\R^N\times\R^N}\Big(g(u(x)-u(y))(v(x)-v(y))-g(u(x))v(x)-g(u(y))v(y)\Big)\mathbf{j}(x-y)\,dxdy.
  \end{align*}
\end{lemma}
\begin{proof}
This follows immediately from the notation in \eqref{integral representation with kernels}.
\end{proof}

It follows by the density statement that we also have
\begin{equation}\label{reltn btwn operator and bilinear}
 \cE_{L,p}(u,v)=\int_{\Omega}L_{\Delta_p}u(x)v(x)\,dx\quad\text{for all $u\in C^{\infty}_c(\R^N)$, $v\in X^p_0(\Omega)$,}   
\end{equation}
if $\Omega$ is bounded and has a continuous boundary. Indeed, in this case \eqref{reltn btwn operator and bilinear} follows from the density of $C^{\infty}_c(\R^N)$ in $X^p_0(\Omega)$ and in $L^p(\R^N)$ for the first and last summand in \eqref{defi-bilinear}. For the middle summand note that it holds for any $v\in L^p(\Omega)$, $u\in C^{\infty}_c(\Omega)$
\begin{align*}
 |\cF_p(u,v)|&\leq \iint_{\R^N\times\R^N}\Big|g(u(x)-u(y))-g(u(x))\Big||v(x)|\mathbf{j}(x-y)\,dxdy \\
 &= \int_{\R^N}|v(x)| \int_{\R^N}\frac{|g(u(x)-u(y))-g(u(x))|}{|x-y|^N}\, dy\,dx,
\end{align*}
where the inner integral is continuous and bounded as shown in Lemma \ref{some properties}. Since $\Omega$ is bounded, the approximation argument for \eqref{reltn btwn operator and bilinear} follows using the dominated convergence theorem for the middle term.

Similarly to the alternative representation of $L_{\Delta_p}$ given in Lemma \ref{some properties}(3), we can also rewrite $\cE_{L,p}$ in the following way.

\begin{prop}\label{alternative}
Let $1< p<\infty$ and let $\Omega$ be a bounded open subset of $\rn$ and $u,v\in X^p_0(\Omega)$. Then we have
\begin{equation*}\label{quadratic form in domain}
\cE_{L,p}(u,v)=\frac{C_{N,p}}{2}\iint_{\Omega\times\Omega}\frac{g(u(x)-u(y))(v(x)-v(y))}{|x-y|^N}\,dxdy+\int_{\Omega}\left(h_{\Omega}(x)+\rho_N(p)\right)g(u(x))v(x)\,dx.
\end{equation*}
In particular, we have
\begin{equation*}
\cE_{L,p}(u,u)=\frac{C_{N,p}}{2}\iint_{\Omega\times\Omega}\frac{|u(x)-u(y)|^p}{|x-y|^N}\,dxdy+\int_{\Omega}\left(h_{\Omega}(x)+\rho_N(p)\right)|u(x)|^p\,dx,
\end{equation*}
where $h_{\Omega}(x)$ is defined in \eqref{h-omega function}.
\end{prop}
\begin{proof}
Since $u,v\in X^p_0(\Omega)$, we have
\begin{equation*}
    \cE_p(u,v)=\frac{C_{N,p}}{2}\iint_{\substack{x,\,y\in\Omega\\|x-y|<1}}\frac{g(u(x)-u(y))(v(x)-v(y))}{|x-y|^N}\,dxdy+C_{N,p}\int_{\Omega}g(u(x))v(x)\left(\int_{B_1(x)\setminus\Omega}\frac{dy}{|x-y|^N}\right)dx,
\end{equation*}
and 
\begin{equation*}
\cF_p(u,v)=\frac{C_{N,p}}{2}\iint_{\substack{x,\,y\in\Omega\\|x-y|>1}}\frac{g(u(x)-u(y))(v(x)-v(y))-g(u(x))v(x)-g(u(y))v(y)}{|x-y|^N}\,dxdy.
\end{equation*}
Now, we can split the above integral by using the fact the domain $\Omega$ is bounded, and thus we get
\begin{equation*}
\cF_p(u,v)=\frac{C_{N,p}}{2}\iint_{\substack{x,\,y\in\Omega\\|x-y|>1}}\frac{g(u(x)-u(y))(v(x)-v(y))}{|x-y|^N}\,dxdy-C_{N,p}\int_{\Omega}g(u(x))v(x)\left(\,\int_{\Omega\setminus B_1(x)}\frac{dy}{|x-y|^N}\right)dx.
\end{equation*}
Therefore, by definition of $\cE_{L,p}$, we get the desired result.
\end{proof}

\begin{remark}
Similar to Proposition \ref{alternative}, one can show the following alternative for $p=1$. However, we do not assign the corresponding operator in this case. Let 
  \begin{equation*}\label{defi-bilinear p=1}
  \cE_{L,1}(u,u):=\cE_{1}(u,u)+\cF_1(u,u)+\rho_{N,1}\int_{\R^N}|u(x)|\,dx
  \end{equation*}
  with
  \begin{align*}
\cE_1(u,u)&:=\frac12\iint_{\R^N\times\R^N}|u(x)-u(y)|\mathbf{k}(x-y)\,dxdy,\\
  \cF_1(u,u)&:=\frac12\iint_{\R^N\times\R^N}\Big(|u(x)-u(y)|-|u(x)|-|u(y)|\Big)\mathbf{j}(x-y)\,dxdy.
  \end{align*}
  Then, for $u\in X^1_0(\Omega)$ with $\Omega\subset\R^N$ being a bounded open subset, we have
  $$
  \cE_{L,1}(u,u)=\frac{C_{N,1}}{2}\iint_{\Omega\times\Omega}\frac{|u(x)-u(y)|}{|x-y|^N}\,dxdy+\int_{\Omega}\left(h_{\Omega}(x)+\rho_{N,1}\right)|u(x)|\,dx.
  $$
\end{remark}

\begin{lemma}\label{lower bound absolut value}
For any $u\in X^p_0(\Omega)$, it holds 
$$
\cE_{L,p}(u,u)\geq \cE_{L,p}(|u|,|u|)
$$
and the inequality is strict, if $u$ changes sign.
\end{lemma}
\begin{proof}
Note that, we have by Proposition \ref{alternative}
\begin{align*}
\cE_{L,p}(u,u)&=\frac{C_{N,p}}{2}\iint_{\Omega\times\Omega}\frac{|u(x)-u(y)|^p}{|x-y|^N}\,dxdy+\int_{\Omega}\left(h_{\Omega}(x)+\rho_N(p)\right)|u(x)|^p\,dx\\
&\geq \frac{C_{N,p}}{2}\iint_{\Omega\times\Omega}\frac{\Big||u(x)|-|u(y)|\Big|^p}{|x-y|^N}\,dxdy+\int_{\Omega}\left(h_{\Omega}(x)+\rho_N(p)\right)|u(x)|^p\,dx\\
&=\cE_{L,p}(|u|,|u|).
\end{align*}
If $u$ changes sign, we have indeed
$$
\iint_{\Omega\times\Omega}\frac{|u(x)-u(y)|^p}{|x-y|^N}\,dxdy>\iint_{\Omega\times\Omega}\frac{\Big||u(x)|-|u(y)|\Big|^p}{|x-y|^N}\,dxdy,
$$
and this gives the additional claim.
\end{proof}

\section{A logarithmic boundary Hardy inequality}\label{sec:hardy}
In this section, we prove Theorem \ref{intro:hardy}.
This proof is split into two parts. In the first one, we prove a logarithmic boundary Hardy inequality under some assumptions on the Whitney decomposition of the set. In Subsection~\ref{subs:plump}, we give a~simple sufficient condition for  these assumptions to hold. As an illustration, we also  prove these assumptions for the half-space, which allows us to obtain explicit ---but possibly not optimal--- constants in the  logarithmic boundary Hardy inequality in this case. The last subsection is devoted to applications.

\subsection{Whitney decomposition and logarithmic boundary Hardy inequality}
  We call a~cube $Q\subset \R^N$ \emph{dyadic} if its side length is equal to $2^m$ for some integer $m$ and all coordinates of its vertices are equal to an integer times $2^m$. We denote by $\cD_m$ the collection of all dyadic cubes in $\R^N$ with side length~$2^m$
  and put $\cD = \bigcup_{m\in \Z} \cD_m$.

  Let $\cW(\Omega)$ be a~Whitney decomposition of an open set $\Omega\subset \R^N$ into cubes like in \cite{St70}. In particular,
  $\cW(\Omega) \subset \cD$, and for each $Q\in \cW(\Omega)$,
  \[
  \diam(Q) \leq \dist(Q,\partial \Omega) \leq 4\diam(Q).
  \]
  For $m\in \Z$, let $\cW_m(\Omega) =\{ Q\in \cW(\Omega): \ell(Q) = 2^m\}$, where $\ell(Q)$ is the side length of the cube~$Q$.

\begin{thm}\label{hardy}
  Let $\Omega\subset \R^N$, $\Omega\neq \R^N$ be an open set.
  We assume that
   there exist constants $C_1$, $C_2$, $C_3$ and an integer $j_0\leq 0$  such that
  for each Whitney cube $Q\in \cW_k(\Omega)$ with $k<j_0$ and for each $j$ such that $k<j\leq j_0$, there exists a~Whitney cube $E(Q,j)$ with the following properties:
  \begin{enumerate}
  \item[(i)] $C_1 2^j \leq \ell(E(Q,j)) \leq C_2 2^j$,
  \item[(ii)] it holds $|x-y| < C_3 2^j$ for all $x\in Q$ and $y\in E(Q,j)$.
  \end{enumerate}
  Additionally, we assume that the following holds:
  \begin{enumerate}
  \item[(iii)] there exist constants $\lambda<N$ and $C_4$ such that for each cube $Q_0 \in \cW_n(\Omega)$, $n\in\Z$,
  \[
  \#\{(Q,j): Q\in \cW_m(\Omega) \text{ and } E(Q,j) = Q_0 \} \leq C_4 2^{\lambda(n-m)}\quad\text{for all $m\in \Z$}.
  \]
  \end{enumerate}
Let $0<p<\infty$.  Then, there is $c=c(C_1, C_2, C_3, C_4, j_0, \lambda, p)>0$ such that for every $u \in L^p(\Omega)$
  \begin{equation}\label{eq:hardy}
  \int_\Omega |u(x)|^p \ln^+\left(\frac{1}{\delta_x}\right) \,dx
  \leq c \left(\  \iint_{\substack{\Omega\times \Omega\\ |x-y|<1}}\frac{|u(x)-u(y)|^p}{|x-y|^N}\,dy\,dx
  + \int_\Omega |u(x)|^p\,dx \right).
  \end{equation}
\end{thm}
As we aim at giving an explicit estimate on the constants, we use the following notation
$$
a\vee b:=\max\{a,b\}\quad\text{for $a,b\in \R$.}
$$
We note here that the proof is partially based on ideas from \cite{AdRoSa}.

\begin{proof}
  By decreasing $j_0$ if needed, we may and do assume that $C_3 2^{j_0} < 1$.
  Let $m < j_0$ and let $Q_m\in \cW_m(\Omega)$.
  Then
  \begin{align*}
    \int_{Q_m} &|u(x)|^p\,dx
    =
      \frac{1}{|E(Q_m,j)|} \int_{E(Q_m,j)}  \int_{Q_m} |u(x)-u(y)+u(y)|^p \,dx\,dy \\
      &\leq
      \frac{2^{p-1}\vee 1}{|E(Q_m,j)|} \int_{E(Q_m,j)}  \int_{Q_m} |u(x)-u(y)|^p\,dx\,dy
      +
      \frac{(2^{p-1}\vee 1)|Q_m|}{|E(Q_m,j)|} \int_{E(Q_m,j)}  |u(y)|^p\,dy \\
    &\leq
      \frac{C_3^N (2^{p-1}\vee 1)}{C_1^N} \int_{E(Q_m,j)} \int_{Q_m} 1_{|x-y|< C_32^j}\, \frac{|u(x)-u(y)|^p}{|x-y|^N} \,dx\,dy
      +
      \frac{(2^{p-1}\vee 1) 2^{mN}}{C_1^N 2^{jN}} \int_{E(Q_m,j)}  |u(y)|^p\,dy.
   \end{align*}
  Summing over all $j$ such that $m < j\leq j_0$, we obtain
  \begin{align*}
  (j_0 - m) \int_{Q_m} |u(x)|^p\,dx
  &\leq
  \frac{C_3^N (2^{p-1}\vee 1)}{C_1^N} \int_{Q_m} \int_{\Omega\cap B_1(x)} \frac{|u(x)-u(y)|^p}{|x-y|^N} \,dy\,dx\\
  &+  \frac{2^{p-1}\vee 1}{C_1^N}
  \sum_{j=m+1}^{j_0} 2^{(m-j)N} \int_{E(Q_m,j)}  |u(y)|^p\,dy.
  \end{align*}
  We sum over all $Q_m$ and all $m<j_0$, obtaining
  \begin{align*}
   \sum_{m<j_0} \sum_{Q_m}  (j_0 - m) \int_{Q_m} |u(x)|^p\,dx
   &\leq
   \frac{C_3^N (2^{p-1}\vee 1)}{C_1^N}
   \int_{\Omega} \int_{\Omega\cap B_1(x)} \frac{|u(x)-u(y)|^p}{|x-y|^N} \,dy\,dx\\
  &+ \frac{2^{p-1}\vee 1}{C_1^N}
   \sum_{m<j_0} \sum_{Q_m}
   \sum_{j=m+1}^{j_0} 2^{(m-j)N} \int_{E(Q_m,j)}  |u(y)|^p\,dy\\
   &=: I + S.
  \end{align*}
  Next, we rearrange the last sum $S$.
  To this end, we make the following observations.

  First, if $m<j_0$ and $j\leq j_0$, then by our first assumption
  \[
  j+ \log_2 C_1 \leq \log_2 \ell(E(Q_m, j) \leq j+\log_2 C_2 \leq j_0 + \log_2 C_2,
  \]
  that is, $E(Q_m, j) \in \cW_k(\Omega)$ with $k \leq  j_0 + \log_2 C_2$.

  Next, a fixed cube $R\in \cW_k(\Omega)$ with $k \leq  j_0 + \log_2 C_2$ is equal
  to $E(Q, j)$ for some $Q \in \cW_m(\Omega)$ and some $j$ for at most $C_4 2^{\lambda(k-m)}$
  pairs $(Q, j)$.

  Finally, in the sum we take $m<j$, therefore if $R \in \cW_k(\Omega)$ and $R=E(Q, j)$ with
  $Q\in \cW_m(\Omega)$, then $m<j \leq k - \log_2 C_1$.
  
  With these observations, we have
  \begin{align*}
    S &\leq
    \frac{2^{p-1}\vee 1}{C_1^N}
    \sum_{k \leq j_0 + \log_2 C_2} \sum_{R \in \cW_k(\Omega)} 
    \left( \sum_{m<k-\log_2 C_1} C_4 2^{(k-m)\lambda}  2^{(m-k+\log_2 C_2)N}\right)
    \int_{R}  |u(y)|^p\,dy \\
    &\leq
    \frac{(2^{p-1}\vee 1) C_4 C_2^N}{(1-2^{\lambda-N})C_1^{N-\lambda}}
    \sum_{k \leq j_0 + \log_2 C_2} \sum_{R \in \cW_k(\Omega)} 
    \int_{R}  |u(y)|^p\,dy 
    \leq \frac{(2^{p-1}\vee 1) C_4 C_2^N}{(1-2^{\lambda-N})C_1^{N-\lambda}}
    \int_{\Omega \cap \{\delta_y \leq 5C_2 2^{j_0}\sqrt{N}\}}  |u(y)|^p\,dy,
  \end{align*}
  and we obtain
  \begin{align*}
   \sum_{m<j_0} \sum_{Q_m}  (j_0 - m) \int_{Q_m} |u(x)|^p\,dx
   &\leq
   \frac{C_3^N (2^{p-1}\vee 1)}{C_1^N}
   \int_{\Omega} \int_{\Omega\cap B_1(x)} \frac{|u(x)-u(y)|^p}{|x-y|^N} \,dy\,dx \\
  &+ \frac{(2^{p-1}\vee 1) C_4 C_2^N}{(1-2^{\lambda-N})C_1^{N-\lambda}}
  \int_{\Omega \cap \{\delta_y \leq 5C_2 2^{j_0}\sqrt{N}\} }  |u(y)|^p\,dy.
  \end{align*}
  By the properties of Whitney cubes, if $x\in Q_m$, then
  \[
  \delta_x \geq 5\diam(Q_m),
  \]
  therefore
  \[
  \ln\frac{1}{\delta_x} \leq \ln\frac{1}{5\diam(Q_m)} = -\ln 5 - m \ln 2\leq - m \ln 2.
  \]
  Hence,
  \begin{align*}
    \int_\Omega &|u(x)|^p \ln^+(\frac{1}{\delta_x}) \,dx
    =
    \sum_{m<0} \sum_{Q\in \cW_m(\Omega)} \int_Q |u(x)|^p \ln^+(\frac{1}{\delta_x}) \,dx
    \leq
    \ln 2 \sum_{m<0} \sum_{Q\in \cW_m(\Omega)}  (-m) \int_{Q\cap \{\delta_x<1\}} |u(x)|^p \,dx\\    
    &\leq
    \ln 2 \sum_{m<j_0} \sum_{Q\in \cW_m(\Omega)}  (j_0 -m) \int_Q |u(x)|^p \,dx
     + \ln 2 \sum_{m<0} \sum_{Q\in \cW_m(\Omega)}  (-j_0) \int_{Q\cap \{\delta_x<1\}} |u(x)|^p \,dx\\
    &\leq
 \frac{C_3^N (2^{p-1}\vee 1) \ln 2}{C_1^N}   \int_{\Omega} \int_{\Omega\cap B_1(x)} \frac{|u(x)-u(y)|^p}{|x-y|^N} \,dy\,dx \\
 &+  \left( \frac{(2^{p-1}\vee 1) C_4 C_2^N}{(1-2^{\lambda-N})C_1^{N-\lambda}}-j_0\right) \ln 2
 \int_{\Omega\cap \{\delta_y \leq 5C_22^{j_0}\sqrt{N} \vee 1\}}  |u(y)|^p\,dy.
  \end{align*}
  The claim follows.
\end{proof}

\begin{remark}\label{rem:const}
We note that in the proof of Theorem \ref{hardy} if one wants to track the constant explicitly, one should replace $j_0$ above
by
\[
\min\{j_0, \lceil -\log_2 C_3-1 \rceil \},
\]
because of our initial assumption that $C_32^{j_0}<1$.
Furthermore, perhaps after further decrease of $j_0$, we may assume that $5C_2 2^{j_0}\sqrt{N}<1$,
which lets us reduce the integration domain in
 $\int_\Omega |u(x)|^p\,dx$ on the right side of \eqref{eq:hardy} to~a set $\{x\in \Omega : \delta_x < 1\}$.
\end{remark}

\subsection{Example: The half-space}
\begin{prop}\label{halfspace}
  A half-space $\Omega=\R^N_+=\{(x_1,\cdots,x_N):x_N>0\}$, $N\geq 1$, satisfies assumptions (i), (ii) and (iii)
  of Theorem~\ref{hardy} with constants $C_1=C_2=C_4=1$, $C_3=\sqrt{17N-1}$, $j_0=0$ and $\lambda=N-1$.
\end{prop}
\begin{proof}
  First, we construct a~possible Whitney decomposition of $\Omega$.
  Let $d\in \Z$ be such that $2^d\leq \sqrt{N} < 2^{d+1}$.
  As $\cW_m(\Omega)$ we take the collection of all dyadic cubes with side length $2^m$ contained
  in a~strip $\{(x_1,\cdots,x_N): 2^{m+d+1} \leq x_N \leq 2^{m+d+2} \}$.
  It is easy to check that the whole collection $\bigcup_{m\in \Z} \cW_m(\Omega)$ is a~Whitney decomposition of $\Omega$.

  We take $j_0:=0$. Let $k<j\leq j_0$ and let $Q\in \cW_k(\Omega)$.
  Then $Q=\times_{i=1}^N [2^k t_i, 2^k (t_i+1)]$ for some  $t_1, \cdots, t_N\in \Z$
  with $2^{k+d+1}\leq 2^k t_N < 2^k (t_N+1) \leq 2^{k+d+2}$. As $E(Q,j)$ we take the cube
  $\times_{i=1}^N [2^j \tau_i, 2^j (\tau_i+1)]$ from $\cW_j(\Omega)$ such that
  the $(N-1)$ dimensional cube $\times_{i=1}^{N-1} [2^j \tau_i, 2^j (\tau_i+1)]$ contains
  $\times_{i=1}^{N-1} [2^k t_i, 2^k (\tau_j+1)]$ and that $\tau_N=t_N$, see Figure~\ref{fig}.

  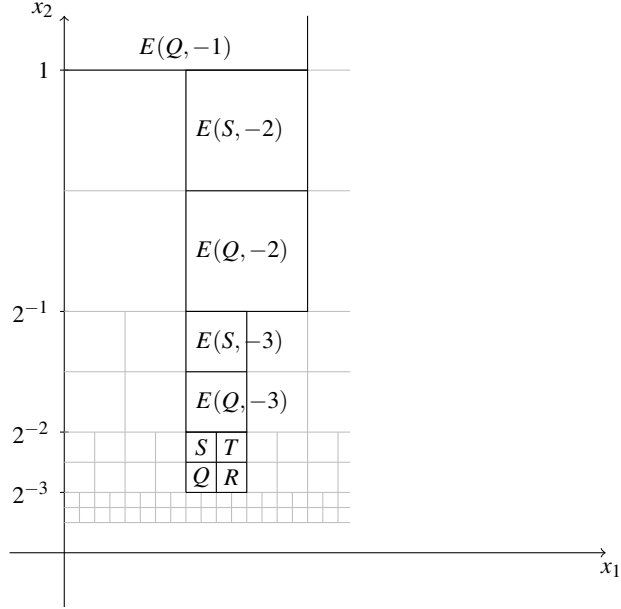
\begin{figure}
    \begin{tikzpicture}[x=1cm,y=1cm,scale=0.8]
      \draw[->,color=black] (-0.9,0) -- (8.9,0);
      \draw[->,color=black] (0,-0.9) -- (0,8.9);
      \draw[shift={(0,1)},color=black] (2pt,0pt) -- (-2pt,0pt) node[left]
           {\footnotesize $2^{-3}$};
      \draw[shift={(0,2)},color=black] (2pt,0pt) -- (-2pt,0pt) node[left]
           {\footnotesize $2^{-2}$};
      \draw[shift={(0,4)},color=black] (2pt,0pt) -- (-2pt,0pt) node[left]
           {\footnotesize $2^{-1}$};
      \draw[shift={(0,8)},color=black] (2pt,0pt) -- (-2pt,0pt) node[left]
           {\footnotesize $1$};
      \draw[shift={(0,9)},color=black] node[left]
           {\footnotesize $x_2$};
      \draw[shift={(9,0)},color=black] node[below]
           {\footnotesize $x_1$};
      \foreach \y in {0.5,0.75,1,1.5,2,3,4,6,8}
      \draw[color=lightgray] (0,\y) -- (4.7,\y);
      \foreach \x in {0.25,0.5,0.75,1,1.25,1.5,1.75,2,2.25,2.5,2.75,3,3.25,3.5,3.75,4,4.25,4.5}
      \draw[color=lightgray] (\x,0.5) -- (\x,1);

      \foreach \x in {0.5,1,1.5,2,2.5,3,3.5,4,4.5}
      \draw[color=lightgray] (\x,1) -- (\x,2);
      \foreach \x in {1,2,3,4}
      \draw[color=lightgray] (\x,2) -- (\x,4);
      \foreach \x in {2,4}
      \draw[color=lightgray] (\x,4) -- (\x,8);
      
      \draw[color=black] (2,1) -- (3,1) -- (3,2) -- (2,2) -- (2,1);
      \draw[color=black] (2.5,1) -- (2.5,2);
      \draw[color=black] (2,1.5) -- (3,1.5);
      \draw[color=black] (2,2) -- (3,2) -- (3,4) -- (2,4) -- (2,2);
      \draw[color=black] (2,3) -- (3,3);      
      \draw[color=black] (2,4) -- (4,4) -- (4,8) -- (2,8) -- (2,4);
      \draw[color=black] (2,6) -- (4,6);      

      \draw[shift={(2.25,1.25)},color=black] node {\footnotesize $Q$};
      \draw[shift={(2.75,1.25)},color=black] node {\footnotesize $R$};
      \draw[shift={(2.25,1.75)},color=black] node {\footnotesize $S$};
      \draw[shift={(2.75,1.75)},color=black] node {\footnotesize $T$};
      \draw[shift={(2,2.5)},color=black] node[right] {\footnotesize $E(Q,-3)$};
      \draw[shift={(2,3.5)},color=black] node[right] {\footnotesize $E(S,-3)$};
      \draw[shift={(2,5)},color=black] node[right] {\footnotesize $E(Q,-2)$};
      \draw[shift={(2,7)},color=black] node[right] {\footnotesize $E(S,-2)$};

      \draw[color=black] (0,8) -- (4,8) -- (4,8.9);      
      \draw[shift={(2,8)},color=black] node[above] {\footnotesize $E(Q,-1)$};
    \end{tikzpicture}
    \caption{Some Whitney cubes $Q$, $R$, $S$, $T\in \cW_{-4}(\R^2_+)$, see proof of Proposition~\ref{halfspace}. Here $d=0$, so these cubes lie in a~strip $2^{-3}\leq x_2 \leq 2^{-2}$.
  We have $E(Q,j)=E(R,j)$ and  $E(S,j)=E(T,j)$ for $j\in \{-3,-2,-1,0\}$ (depicted only for $j=-3$, $j=-2$ and partially $j=-1$).
    } \label{fig}
\end{figure}
  
  Since $E(Q,j) \in \cW_j(\Omega)$, condition (i) holds with $C_1=C_2=1$.

  To verify condition (ii), take $x\in Q$ and $y\in E(Q,j)$.
  Then by the choice of $E(Q,j)$, both points $(x_1,\cdots, x_{N-1})$  and $(y_1,\cdots, y_{N-1})$
  lie in the same $(N-1)$-dimensional cube of side length $2^j$. Therefore,
  \[
  \| (x_1,\cdots, x_{N-1}) -  (y_1,\cdots, y_{N-1}) \| \leq 2^j\sqrt{N-1}.
  \]
  Furthermore, $x_N, y_N \in (0, 2^{j+d+2}]$, hence $|x_N-y_N| \leq 2^{j+d+2}$.
This gives us a bound
\[
\|x-y\| \leq 2^j\sqrt{N-1 + 2^{2d+4}} \leq 2^j\sqrt{17N-1},
\]
hence (ii) follows with $C_3=\sqrt{17N-1}$ or, in fact, even with $C_3=\sqrt{N-1 + 2^{2d+4}}$.

Finally, to verify (iii), let $Q_0\in \cW_n(\Omega)$ and let $Q\in \cW_m(\Omega)$ be such
that $E(Q,j)=Q_0$. By our construction of $E(Q,j)$, this is only possible for  $j=n$ and $m<n$.
Therefore for $m\geq n$ the collection in (iii) is empty, hence there is nothing to prove,
and consequently, we may assume that $m<n$.
We note that the last coordinate of the vertices of $E(Q,j)$
is uniquely determined by $Q$, and vice versa, by the choice $\tau_N=t_N$. Therefore the number
of cubes in (iii) is the number of dyadic $(N-1)$-dimensional cubes of side length $2^m$
contained in a~fixed  dyadic $(N-1)$-dimensional cubes of side length $2^j=2^n$.
This number equals $2^{(N-1)(n-m)}$, hence (iii) holds with $C_4=1$ and $\lambda=N-1$.
\end{proof}

\begin{cor}
  Let $0<p<\infty$.  Then, for every $u \in L^p(\R^N_+)$
  \[
  \int_{\R^N_+} |u(x)|^p \ln^+\left(\frac{1}{x_N}\right) \,dx
  \leq c_1(p,N)  \iint_{\substack{{\R^N_+}\times {\R^N_+}\\ |x-y|<1}}\frac{|u(x)-u(y)|^p}{|x-y|^N}\,dy\,dx
  + c_2(p,N) \int_{\{x_N<1\}} |u(x)|^p\,dx,
  \]
  with\footnote{Here, $\lfloor x\rfloor$ denotes as usual the largest integer $n$ such that $n\leq x$.}
  \begin{align*}
    c_1(p,N) &= (17N-1)^{N/2}(2^{p-1}\vee 1)\ln 2,\\
    c_2(p,N) &= \left((2^{p}\vee 2) + 1 + \lfloor\frac{\log_2(17N-1)}{2}\rfloor\right) \ln 2.
  \end{align*}
  In the one or two-dimensional case, we may take
  \[
  c_1(p,1)=4(2^{p-1}\vee 1)\ln 2 \;,\quad
  c_1(p,2)=17(2^{p-1}\vee 1)\ln 2
  \quad \textrm{and} \quad  c_2(p,1)=c_2(p,2)=\left((2^{p}\vee 2) + 3 \right) \ln 2;
  \]
  in particular, $c_1(2,1)<6$, $c_1(2,2) < 24$ and $c_2(2,1)=c_2(2,2)<5$. 
\end{cor}
\begin{proof}
  The first part follows from Theorem~\ref{hardy} and Proposition~\ref{halfspace} together with Remark~\ref{rem:const}; for the integration domain $\{x_N<1\}$, this is because of Remark~\ref{rem:const} together with inequality $5C_22^{j_0} \leq 5/8 < 1$. For the constants for $N\in \{1,2\}$ we use the fact that in Proposition~\ref{halfspace}
  we may in fact take $C_3=\sqrt{N-1 + 2^{2d+4}}$, where $d=0$ for $N\leq 2$, thus giving
  $C_3=4$ or  $C_3=\sqrt{17}$ for $N=1$ or $2$, respectively, and (modified) $j_0=-3$ for $N\leq 2$.
\end{proof}

\subsection{Sufficient conditions for logarithmic boundary Hardy inequality}\label{subs:plump}
As we will show in
this subsection, in Corollary~\ref{cor:hardyplump} below, for inequality \eqref{eq:hardy} to hold
it is enough for the open set $\Omega$ to be \emph{locally plump} in the sense of the following definition.

\begin{definition}
A set $A\subset \R^N$ is \emph{locally $\kappa$-plump}
with $\kappa\in (0,1)$ if, for each $0<r< 1$ and each $x\in \bar{A}$, there
is $z\in \bar B_r(x)$ such that
$B_{\kappa r}(z)\subset A$. We also say that $A$ is locally plump, if there is some $\kappa\in(0,1)$ such that $A$ is locally $\kappa$-plump.
\end{definition}

In the usual definition of plumpness, the same is asserted for $0<r<\diam(A)$ instead of $0<r<1$ in our definition. This makes
a~difference for unbounded sets only. For example, an infinite strip is locally plump, but not plump.

The following lemma is the main tool in this subsection.
  \begin{lemma}\label{assouad}
    Assume that an open set $\Omega \subset \R^N$, $\Omega\neq\R^N$ is locally $\kappa$-plump.
    Let $n\in \Z$ and $R>0$ satisfy $0<2^n \le R < 4$.
Then, for every $\omega\in \partial \Omega$
\[
\#  \{Q\in \cW_n(\Omega)\,:\, Q\subset B_R(\omega)\}
\le C \bigg(\frac{R}{2^n}\bigg)^\lambda\,,
\]
with constants $\lambda<N$ and $C$ depending only on $N$ and $\kappa$.
  \end{lemma}

 We note that if $\Omega$ is plump, then $\partial \Omega$ is porous and hence  the (upper) Assouad dimension of $\partial \Omega$ is strictly smaller than $N$,
 see \cite[Theorem 5.2]{MR1608518}. The latter in turn implies condition (T1) on \cite[page 685]{DyVa}. This allows one to use  \cite[Lemma 10]{DyVa}, which is very similar to our Lemma~\ref{assouad}.
 However, going that route would require checking what changes when one replaces plumpness
 by local plumpness, which we assume, and also checking if different versions of definitions are compatible. 
 Therefore, we prefer to keep this subsection self-contained and prove Lemma~\ref{assouad} directly. To this end, we need the following two simple lemmas. The proof of the first one resembles a~part of the proof of
 \cite[Theorem 5.2]{MR1608518}.

\begin{lemma}\label{plump:alt}
  Assume that an open set $\Omega \subset \R^N$, $\Omega\neq\R^N$ is locally $\kappa$-plump.
  Then there exists a~natural number $K=K(N,\kappa)$ with the following property:
  for any cube $Q \in \cD_m$ with $m\leq 1$, there exists a~cube $R\in \cD_{m-K}$, $R\subset Q$
  which does not contain any Whitney cube.
\end{lemma}
\begin{proof}
  We will show that $K$ such that $2^K > 2(5\sqrt{N}/\kappa + 2)$ satisfies the property in the Lemma.
  To this end, let us fix an arbitrary integer $m\leq 1$ and an arbitrary cube $Q\in \cD_m$.
  Let $z_0\in Q$ be the centre of the cube~$Q$ and let $S\in \cD_{m-K}$ be a~cube
  containing $z_0$. If $S\cap \Omega = \emptyset$, then we may take $R=S$ and we are done.
  In the other case, choose a~point $x\in S\cap \Omega$, take $r=2^{m-1}-2^{m-K+1}$ and consider a~ball
  $B(x,r)$. Since $0<r<1$, by local plumpness it follows that there exists a~point $y\in \overline{B_r(x)}$ such that $B_{\kappa r}(y)\subset \Omega$.
  Let $T$ be a~Whitney cube containing $y$.

  First, we claim that $\ell(T)>2^{m-K}$. Indeed, since $T$ contains a~point $y$ with $\delta_y \geq \kappa r$, it holds
  \[
  \kappa r \leq \delta_y \leq \dist(T,\partial \Omega) + \diam(T) \leq 5 \diam(T),
  \]
  therefore $\diam(T) \geq \kappa r/5$, and $\ell(T) \geq \frac{\kappa r}{5\sqrt{N}} > 2^{m-K}$ by our choice of~$r$ and~$K$.

  Furthermore, we note that $\overline{B_r(x)}\subset \inn Q$, because $x$ and $z_0$ lie both in a~cube $S$ with
  $\ell(S)=2^{m-K}$ and $r + \ell(S) = 2^{m-1} - 2^{m-K} < \ell(Q)/2$. Therefore $y\in \inn Q$ and hence all dyadic cubes
  of side length $2^{m-K}$ containing $y$ lie inside $Q$ (usually there is just one such cube,
  unless $y$ lies on their boundary). As $R$ we take one of these cubes that is contained in~$T$
  and the lemma follows.
\end{proof}

\begin{lemma}\label{plump:lambda}
  Assume that an open set $\Omega \subset \R^N$, $\Omega\neq\R^N$ is locally $\kappa$-plump.
  There exists a~constant $0 < \lambda < N$ depending only on $N$ and $\kappa$, with the following property:
  for all cubes $Q \in \cD_m$ with $m\leq 1$ and all $s\in \{0,1,\cdots\}$,
  \begin{equation*}\label{plump:Wcount}
  \# \{ S \in \cW_{m-Ks}(\Omega) : S\subset Q \} \leq 2^{Ks \lambda},
  \end{equation*}
  where $K$ is as in Lemma~\ref{plump:alt}.
\end{lemma}
\begin{proof}
  Let $\lambda$ be such that $(2^{KN}-1)^{1/K} = 2^\lambda$; clearly, $0<\lambda<N$.
  Let us fix $s\geq 0$ and put
  \[
  \cF_{m-Kj} =  \{ S \in \cD_{m-Kj} : R \subset S \subset Q\textrm{ for some $R \in \cW_{m-Ks}(\Omega)$} \}.
  \]
  We will prove the following bound
  \begin{equation}\label{eq:ind}
  \# \cF_{m-Kj} \leq  2^{Kj \lambda}, \qquad j=0,1,\cdots, s,
  \end{equation}
  using induction by~$j$. For $j=0$ inequality \eqref{eq:ind} holds trivially, as $\cF_m$ may contain at most one cube, namely~$Q$.
  Assume that $\# \cF_{m-Kj} \leq  2^{Kj \lambda}$ for some $j\geq 0$ and $j<s$.
  Let $S\in \cF_{m-Kj}$. We decompose $S$ into a~union of $2^{KN}$ dyadic cubes from $\cD_{m-K(j+1)}$. By Lemma~\ref{plump:alt},
  at least one of those cubes  contains
  no Whitney cubes with side length $2^{m-Ks}$. Therefore $\cF_{m-K(j+1)}$ has at most $2^{KN}-1=2^{\lambda K}$ times more elements than $\cF_{m-Kj}$,
  which ends the proof of the induction step.

  The Lemma follows from \eqref{eq:ind} with $j=s$.
\end{proof}

  \begin{proof}[Proof of Lemma~\ref{assouad}]
     Let $K$ be as in Lemma~\ref{plump:alt} and $\lambda$ as in Lemma~\ref{plump:lambda}.
     Let  $\omega\in \partial \Omega$. Let  $n\in \Z$ and $R>0$ satisfy $0<2^n \le R < 4$.
     We choose $s\in \{0,1,\cdots\}$ such that
     \[
     2^{n+sK} \leq R < 2^{n+sK+K},
     \]
     and put $m=n+sK$. We define
     \[
     \cF = \{Q\in \cD_m : Q \cap B_R(\omega) \neq \emptyset \},
     \]
     i.e., $\cF$ is the minimal collection of dyadic cubes from $\cD_m$ that covers $B_R(\omega)$.
     Since $\frac{2R}{2^m} < 2^{K+1}$, it follows that $\# \cF \leq (2^{K+1}+1)^N$.

     To each cube from the collection $\cF$ we apply Lemma~\ref{plump:lambda} (here we use
     the assumption $R<4$, which implies $m\leq 1$). We conclude
     that each such a~cube contains at most $2^{sK\lambda}$ Whitney cubes from $\cW_{m-sK} = \cW_n$.
     Therefore
     \[
       \#  \{Q\in \cW_m(\Omega)\,:\, Q\subset B_R(\omega)\}
       \leq
       \# \cF \cdot 2^{sK\lambda}
       \leq (2^{K+1}+1)^N \bigg(\frac{R}{2^n}\bigg)^\lambda\,. \qedhere
       \]
  \end{proof}

\begin{prop}\label{plump}
 If an open set $\Omega \subset \R^N$ is locally $\kappa$-plump,
  then it satisfies the assumptions (i), (ii), and (iii) of Theorem~\ref{hardy}.
\end{prop}
\begin{proof}
  We take an integer $j_0\leq 0$ such that $(17\sqrt{N}+1)2^{j_0}< 4$.
  Let  $Q\in \cW_k(\Omega)$ with $k<j_0$  and let $j$ be such that $k<j\leq j_0$.
  By our assumption on $j_0$, we have $2^j \sqrt{N} \leq 2^{j_0}\sqrt{N} < 1$.
  Therefore by local plumpness condition with $r=2^j \sqrt{N}$
  and $x_0\in Q$ such that $\dist(x_0,\partial\Omega)=\dist(Q,\partial\Omega)$, there exists a~ball $B_{\kappa r}(z)\subset \Omega$
  with $z\in  \bar B_r(x_0)$. As $E(Q,j)$ we select any Whitney cube $\tilde{Q}$ containing $z$ (usually such a~cube is unique, unless $z$ lies at the boundary of Whitney cubes).
  Observe that
  \[
  \kappa r \leq \dist(z, \partial\Omega) \leq 5\, \diam(\tilde{Q}),
  \]
  and
  \[
  \diam(\tilde{Q}) \leq \dist(z, \partial \Omega)
   \leq r + \dist(x_0, \partial \Omega) \leq r + 4\sqrt{N} 2^k 
  \leq 3r,
  \]
  therefore $\frac{\kappa r}{5} \leq \diam(\tilde{Q}) \leq 3r$.
  Hence (i) holds with $C_1 = \frac{\kappa}{5}$ and $C_2 = 3$.

  By triangle inequality, we observe that for all $x\in Q$ and $y\in E(Q,j)$ it holds
  \[
  |x-y| \leq \diam(Q) + |x_0 - z| + \diam(\tilde{Q}) < r + r + 3r = 5r.
  \]
  Therefore (ii) holds with $C_3 = 5\sqrt{N}$.

  Finally, to verify (iii), let us fix a cube $Q_0 \in \cW_n(\Omega)$ and let $m\in \Z$. If $Q_0 = E(Q,j)$ for some $Q\in \cW_m(\Omega)$ and some $j>m$, then
  \[
  2^n = \ell(Q_0) \geq C_1 2^j \geq C_1 2^{m+1},
  \]
  hence $m \leq n - 1 - \log_2 C_1$. In other words, if $m > n-1-\log_2 C_1$, then the set $\{(Q,j): Q\in \cW_m(\Omega) \text{ and } E(Q,j) = Q_0 \}$ is empty.
  Therefore we may assume that $m \leq n - 1 - \log_2 C_1$.

  Let us consider some $j>m$ and suppose that $E(Q, j)=Q_0$ for some $Q\in \cW_m(\Omega)$.
  Let $\omega\in \partial \Omega$ be such that $\dist(\omega, Q_0) = \dist(Q_0, \partial\Omega)$.
  Then for all $x\in Q$, by triangle inequality
\[
    |\omega-x| \leq \dist(\omega, Q_0)  + C_3 2^j 
    \leq
    4\,\diam(Q_0) + C_32^j \leq (4C_2 \sqrt{N} + C_3)2^j,
\]
which means that $Q\subset B_R(\omega)$ with $R= (4C_2 \sqrt{N} + C_3 + 1)2^j$.
Therefore, by Lemma~\ref{assouad}
(note that our assumption on~$j_0$ implies that $R<4$), the number of such cubes $Q$ cannot exceed
\[
C_A (14\sqrt N)^{N+\lambda}\bigg(\frac{R}{2^m}\bigg)^\lambda
= C_A (14\sqrt N)^{N+\lambda} (4C_2 \sqrt{N} + C_3 + 1)^\lambda 2^{(j-m)\lambda}
\,.
\]
On the other hand,
\[
2^n = \ell(Q_0) = \ell(E(Q,j)) \geq C_1 2^j,
\]
so $m<j\leq n - \log_2 C_1$. Hence, the number of pairs $(Q,j)$ in question cannot exceed
\begin{align*}
\sum_{j\in \Z:m<j\leq n - \log_2 C_1} & C_A (14\sqrt N)^{N+\lambda} (4C_2 \sqrt{N} + C_3 + 1)^\lambda 2^{(j-m)\lambda} \\
&\leq  C_A (14\sqrt N)^{N+\lambda} (4C_2 \sqrt{N} + C_3 + 1)^\lambda \frac{C_1^{-\lambda}}{1-2^{-\lambda}} 2^{(n-m)\lambda}.
\end{align*}
So also (iii) holds with $C_4 = C_A (14\sqrt N)^{N+\lambda} (4C_2 \sqrt{N} + C_3 + 1)^\lambda \frac{C_1^{-\lambda}}{1-2^{-\lambda}}$. 
\end{proof}

Theorem \ref{intro:hardy} follows immediately as a special case from the following.

\begin{cor}\label{cor:hardyplump}
  Let $\Omega\subset \R^N$, $\Omega\neq \R^N$ be an open, locally plump set.
  In particular, $\Omega$ may be a Lipschitz set (bounded or not).
  Let $0<p<\infty$.  Then, there is $c>0$ such that for every $u \in L^p(\Omega)$ inequality \eqref{eq:hardy} holds.
\end{cor}
\begin{proof}
The corollary follows by combining Theorem~\ref{hardy} and Proposition~\ref{plump}, noting that Lipschitz sets are also locally plump.
\end{proof}

\subsection{Applications}
\begin{cor}
Let $\Omega$ be an open bounded Lipschitz subset of $\rn.$ Then there exists a positive constant $C=C(p,\Omega)$ such that for all $u\in C_c^\infty(\Omega)$ we have
$$
\int_{\Omega}|u(x)|^p k_{\Omega}(x)dx\leq C\iint_{\substack{x,\,y\in\Omega\\|x-y|<1}}\frac{|u(x)-u(y)|^p}{|x-y|^N}\,dxdy+C\int_{\Omega}|u(x)|^p\,dx,
$$
where $k_{\Omega}(x)=\int_{\rn\setminus\Omega}\textbf{k}(x-y)dy$ is the killing measure associated to the kernel $\textbf{k}$. 
\end{cor}
\begin{proof}
It follows from Corollary \ref{cor:hardyplump} and the following estimate of the killing measure $k_{\Omega}$ for bounded Lipschitz set $\Omega$:
$$
k_\Omega(x)\leq C+C\ln^+\left(\frac{1}{\delta(x)}\right)\,\,\text{ for a constant }C>0\,\,\text{ and }x\in\Omega.
$$
\end{proof}

\begin{cor}\label{extend}
Let $\Omega \subset \R^N$ be a nonempty, open, locally plump set. Let $u\in L^p(\Omega)$ be a function satisfying 
$$
\iint_{\substack{\Omega\times\Omega\\ |x-y|<1}}\frac{|u(x)-u(y)|^p}{|x-y|^N}\,dxdy<\infty.
$$
Then, the trivial extension $\tilde{u}:\R^N\to\R$ of $u$ belongs to $X^p_0(\Omega)$ and there is $C>0$ such that
$$
\|u\|_{X^p_0(\Omega)}\leq C\left(\|u\|_{L^p(\Omega)}+\iint_{\substack{\Omega\times\Omega\\ |x-y|<1}}\frac{|u(x)-u(y)|^p}{|x-y|^N}\,dxdy\right).
$$
\end{cor}
\begin{proof}
By symmetry it holds
\begin{align*}
[u]_{X^p_0(\Omega)}&=C_{N,p}\iint_{\substack{\Omega\times\Omega\\ |x-y|<1}}\frac{|u(x)-u(y)|^p}{|x-y|^N}\,dxdy+ 2C_{N,p}\int_{\Omega}|u(x)|^p\int_{\R^N\setminus \Omega}1_{B_1(x)}(y)|x-y|^{-N}\,dy\,dx.
\end{align*}
As calculated in \cite[equation (3.6)]{CW19}, it holds 
\begin{equation}\label{cutting off}
\int_{\R^N\setminus \Omega}1_{B_1(x)}(y)|x-y|^{-N}\,dy\leq 2\ln(\delta_x^{-1})\quad\text{for $x\in \Omega$ with $\delta_x<1$}.
\end{equation}
In particular, there is $C>0$ such that
$$
\int_{\Omega}|u(x)|^p\int_{\R^N\setminus \Omega}1_{B_1(x)}(y)|x-y|^{-N}\,dy\,dx\leq C\left(\,\|u\|_{L^p(\Omega)}^p+\int_\Omega |u(x)|^p \ln^+\left(\frac{1}{\delta_x}\right) \,dx\right).
$$
The claim follows from here using Corollary \ref{cor:hardyplump}.
\end{proof}

\begin{remark}\label{cut off}
Let us mention that if $\Omega$ is a bounded open Lipschitz set, then in particular $\Omega$ satisfies the assumptions of Corollaries \ref{cor:hardyplump} and \ref{extend}. Moreover, inequality \eqref{cutting off} also holds reversed with a suitable constant, that is, $\int_{\R^N\setminus \Omega}1_{B_1(x)}(y)|x-y|^{-N}\,dy$ is comparable to $\ln(\delta_x^{-1})$ for $x\in B_1(\partial\Omega)\cap\Omega$.
\end{remark}

The next result connects the density result by  Proposition \ref{dense} with a density of a seemingly different norm in connection with Theorem \ref{hardy}.

\begin{cor}\label{density2}
Let $\Omega\subset \R^N$ be an open bounded Lipschitz set. Then, 
$X^p_0(\Omega)$ coincides with the closure of $C^{\infty}_c(\Omega)$ with respect to the norm
$$
\|u\|_{X^p(\Omega)}=\left(\int_\Omega |u(x)|^p \,dx+\iint_{\substack{\Omega\times\Omega\\ |x-y|<1}}\frac{|u(x)-u(y)|^p}{|x-y|^N}\,dxdy\right)^{\frac{1}{p}}.
$$
\end{cor}
\begin{proof}
By the fact that $\int_{\R^N\setminus \Omega}1_{B_1(x)}(y)|x-y|^{-N}\,dy$ and $\ln(\delta_x^{-1})$ for $x\in B_1(\partial\Omega)\cap\Omega$ are comparable in $B_1(\partial\Omega)\cap \Omega$, see Remark \ref{cut off} or \cite[equations (3.6) and (3.7)]{CW19}, it follows that $\|\cdot\|_{X^p(\Omega)}$ and $\|\cdot \|_{X^p_0(\Omega)}$ are equivalent due to Corollary \ref{cor:hardyplump} with a similar calculation as in the proof of Corollary \ref{extend}. This finishes the proof.
\end{proof}

\begin{remark}
    Corollaries \ref{extend} and \ref{density2} should be seen as an analog to the fact that $W^{s,p}_0(\Omega)$, $W^{s,p}(\Omega)$, and $\cW^{s,p}_0(\Omega)$ coincide for $sp<1$. For the definition of these function spaces, see Section \ref{function spaces}.
\end{remark}

\section{The Dirichlet problem}\label{dirichlet problem}
Throughout this section, we assume $\Omega\subset \R^N$ is an open bounded set and $1<p<\infty$. Recall

$$
V(\Omega,\R^N)=\left\{u\in L^{p}_0\cap L^{\min\{1,p-1\}}_0\;:\; \int_{\Omega}\int_{\R^N}|u(x)-u(y)|^p\textbf{k}(x-y)\,dxdy<\infty\right\}.
$$
The space $V(\Omega,\R^N)$ is used to define supersolutions in the weak setting and uses that supersolutions must satisfy a certain minimal regularity across the boundary of $\Omega$, so that $X^p_0(\Omega)$ can be used as a test-function space. For the definition of the \textit{tail spaces} $L^{p}_0$, $L^{p-1}_0$, and $L^1_0$ see \eqref{tail space}. We begin by showing that $\cE_{L,p}$ is well-defined on $V(\Omega,\R^N)\times X^p_0(\Omega)$.

\begin{lemma}\label{lemma-welldefined}
Let $u\in V(\Omega,\R^N)$ and $v\in X^p_0(\Omega)$. Then $\cE_{L,p}(u,v)$ is well-defined
and finite.
\end{lemma}
\begin{proof}
Note that, $u\in L^p_{loc}(\R^N)$ by assumption and $u\in L^t_0$ for any $t\in[1,p]$, see Lemma \ref{included}. Thus, it holds 
$$
\int_{\R^N}|g(u(x))v(x)|\,dx=\int_{\Omega}|g(u(x))v(x)|\,dx\leq \|u\|_{L^{p}(\Omega)}^{p-1}\|v\|_{L^p(\Omega)},
$$
which bounds the last term in $\cE_{L,p}(u,v)$. For $\cF_p(u,v)$, note that
\begin{align*}
\iint_{\R^N\times\R^N}&\Big|g(u(x)-u(y))(v(x)-v(y))-g(u(x))v(x)-g(u(y))v(y))\Big|\textbf{j}(x-y)\,dxdy\\
&\leq \iint_{\Omega\times\Omega}\Big|g(u(x)-u(y))(v(x)-v(y))\Big|\textbf{j}(x-y)\,dxdy+2C_{N,p} |\Omega| \|u\|_{L^{p}(\Omega)}^{p-1}\|v\|_{L^p(\Omega)}\\
&\qquad + 2\int_{\Omega}\int_{\R^N\setminus \Omega}\Big|g(u(x)-u(y))-g(u(x))\Big||v(x)|\textbf{j}(x-y)\,dydx.
\end{align*}
Here, the first term can be bounded by
\begin{align*}
\iint_{\Omega\times\Omega}\Big|g(u(x)-u(y))&(v(x)-v(y))\Big|\textbf{j}(x-y)\,dxdy\\
&\leq \left(\,\,\iint_{\Omega\times\Omega}|u(x)-u(y)|^p\textbf{j}(x-y)\,dxdy\right)^{\frac{p-1}{p}}\left(\,\,\iint_{\Omega\times\Omega}|v(x)-v(y)|^p\textbf{j}(x-y)\,dxdy\right)^{\frac{1}{p}},
\end{align*}
where the right-hand side is finite, since
$$
\iint_{\Omega\times\Omega}|w(x)-w(y)|^p\textbf{j}(x-y)\,dxdy\leq 2^{p}C_{N,p}|\Omega|\|w\|_{L^p(\Omega)}^p
$$
for $w=u$ or $w=v$. For the last term in $\cF_p(u,v)$, we have to separately consider $p\in(1,2]$ and $p>2$.\\
\textbf{Case 1:} $p\in(1,2]$. By Lemma \ref{estimate-general-1} we have, for a constant $c=c(p)>0$
$$
|g(a-b)-g(a)|\leq c|b|^{p-1}\quad \text{for all $a,b\in \R$.}
$$
Thus,
\begin{align*}
\int_{\Omega}\int_{\R^N\setminus \Omega}\Big|g(u(x)-u(y))-g(u(x))\Big||v(x)|\textbf{j}(x-y)\,dydx&\leq c_p\int_{\Omega}|v(x)|\int_{\R^N\setminus \Omega}|u(y)|^{p-1}\textbf{j}(x-y)\,dydx\\
&\leq \tilde{c}\|v\|_{L^p(\Omega)} \|u\|_{L^{p-1}_0}^{p-1},
\end{align*}
for a constant $\tilde{c}=\tilde{c}(p,|\Omega|,N)>0$.\\
\textbf{Case 2:} $p>2$. By Lemmas \ref{estimate-general-1} and \ref{estimate-general-2}, we have for a constant $c=c(p)>0$
$$
|g(a-b)-g(a)|\leq c|b|(|b|^{p-2}+|a|^{p-2})\quad \text{for all $a,b\in \R$.}
$$
Thus, in a similar way
\begin{align*}
\int_{\Omega}&\int_{\R^N\setminus \Omega}\Big|g(u(x)-u(y))-g(u(x))\Big||v(x)|\textbf{j}(x-y)\,dydx\\
&\leq c\int_{\Omega}\int_{\R^N\setminus \Omega}|u(y)|^{p-1}|v(x)|\textbf{j}(x-y)\,dydx+c\int_{\Omega}\int_{\R^N\setminus \Omega}|u(x)|^{p-2}|v(x)||u(y)|\textbf{j}(x-y)\,dydx\\
&\leq\tilde{c}\|v\|_{L^p(\Omega)} \|u\|_{L^{p-1}_0}^{p-1}+\tilde{c}\|u\|_{L^{p-1}(\Omega)}^{p-2}\|v\|_{L^p(\Omega)}\|u\|_{L^{1}_0},
\end{align*}
for some constant $\tilde{c}=\tilde{c}(p,|\Omega|,N)>0$. It remains to bound $\cE_p(u,v)$. Using again that $v=0$ on $\R^N\setminus \Omega$, we have
\begin{align*}
\iint_{\substack{\R^N\times\R^N\\ |x-y|\leq 1}}\frac{|u(x)-u(y)|^{p-1}|v(x)-v(y)|}{|x-y|^N}\,dxdy&\leq2\int_{\Omega}\int_{B_1(x)}\frac{|u(x)-u(y)|^{p-1}|v(x)-v(y)|}{|x-y|^N}\,dxdy\\
&\leq 2\Bigg(\int_{\Omega}\int_{B_1(x)}\frac{|u(x)-u(y)|^{p}}{|x-y|^N}\,dxdy\Bigg)^{p-1}\cE_p(v,v)^{\frac{1}{p}}<\infty,
\end{align*}
and thus $\cE_{L,p}(u,v)$ is well-defined and finite.
\end{proof}

 Next, we show an alternative representation of $\cE_{L,p}$ as in Proposition \ref{alternative}, where one function is allowed to be in $V(\Omega,\R^N)$. 
\begin{prop}\label{alternative2}
Let $1<p<\infty$ and let $\Omega$ be a bounded open subset of $\rn$ and $u\in V(\Omega,\R^N)$ and $v\in X^p_0(\Omega)$. Then we have
\begin{equation*}\label{quadratic form in domain2}
\begin{split}
 \cE_{L,p}(u,v)&=\frac{C_{N,p}}{2}\iint_{\Omega\times\Omega}\frac{g(u(x)-u(y))(v(x)-v(y))}{|x-y|^N}\,dxdy+\int_{\Omega}\Big(\rho_N(p)+h_{\Omega}(x)\Big)g(u(x))v(x)\,dx\\
 &\qquad + C_{N,p}\int_{\Omega}v(x)\int_{\R^N\setminus \Omega}\frac{g(u(x)-u(y))-g(u(x))}{|x-y|^N}\,dy\,dx.
\end{split}
\end{equation*}
\end{prop}
\begin{proof}
Since $v\in X^p_0(\Omega)$, we have
\begin{equation*}
    \cE_p(u,v)=\frac{C_{N,p}}{2}\iint_{\substack{x,\,y\in\Omega\\|x-y|<1}}\frac{g(u(x)-u(y))(v(x)-v(y))}{|x-y|^N}\,dxdy+C_{N,p}\int_{\Omega}v(x)\left(\int_{B_1(x)\setminus\Omega}\frac{g(u(x)-u(y))}{|x-y|^N}\,dy\right)dx,
\end{equation*}
and 
\begin{align*}
\cF_p(u,v)&=\frac{C_{N,p}}{2}\iint_{\substack{x,\,y\in\Omega\\|x-y|>1}}\frac{g(u(x)-u(y))(v(x)-v(y))-g(u(x))v(x)-g(u(y))v(y)}{|x-y|^N}\,dxdy\\
&\quad +C_{N,p}\int_{\Omega} v(x)\left(\ \int_{(\R^N\setminus B_1(x))\setminus \Omega}\frac{g(u(x)-u(y))-g(u(x))}{|x-y|^N}\,dy\right)dx.
\end{align*}
Now, we can split the above first integral by using the fact that the domain $\Omega$ is bounded, and thus we get
\begin{align*}
\cF_p(u,v)&=\frac{C_{N,p}}{2}\iint_{\substack{x,\,y\in\Omega\\|x-y|>1}}\frac{g(u(x)-u(y))(v(x)-v(y))}{|x-y|^N}\,dxdy-C_{N,p}\int_{\Omega}g(u(x))v(x)\left(\,\int_{\Omega\setminus B_1(x)}\frac{dy}{|x-y|^N}\right)dx\\
&\qquad +C_{N,p}\int_{\Omega} v(x)\left(\int_{\R^N\setminus (B_1(x)\cup\Omega)}\frac{g(u(x)-u(y))-g(u(x))}{|x-y|^N}\,dy\right)dx.
\end{align*}
Therefore, by definition of $\cE_{L,p}$, we get the desired result.
\end{proof}

\begin{remark}\label{cut off2}
Let $u\in V(\Omega,\R^N)$, where $\Omega$ is a bounded open set. Then by Remark \ref{cut off} and Proposition \ref{poincare}, we immediately have that $u1_{\Omega}$ belongs to $X^p_0(\Omega)$. 
\end{remark}

\begin{definition}\label{weak solution def}
Given $f\in L^{\frac{p}{p-1}}(\Omega)$, a function $u\in V(\Omega,\R^N)$ is called a supersolution of $L_{\Delta_p}u=f$ in $\Omega$, if
$$
\cE_{L,p}(u,v)\geq \int_{\Omega} fv\,dx\quad\text{for all nonnegative $v\in X^p_0(\Omega)$,}
$$
where $\cE_{L,p}$ is given in \eqref{defi-bilinear}.
We call $u$ a subsolution of $L_{\Delta_p}u=f$ if $-u$ is a supersolution of this equation. A super- and subsolution is called a solution. We also say $u$ satisfies weakly $L_{\Delta_p}u\geq f$ in $\Omega$ if $u$ is a supersolution (resp. $\leq$ and $=$ for subsolutions and solutions).
\end{definition}

\begin{lemma}\label{restriction}
Let $\Omega\subset \R^N$ be a bounded open set. Let $f\in L^{\frac{p}{p-1}}(\Omega)$ and let $u\in V(\Omega,\R^N)$ satisfy weakly $L_{\Delta_p}u\geq f$ in $\Omega$. Then, the function $v=1_{\Omega}u$ belongs to $X^p_0(\Omega)$ and satisfies weakly\footnote{We emphasize that the right-hand side does in general not belong to $L^{\frac{p}{p-1}}(\Omega)$, but it belongs to the dual of $X^p_0(\Omega)$ and the inequality is to be understood with the usual generalization of weak solutions for the right-hand side.}
$$
L_{\Delta_p}v\geq f-C_{N,p}\int_{\R^N\setminus \Omega}\frac{g(v(x)-u(y))-g(v(x))}{|x-y|^N}\,dy\quad\text{in $\Omega$.}
$$
\end{lemma}
\begin{proof}
By Remark \ref{cut off2}, it follows that $v\in X^p_0(\Omega)$. Moreover, with Proposition \ref{alternative2} we have for any nonnegative $\phi\in X^p_0(\Omega)$:
\begin{align*}
\int_{\Omega}f(x)\phi(x)\,dx&\leq\cE_{L,p}(u,\phi)\\
&=\frac{C_{N,p}}{2}\iint_{\Omega\times\Omega}\frac{g(v(x)-v(y))(\phi(x)-\phi(y))}{|x-y|^N}\,dxdy+\int_{\Omega}\Big(\rho_N(p)+h_{\Omega}(x)\Big)g(v(x))\phi(x)\,dx\\
 &\qquad + C_{N,p}\int_{\Omega}\phi(x)\int_{\R^N\setminus \Omega}\frac{g(v(x)-u(y))-g(v(x))}{|x-y|^N}\,dy\,dx\\
 &=\cE_{L,p}(v,\phi)+C_{N,p}\int_{\Omega}\phi(x)\int_{\R^N\setminus \Omega}\frac{g(v(x)-u(y))-g(v(x))}{|x-y|^N}\,dydx.
 \end{align*}
The claim follows from here.
\end{proof}

\begin{lemma}[Scaling behavior of solutions]\label{weak scaling}
Let $f\in L^{\frac{p}{p-1}}(\Omega)$ and let $u\in V(\Omega,\R^N)$ be a supersolution of $L_{\Delta_p}u=f$ in $\Omega$. Then, the function $u_r:\R^N\to\R$, $u_r(x)=u(x/r)$ for $r>0$ belongs to $V(r\Omega,\R^N)$ and is a supersolution of $L_{\Delta_p}u_r=f_r-p\ln(r)g(v)$ in $r\Omega$, where $f_r(x)=f(x/r)$.
\end{lemma}
\begin{proof}
Let $r>0$. By substitution we have
\begin{align*}
\int_{r\Omega}\int_{B_1(y)}\frac{|u(x/r)-u(y/r)|^p}{|x-y|^N}\,dxdy&=\int_{\Omega}\int_{B_{1}(r\tilde{y})}\frac{|u(x/r)-u(\tilde{y})|^p}{|x-r\tilde{y}|^N}r^N\,dxd\tilde{y}\\
&=\int_{\Omega}\int_{r B_{1/r}(y)}\frac{|u(x/r)-u(\tilde{y})|^p}{|x-r\tilde{y}|^N}r^N\,dxd\tilde{y}\\
&=r^N\int_{\Omega}\int_{B_{1/r}(\tilde{y})}\frac{|u(\tilde{x})-u(\tilde{y})|^p}{|\tilde{x}-\tilde{y}|^N}\,dxd\tilde{y}.
\end{align*}
If $r\geq 1$, then 
$$
\int_{r\Omega}\int_{B_1(y)}\frac{|u(x/r)-u(y/r)|^p}{|x-y|^N}\,dxdy\leq r^N\int_{\Omega}\int_{B_1(y)}\frac{|u(x)-u(y)|^p}{|x-y|^N}\,dxdy<\infty,
$$
and if $r<1$, note that
$$
\int_{\Omega}\int_{B_{1/r}(y)\setminus B_1(y)}\frac{|u(x)-u(y)|^p}{|x-y|^N}\,dxdy\leq \int_{\Omega}\int_{B_{1/r}(y)\setminus B_1(y)}|u(x)-u(y)|^p\,dxdy<\infty,
$$
since $u\in L^p_{loc}(\R^N)$. This shows $u_r\in V(r\Omega,\R^N)$.
Next, note that differently to the linear case, we cannot argue with Lemma \ref{dense}, Proposition \ref{weak to strong}, and Lemma \ref{some properties}(2) due to the nonlinearity. Let $\phi\in X(r\Omega)$ be a nonnegative function and note that $\psi:\R^N\to\R$, $\psi(x)=r^{N}\phi(rx)$ belongs to $X( \Omega)$ with a similar argument as above and it is also nonnegative. Then
\begin{equation}\label{eq:scaling right}
\int_{r\Omega} f_r(x)\phi(x)\,dx=\int_{\Omega}f(x)\psi(x)\,dx,
\end{equation}
and similarly
\begin{align}\label{eq:scaling left}
\int_{\R^N}g(u_r(x))\phi(x)\,dx&=\int_{\R^N}g(u(x))\psi(x)\,dx.
\end{align}
Moreover,
\begin{align*}
\cE_p(u_r,\phi)&=\frac12\iint_{\R^N\times\rn}(g(u(x))-g(u(y)))(\psi(x)-\psi(y))r^{N}{\bf k}(r(x-y))\,dxdy,
\end{align*}
and
\begin{align*}
\cF_p(u_r,\phi)&=\frac12\iint_{\R^N\times\rn}\Big(g(u(x)-u(y))(\psi(x)-\psi(y))-g(u(x))\psi(x)-g(u(y))\psi(y)\Big)r^{N}{\bf j}(r(x-y))\,dxdy.
\end{align*}
Since 
$$
\frac{r^N{\bf k}(rz)}{C_{N,p}}=1_{B_{\frac{1}{r}}}(z)|z|^{-N}\quad\text{and}\quad \frac{r^N{\bf j}(rz)}{C_{N,p}}=1_{B_{\frac{1}{r}}^c}(z)|z|^{-N},
$$
we may proceed as follows. If $r\in(0,1)$, then
$$
r^N{\bf k}(rz)={\bf k}(z)+C_{N,p}1_{[B_{\frac{1}{r}}\setminus B_1]}(z)|z|^{-N}\quad\text{and}\quad r^N{\bf j}(rz)={\bf j}(z)-C_{N,p}1_{[B_{\frac{1}{r}}\setminus B_1]}(z)|z|^{-N},
$$
and using the symmetry of ${\bf j}$
\begin{align*}
\cE_p(u_r,\phi)+\cF_p(u_r,\phi)&=\cE_p(u,\psi)+\cF_p(u,\psi)+C_{N,p}\int_{\R^N}g(u(x))\psi(x)\int_{B_{\frac1r}(x)\setminus B_1(x)}|x-y|^{-N}\,dy\,dx\\
&=\cE_p(u,\psi)+\cF_p(u,\psi)-C_{N,p}\omega_N\ln(r)\int_{\R^N}g(u(x))\psi(x)\,dx.
\end{align*}
Thus, with \eqref{eq:scaling right}, \eqref{eq:scaling left}, and using that $C_{N,p}\omega_N=p$
\begin{align*}
\cE_{L_{\Delta_p}}(u_r,\phi)&=\cE_{L_{\Delta_p}}(u,\psi)-C_{N,p}\omega_N\int_{\R^N}\ln(r)g(u(x))\psi(x)\,dx\\
&\geq \int_{\R^N}f(x)\psi(x)-p\ln(r)g(u(x))\psi(x)\,dx\\
&=\int_{\R^N}\Big(f_r(x)-p\ln(r)g(u_r(x))\Big)\phi(x)\,dx
\end{align*}
as claimed. The case $r>1$ follows similarly.
\end{proof}

\subsection{On \texorpdfstring{$L^{\infty}$}{L-infty} bounds}
\begin{thm}\label{bounded1}
Let $A,B>0$ and assume $u\in V(\Omega,\R^N)$ satisfies
$$
\cE_{L,p}(u,v)\leq \int_{\Omega}(A+B|u|^{p-1})v\,dx\quad\text{for all $v\in X^p_0(\Omega)$.}
$$
If $u^+\in L^{\infty}(B_1(\Omega)\setminus \Omega)$, then $u^+\in L^{\infty}(\Omega)$. More precisely, there is $C=C(N,p,\Omega,B)>0$ such that
\begin{enumerate}
    \item[] if $p\in(1,2]$, we have 
    $$
    \|u^+\|_{L^{\infty}(\Omega)}\leq C\Big(A^{\frac{1}{p-1}}+ \|u\|_{L^{\infty}(B_1(\Omega)\setminus \Omega)}+\|u\|_{L^{p}(\Omega)}+\|u\|_{L^{p-1}_0}\Big),
    $$
    \item[] and if $p>2$, we have
    $$
    \|u^+\|_{L^{\infty}(\Omega)}\leq C\Big(1+A+\|u^+\|_{L^{\infty}(B_1(\Omega)\setminus \Omega)}+\|u\|_{L^p(\Omega)}+\|u\|_{L^1_0}+\|u\|_{L^{p-1}_0}\Big).
    $$
\end{enumerate}
\end{thm}
\begin{remark}
By definition, it follows that the assumption in Theorem \ref{bounded1} implies that $u$ satisfies weakly
$$
L_{\Delta_p}u\leq A+B|u|^{p-1}\quad\text{in $\Omega$.}
$$
In particular, Theorem \ref{bounded1} holds for any function $u\in V(\Omega,\R^N)$ which satisfies weakly  
$$
L_{\Delta_p}u\leq A+Bg(u)\quad\text{in $\Omega$,}\qquad u\leq 0\quad\text{in $\R^N\setminus \Omega$},
$$
for some constants $A,B>0$.
\end{remark}
\begin{proof}[Proof of Theorem \ref{bounded1}]
We begin by estimating $\cF_p(u,v)$, where $v\in X^p_0(\Omega)$ is a nonnegative function. By symmetry, it holds
\begin{align}
\cF_p(u,v)&=\frac12\iint_{\Omega\times\Omega}\Big(g(u(x)-u(y))(v(x)-v(y))-g(u(x))v(x)-g(u(y))v(y)\Big)\textbf{j}(x-y)\,dxdy\notag\\
&\qquad+ \int_{\Omega}\int_{\R^N\setminus \Omega}\Big(g(u(x)-u(y))(v(x)-v(y))-g(u(y))v(y)\Big)\textbf{j}(x-y)\,dxdy\notag\\
&=\int_{\R^N}\int_{\Omega}\Big(g(u(x)-u(y))-g(u(x))\Big)v(x)\textbf{j}(x-y)\,dxdy.\label{bdd est f}
\end{align}
Consider next $w=\phi u$, where $\phi\in C^{\infty}_c(\R^N)$ is such that $\phi=1$ in $B_{1/2}(\Omega)$ and $\phi=0$ on $\R^N\setminus B_1(\Omega)$. Then, for $v\in X^p_0(\Omega)$, $v\geq0$ we have, using the fact that $\supp\,\textbf{k}\subset \overline{B_1(0)}$
\begin{align*}
\cE_p(u,v)&=\cE_p(w+(1-\phi)u,v)\\
&=\frac12\int_{B_1(\Omega)}\int_{B_1(\Omega)}g(w(x)-w(y))(v(x)-v(y))\textbf{k}(x-y)\,dxdy\\
&\quad+\int_{\R^N\setminus B_1(\Omega)}\int_{\Omega}g(w(x)-w(y)-(1-\phi(y))u(y))v(x)\textbf{k}(x-y)\,dxdy\\
&=\frac12\int_{B_1(\Omega)}\int_{B_1(\Omega)}g(w(x)-w(y))(v(x)-v(y))\textbf{k}(x-y)\,dxdy=\cE_p(w,v).
\end{align*}
From here, we split $\textbf{k}$ into
$$
k_{\delta}(z)=1_{B_\delta}(z)|z|^{-N}\quad\text{and}\quad q_{\delta}(z)=1_{B_1\setminus B_\delta}(z)|z|^{-N},
$$
where $\delta\in(0,1)$ is such that
$$
\|q_{\delta}\|_{L^1(\R^N)}\geq B+|\rho_{N}(p)|+1.
$$
We write $\cE_p(w,v)=\cE_\delta(w,v)+\cG_{\delta}(w,v)$, where
\begin{align*}
\cE_{\delta}(w,v)&=\frac12\iint_{\R^N\times\rn}g(w(x)-w(y))(v(x)-v(y))k_{\delta}(x-y)\,dxdy\quad\text{and}\\
\cG_{\delta}(w,v)&=\frac12\iint_{\R^N\times\rn}g(w(x)-w(y))(v(x)-v(y))q_{\delta}(x-y)\,dxdy\\
&=\int_{\R^N}\int_{\Omega}g(w(x)-w(y))v(x)q_{\delta}(x-y)\,dxdy.
\end{align*}
Then
\begin{align*}
   \int_{\Omega}(A+B|u|^{p-1})v\,dx&=\int_{\Omega}(A+B|w|^{p-1})v\,dx\geq \cE_{L,p}(w+(1-\phi)u,v)\\
   &=\cE_{\delta}(w,v)+\cG_{\delta}(w,v)+\cF_p(u,v)+\rho_{N}(p)\int_{\Omega}g(w(x))v(x)\,dx.
\end{align*}
So after rearranging and using \eqref{bdd est f}, we have 
\begin{align}
\cE_{\delta}(w,v)&\leq \int_{\Omega}(A+(B+|\rho_{N}(p)|)|u|^{p-1})v\,dx-\cG_{\delta}(w,v)-\cF_p(u,v)\notag\\
&=\int_{\Omega}\Big(A+(B+|\rho_{N}(p)|)|w(x)|^{p-1}-\|q_{\delta}\|_{L^1(\R^N)}g(w(x))\Big)v(x)\,dx\notag\\
&\qquad -\int_{\R^N}\int_{\Omega}\Big(g(w(x)-w(y))-g(w(x)))v(x)q_{\delta}(x-y)\,dxdy\notag\\
&\qquad-\int_{\R^N}\int_{\Omega}\Big(g(w(x)-u(y))-g(w(x))\Big)v(x)\textbf{j}(x-y)\,dxdy,\label{bdd est 3}
\end{align}
noting that $u(x)=w(x)$ for $x\in \Omega$.
Next, we choose as a testfunction $v=v_t=(w-t)^+$, where 
$$
t\geq \|u^+\|_{L^{\infty}(B_1(\Omega)\setminus \Omega)}
$$
is to be chosen. Note that $v\in X^p_0(\Omega)$ by Remark \ref{cut off2} using that $v1_{\Omega}=v$ for any such $t$. Notice that $w\geq t$ in $\supp\,v_t$ and we have
\begin{equation}\label{bdd est 2}
    \begin{split}
     \cE_{\delta}(w,v_t)&=\cE_{\delta}(v_t,v_t)\\
&\quad+\frac12\iint_{\R^N\times\R^N}\Big(g(w(x)-w(y))-g(v_t(x)-v_t(y))  \Big)(v_t(x)-v_t(y))k_{\delta}(x-y)\,dxdy.   
    \end{split}
\end{equation}
Note that with
$$
Q(x,y)=(p-1)\int_0^1\Big|v_t(x)-v_t(y)+\tau\Big(w(x)-w(y)-v_t(x)+v_t(y)\Big)\Big|^{p-2}\,d\tau,
$$
we have
\begin{align*}
&\left(g(w(x)-w(y))-g(v_t(x)-v_t(y))\right)(v_t(x)-v_t(y))\\
&=Q(x,y)\Big(w(x)-w(y)-v_t(x)+v_t(y)\Big)(v_t(x)-v_t(y))\\
&=-Q(x,y)((w(x)-t)^--(w(y)-t)^-)(v_t(x)-v_t(y))\\
&=-Q(x,y)\Big(-(w(x)-t)^-(x)v_t(y)-(w(y)-t)^-v_t(x)\Big)\geq0.
\end{align*}
Thus, by the Poincar\'e inequality, Proposition \ref{poincare}, we have for some $C>0$ (depending also on $\delta$) combined with \eqref{bdd est 2} and \eqref{bdd est 3} 
\begin{align}
0\leq C\|v_t\|_{L^p(\Omega)}^p&\leq \cE_{\delta}(v_t,v_t)\leq \cE_{\delta}(w,v_t)
\notag\\
&\leq\int_{\Omega}\Big(A-w^{p-1}(x)\Big)v_t(x)\,dx-\int_{\R^N}\int_{\Omega}\Big(g(w(x)-w(y))-w^{p-1}(x)\Big)v_t(x)q_{\delta}(x-y)\,dxdy\label{bdd est 4}\\
&\qquad-\int_{\R^N}\int_{\Omega}\Big(g(w(x)-u(y))-g(w(x))\Big)v_t(x)\textbf{j}(x-y)\,dxdy.\notag
\end{align}
We discuss now separately the cases $p\in(1,2]$ and $p>2$.\\
\smallskip
\textbf{Case 1:} $p\in(1,2]$. By Lemma \ref{estimate-general-1}, there is $c=c(p)>0$ such that
$$
g(a-b)-g(a)\geq -c|b|^{p-1}\quad\text{for all $a,b\in \R$.}
$$
With this, we have from \eqref{bdd est 4}
\begin{align*}
 C\|v_t\|_{L^p(\Omega)}^p&\leq  \int_{\Omega}\Big(A-w^{p-1}(x)\Big)v_t(x)\,dx\notag\\
&\qquad +c\int_{\Omega}v_t(x) \Bigg(\,\,\int_{\R^N}|w(y)|^{p-1}q_{\delta}(x-y)\,dy +|u(y)|^{p-1}\textbf{j}(x-y)\,dy\Bigg)\,dx\\
&\leq  \int_{\Omega}\Big(A-w^{p-1}(x)\Big)v_t(x)\,dx\notag\\
&\qquad +c\int_{\Omega}v_t(x) \Bigg(\,\|w\|_{L^p(\R^N)}^{p-1}\|q_{\delta}\|_{L^p(\R^N)} +\tilde{c}\|u\|_{L^{p-1}_0}^{p-1}\Bigg)\,dx,
\end{align*}
for a constant $\tilde{c}>0$ independent of $u$. Now, choose  $t\geq\|u^+\|_{L^{\infty}(B_1(\Omega)\setminus\Omega)}$ such that also
$$
A+c\|w\|_{L^p(\R^N)}^{p-1}\|q_{\delta}\|_{L^p(\R^N)} +c\tilde{c}\|u\|_{L^{p-1}_0}^{p-1}\leq t^{p-1}.
$$
Then it follows $\|v_t\|_{L^p(\Omega)}=0$. Thus, choosing $t$ such that
$$
t^{p-1}=A+c\|w\|_{L^p(\R^N)}^{p-1}\|q_{\delta}\|_{L^p(\R^N)} +c\tilde{c}\|u\|_{L^{p-1}_0}^{p-1}+\|u^+\|_{L^{\infty}(B_1(\Omega)\setminus\Omega)}^{p-1},
$$
we have for a.e. $x\in \Omega$,
$$
u(x)\leq t\leq C\Big(A^{\frac{1}{p-1}}+ \|u^+\|_{L^{\infty}(B_1(\Omega)\setminus \Omega)}+\|u\|_{L^{p}(\Omega)}+\|u\|_{L^{p-1}_0}\Big),
$$
where $C=C(N,p,\Omega,B)>0$, as claimed.\\
\smallskip
\textbf{Case 2:} $p>2$. We use the second inequality in Lemma \ref{estimate-general-1}, which gives
$$
g(a-b)-g(a)\geq -(p-1)\Big(|b|^{p-1}+|b||a|^{p-2}\Big)\quad\text{for all $a,b\in \R$.}
$$
With this, we have from \eqref{bdd est 4}
\begin{align*}
 C\|v_t\|_{L^p(\Omega)}^p&\leq  \int_{\Omega}\Big(A-w^{p-1}(x)\Big)v_t(x)\,dx\notag\\
&\qquad +(p-1)\int_{\Omega}v_t(x) \Bigg(\,\,\int_{\R^N}|w(y)|^{p-1}q_{\delta}(x-y)\,dy +|u(y)|^{p-1}\textbf{j}(x-y)\,dy\Bigg)\,dx\\
&\qquad +(p-1)\int_{\Omega}v_t(x)w(x)^{p-2} \Bigg(\,\,\int_{\R^N}|w(y)| q_{\delta}(x-y)\,dy +|u(y)|\textbf{j}(x-y)\,dy\Bigg)\,dx\\
&\leq  \int_{\Omega}\Big(A-w^{p-1}(x)\Big)v_t(x)\,dx\notag\\
&\qquad +(p-1)\int_{\Omega}v_t(x) \Bigg(\,\|w\|_{L^p(\R^N)}^{p-1}\|q_{\delta}\|_{L^p(\R^N)} +\tilde{c}\|u\|_{L^{p-1}_0}^{p-1}\Bigg)\,dx\\
&\qquad +(p-1)\int_{\Omega}v_t(x)w(x)^{p-2} \Bigg(\,\|w\|_{L^p(\R^N)}\|q_{\delta}\|_{L^{p'}(\R^N)} +\tilde{c}\|u\|_{L^{1}_0}\Bigg)\,dx\\
&=  \int_{\Omega}\Bigg(A+(p-1)\Big(\,\|w\|_{L^p(\R^N)}^{p-1}\|q_{\delta}\|_{L^p(\R^N)} +\tilde{c}\|u\|_{L^{p-1}_0}^{p-1}\Big)\Bigg)v_t(x)\,dx\notag\\
&\qquad -\int_{\Omega}v_t(x)w(x)^{p-2}\Bigg(w(x) - (p-1)\Big(\,\|w\|_{L^p(\R^N)}\|q_{\delta}\|_{L^{p'}(\R^N)} +\tilde{c}\|u\|_{L^{1}_0}\Big)\Bigg)\,dx,
\end{align*}
for a constant $\tilde{c}>0$ independent of $u$. From here, choosing 
$$
t\geq C\Big(1+A+\|u^+\|_{L^{\infty}(B_1(\Omega)\setminus \Omega)}+\|u\|_{L^p(\Omega)}+\|u\|_{L^1_0}+\|u\|_{L^{p-1}_0}\Big),
$$
for a constant $C=C(N,p,\Omega,B)>0$ entails the claim.
\end{proof}

\begin{proof}[Proof of Theorem \ref{bounded2}]
The statement follows immediately from Theorem \ref{bounded1} by an application on $-u$ and $u$.
\end{proof}

\begin{lemma}[Perturbation of supersolutions]\label{perturbation lemma}
Let $\Omega\subset U\subset \R^N$ be open and bounded sets, $f\in L^{\frac{p}{p-1}}(\Omega)$, and let $u\in V(\Omega,\R^N)\cap L^{\infty}(\R^N)$ be a supersolution of $L_{\Delta_p}u=f$ in $\Omega$ such that $\supp\,u\subset U$. Let $\psi\in C^{\infty}_c(\Omega)$. Then, there is $C>0$ such that $u+\psi\in V(\Omega,\R^N)$ satisfies weakly
$$
L_{\Delta_p}(u+\psi)\geq f-C\max\{\|\psi\|_{L^{\infty}(\Omega)}^{p-1},\|\psi\|_{L^{\infty}(\Omega)}\}\quad\text{in $\Omega$}.
$$
\end{lemma}
\begin{proof}
Clearly, $u+\psi\in V(\Omega,\R^N)$ and also $L_{\Delta_p}\psi\in L^{\infty}(\Omega)$ by Lemma \ref{some properties 2}. Let $v\in X^p_0(\Omega)$, $v\geq 0$. Then, by Proposition \ref{alternative2}
\begin{align}
&\cE_{L,p}(u+\psi,v)-\int_{\Omega}fv\,dx\notag\\
&\geq \cE_{L,p}(u+\psi,v)-\cE_{L,p}(u,v)\notag\\
\begin{split}\label{eq:perturbation}
&= \frac{C_{N,p}}{2}\iint_{\Omega\times \Omega}\frac{\Big(g(u(x)-u(y)+\psi(x)-\psi(y))-g(u(x)-u(y))\Big)(v(x)-v(y))}{|x-y|^N}\,dxdy\\
&\quad +\int_{\Omega}(\rho_{N}(p)+h_{\Omega}(x))\Big(g(u(x)+\psi(x))-g(u(x))\Big)v(x)\,dx\\
&\quad+ C_{N,p}\int_{\Omega}v(x)\int_{\R^N\setminus \Omega}\frac{g(u(x)+\psi(x)-u(y))-g(u(x)+\psi(x))-g(u(x)-u(y))+g(u(x))}{|x-y|^N}\,dydx.
\end{split}
\end{align}
In the following, let 
$$
P:=\max\{\|\psi\|_{L^{\infty}(\Omega)},\|\psi\|_{L^{\infty}(\Omega)}^{p-1}\}.
$$
For the first integral in \eqref{eq:perturbation}, let $Q(a,b)$ for $a,b\in \R$ be defined by
$$
g(a)-g(b)=Q(a,b)(a-b),\quad\text{that is,}\quad Q(a,b)=(p-1)\int_0^1|b+t(a-b)|^{p-2}\,dt.
$$
Then, with $a=a(x,y)=u(x)-u(y)+\psi(x)-\psi(y)$ and $b=b(x,y)=u(x)-u(y)$, we have $Q(a(x,y),b(x,y))=Q(a(y,x),b(y,x))$, and thus
\begin{align*}
&\iint_{\Omega\times \Omega}\frac{\Big(g(u(x)-u(y)+\psi(x)-\psi(y))-g(u(x)-u(y))\Big)(v(x)-v(y))}{|x-y|^N}\,dxdy\\
&=\iint_{\Omega\times \Omega}\frac{Q(a,b)(\psi(x)-\psi(y))(v(x)-v(y))}{|x-y|^N}\,dxdy\\
&=2\int_{\Omega} v(x)\int_{\Omega}\frac{Q(a,b)(\psi(x)-\psi(y))}{|x-y|^N}\,dydx\\
&=2\int_{\Omega} v(x)\int_{\Omega}\frac{g(u(x)-u(y)+\psi(x)-\psi(y))-g(u(x)-u(y))}{|x-y|^N}\,dydx.
\end{align*}
Then by Lemma \ref{estimate-general-3}, using the boundedness of $u$, there is $c=c_p>0$ such that
$$
g(u(x)+\psi(x))-g(u(x))\geq -cP\quad\text{for all $x\in \Omega$},
$$
and for $x,y\in \Omega$
\begin{align*}
g(u(x)-u(y)+\psi(x)-\psi(y))-g(u(x)-u(y))
&\geq -c\max\{|\psi(x)-\psi(y)|,|\psi(x)-\psi(y)|^{p-1}\}\\
&\geq -c\tilde{c}P \max\{|x-y|,|x-y|^{p-1}\},
\end{align*}
for some constant $\tilde{c}$ depending on $\psi$. With this, and the previous estimate we have
\begin{align*}
\frac{C_{N,p}}{2}&\iint_{\Omega\times \Omega}\frac{\Big(g(u(x)-u(y)+\psi(x)-\psi(y))-g(u(x)-u(y))\Big)(v(x)-v(y))}{|x-y|^N}\,dxdy\\
&\geq -C_{N,p}c\tilde{c}P\int_{\Omega} v(x)\int_{\Omega} \max\{|x-y|^{1-N},|x-y|^{p-1-N}\}\,dydx\geq -C_1P\int_{\Omega} v(x)\,dx.
\end{align*}
Moreover, $h_{\Omega}$ is bounded in $\supp\,\psi$, since $\psi$ is compactly supported in $\Omega$, and thus also
\begin{align*}
 &\int_{\Omega}(\rho_{N}(p)+h_{\Omega}(x))\Big(g(u(x)+\psi(x))-g(u(x))\Big)v(x)\,dx\\
 &\geq -c\sup_{x\in \supp\,\psi}|\rho_{N}(p)+h_{\Omega}(x)|P\int_{\Omega}v(x)\,dx=-C_2P\int_{\Omega}v(x)\,dx,
\end{align*}
which bounds the second integral in \eqref{eq:perturbation}. To bound the last integral in \eqref{eq:perturbation}, first note that 
\begin{align*}
\int_{\Omega}&v(x)\int_{\R^N\setminus \Omega}\frac{g(u(x)+\psi(x)-u(y))-g(u(x)+\psi(x))-g(u(x)-u(y))+g(u(x))}{|x-y|^N}\,dydx
\\
&=\int_{\supp\,\psi}v(x)\int_{U\setminus \Omega}\frac{g(u(x)+\psi(x)-u(y))-g(u(x)+\psi(x))-g(u(x)-u(y))+g(u(x))}{|x-y|^N}\,dydx\\
&\geq -c'\int_{\Omega}v(x)\int_{U\setminus \Omega} \Big|g(u(x)+\psi(x)-u(y))-g(u(x)+\psi(x))-g(u(x)-u(y))+g(u(x))\Big| \,dydx,
\end{align*}
for a constant $c'$ depending on $N$ and $\psi$. Again, with Lemma \ref{estimate-general-3} we have 
$$
\Big|g(u(x)+\psi(x)-u(y))-g(u(x)+\psi(x))-g(u(x)-u(y))+g(u(x))\Big|\geq -2cP,
$$
and thus
\begin{align*}
C_{N,p}\int_{\Omega}&v(x)\int_{\R^N\setminus \Omega}\frac{g(u(x)+\psi(x)-u(y))-g(u(x)+\psi(x))-g(u(x)-u(y))+g(u(x))}{|x-y|^N}\,dydx\\
&\geq -C_3P\int_{\Omega}v(x)\,dx,
\end{align*}
for some constant $C_3>0$. The claim thus follows with $C=C_1+C_2+C_3$.
\end{proof}

\subsection{Maximum and comparison principles}
\label{sec:maximum principle}

\begin{lemma}\label{the negative part}
Let $\Omega\subset \R^N$ be an open bounded set and let $c\in L^{\infty}(\Omega)$. Let $u\in V(\Omega,\R^N)$ be a supersolution of $L_{\Delta_p}u=c(x)g(u)$ in $\Omega$ with $u\geq 0$ in $\R^N\setminus \Omega$. Then, $u^-\in X^p_0(\Omega)$ and 
$$
\cE_{L,p}(u^-,u^-)\leq \|c^+\|_{L^{\infty}(\Omega)}\int_{\Omega}|u^-(x)|^p\,dx.
$$
\end{lemma}
\begin{proof}
It is easy to check that $u^-\in V(\Omega,\R^N)$. Thus $u^-\in X^p_0(\Omega)$, see Remark \ref{cut off2}. Testing with $u^-$ implies
\begin{align*}
-\|c^+\|_{L^{\infty}(\Omega)}\int_{\Omega}|u^-(x)|^p\,dx&\leq \int_{\Omega} c(x)g(u(x))u^-(x)\,dx\leq \cE_{L,p}(u,u^-)\\
&=\cE_p(u,u^-)+\cF_p(u,u^-)-\rho_{N}(p)\int_{\R^N}|u^-(x)|^p\,dx.
\end{align*}
Writing for $a,b\in \R$
$$
g(a-b)+g(b)=g(a-b)-g(-b)=(p-1)\int_0^1|b+ta|\,dta=:Q(a,b)a,
$$
we find with $a=a(x,y):=u^+(x)-u^+(y)$ and $b=b(x,y):=u^-(x)-u^-(y)$
\begin{align*}
\cF_p(u,u^-)&=-\cF_p(u^-,u^-)\\
&\quad +\frac12\iint_{\R^N\times\rn}\Big(g(u(x)-u(y))+g(u^-(x)-u^-(y))\Big)(u^-(x)-u^-(y))\mathbf{j}(x-y)\,dxdy\\
&=-\cF_p(u^-,u^-)\\
&\quad +\frac12\iint_{\R^N\times\rn}Q(a(x,y),b(x,y))(u^+(x)-u^+(y))(u^-(x)-u^-(y))\mathbf{j}(x-y)\,dxdy,
\end{align*}
and
\begin{align*}
\cE_p(u,u^-)&=-\cE_p(u^-,u^-)\\
&\quad +\frac12\iint_{\R^N\times\rn}\Big(g(u(x)-u(y))+g(u^-(x)-u^-(y))\Big)(u^-(x)-u^-(y))\mathbf{k}(x-y)\,dxdy\\
&=-\cE_p(u^-,u^-)\\
&\quad +\frac12\iint_{\R^N\times\rn}Q(a(x,y),b(x,y))(u^+(x)-u^+(y))(u^-(x)-u^-(y))\mathbf{k}(x-y)\,dxdy.
\end{align*}
Since  $Q\geq0$ and
$$
(u^+(x)-u^+(y))(u^-(x)-u^-(y))=-u^+(x)u^-(y)-u^+(y)u^-(x)\leq 0,
$$
it follows 
$$
-\|c^+\|_{L^{\infty}(\Omega)}\int_{\Omega}|u^-(x)|^p\,dx\leq \cE_{L,p}(u,u^-)\leq -\cE_{L,p}(u^-,u^-),
$$
as claimed.
\end{proof}

\begin{lemma}\label{maximum principle-pre}
Let $\Omega\subset \R^N$ be an open bounded set and $c\in L^{\infty}(\Omega)$. If
$$
\lambda_{L,p}^1(\Omega):=\inf\left\{\cE_{L,p}(u,u):u\in X^p_0(\Omega)\text{ and }\|u\|_{L^p(\Omega)}=1\right\}>\|c^+\|_{L^{\infty}(\Omega)},
$$ 
then $L_{\Delta_p}-c(x)$ satisfies the maximum principle in $\Omega$. Here, we say $L_{\Delta_p}-c(x)$ satisfies the maximum principle in $\Omega$, if for all supersolutions $v\in V(\Omega,\R^N)$ of $L_{\Delta_p}v=c(x)g(v)$ in $\Omega$ and with $v\geq 0$ in $\R^N\setminus \Omega$ it follows that $v\geq 0$ (a.e.) in $\R^N$.
\end{lemma}
\begin{proof}
Let $v\in V(\Omega,\R^N)$ be a supersolution of $L_{\Delta_p}v=0$ in $\Omega$ with $v\geq 0$ in $\R^N\setminus \Omega$. Then, Lemma \ref{the negative part} implies
$$
\lambda_{L,p}^1(\Omega)\|v^-\|_{L^p(\Omega)}^p\leq \cE_{L,p}(v^-,v^-)\leq \|c^+\|_{L^{\infty}(\Omega)}\|v^-\|_{L^p(\Omega)}^p.
$$
Thus, $v^-=0$ a.e. in $\Omega$ and the claim follows.
\end{proof}

$\lambda_{L,p}^1(\Omega)$ is called the first eigenvalue, which we investigate in detail in the next section. For our investigation of the maximum principle, however, we need the following properties of $\lambda_{L,p}^1(\Omega)$.

\begin{prop}\label{prop of lambda1}
Let $\Omega\subset \R^N$ be an open bounded set and $p>1$. Then, the following properties hold for the first eigenvalue $\lambda_{L,p}^1(\Omega)$.
\begin{enumerate}
\item There is a nonnegative $L^p$-normalized function $u\in X^p_0(\Omega)$ such that $\lambda_{L,p}^1(\Omega)=\cE_{L,p}(u,u)$. Moreover, $u$ satisfies in weak sense
$$
L_{\Delta_p}u=\lambda_{L,p}^1(\Omega)g(u)\quad\text{in $\Omega$},\quad u=0\quad\text{in $\R^N\setminus \Omega$}.
$$
 \item $\lambda^{1}_{L,p}(U)\geq \lambda^{1}_{L,p}(\Omega)$ for $U\subset\Omega$.
    \item $\lambda_{L,p}^1(r\Omega)=\lambda_{L,p}^1(\Omega)-p\ln(r)$ for $r>0$. 
\end{enumerate}
\end{prop}
\begin{proof}
Let $\{u_n\}\subset X^p_0(\Omega)$ be a minimizing sequence for $\lambda^1_{L,p}(\Omega)$ such that 
$$
\|u_n\|_{L^p(\Omega)}=1\text{ for all }n,\,\, \text{ and }\,\,\cE_{L,p}(u_n,u_n)\to\lambda^1_{L,p}(\Omega)\,\,\text{ as }n\to\infty.
$$
Note that,
\begin{equation*}
\begin{split}
\cF_p(u_n,u_n)&:=\frac{C_{N,p}}{2}\iint_{|x-y|>1}\frac{|u_n(x)-u_n(y)|^p-|u_n(x)|^p-|u_n(y)|^p}{|x-y|^N}dxdy
\\
&=\frac{C_{N,p}}{2}\iint_{\substack{x,y\in\Omega\\|x-y|>1}}\frac{|u_n(x)-u_n(y)|^p-|u_n(x)|^p-|u_n(y)|^p}{|x-y|^N}dxdy,
\end{split}
\end{equation*}
Since, $\|u_n\|_{L^p(\Omega)}=1$ then from above it follows that 
$$
|\cF_p(u_n,u_n)|\leq(2^{p-1}+1)C_{N,p}\int_{\Omega}|u_n(x)|^p\int_{B_R(x)\setminus B_1(x)}|x-y|^{-N}\,dydx=(2^{p-1}+1)C_{N,p}\omega_N\ln(R),
$$
where $R>\text{diam}(\Omega)$. This implies 
$$
C:=\sup_{n}|\cF_p(u_n,u_n)|<\infty.
$$
Thus, for any $n$ we have
$$
|\cE_p(u_n,u_n)|\leq|\cE_{L,p}(u_n,u_n)|+C+\rho_N(p)<\infty.
$$
Therefore, we conclude that $\{u_n\}$ is a bounded sequence in $X^p_0(\Omega)$. Hence, by the reflexivity property of $X^p_0(\Omega)$, we get up to a subsequence $u_n \rightharpoonup u$ in $X^p_0(\Omega)$ as $n\to\infty$ and the compact embedding $X^p_0(\Omega)\hookrightarrow L^p(\Omega)$, we get $u_n\to u$ in $L^p(\Omega)$ as $n\to\infty$. This gives that $\|u\|_{L^p(\Omega)}=1.$ By the generalized dominated convergence theorem, we have 
$$
\lim_{n\to\infty}\cF_p(u_n,u_n)=\cF_p(u,u),
$$
and by the weak lower semincontinuity of the norm, we obtain
$$
[u_n]_{X^p_0(\Omega)}^p=\cE_p(u,u)\leq\liminf_{n\to\infty}\cE_p(u_n,u_n).
$$
Thus, we have
$$
\lambda^1_{L,p}(\Omega)\leq\cE_{L,p}(u,u)\leq\liminf_{n\to\infty}\cE_{L,p}(u_n,u_n)=\lambda^1_{L,p}(\Omega).
$$
Hence, $\lambda^1_{L,p}(\Omega)$ is achieved by a function $u\in X^p_0(\Omega)$ with $\|u\|_{L^p(\Omega)}=1$. Note that by Lemma \ref{lower bound absolut value}, we have
\begin{align*}
\lambda_{L,p}^1(\Omega)&=\cE_{L,p}(u,u)\geq \cE_{L,p}(|u|,|u|)\geq \lambda_{L,p}^1(\Omega),
\end{align*}
implying $u\geq 0$ in $\Omega$, since otherwise the above inequality would be strict. Next, we show that the minimizer $u$ satisfies weakly the claimed associated equation. For this, consider the following function 
$$\phi(x,t)=u(x)+tv(x),\,\,v\in X^p_0(\Omega),\,t>0.$$
Clearly, $\phi(\cdot,t)\in X^p_0(\Omega)$. Since $u$ is the minimizer for $\lambda^1_{L,p}$, we then have
$$
\frac{d}{dt}\left\{\frac{\cE_{L,p}(\phi(\cdot,t),\phi(\cdot,t))}{\displaystyle\int_{\Omega}|\phi(x,t)|^p\,dx}\right\}=0\,\text{ at }\,t=0.
$$
This gives that 
$$
\cE_{p}(u,v)+\cF_p(u,v)+\rho_{N}(p)\int_{\rz^N}g(u)(x)v(x)\,dx=\lambda^1_{L,p}\int_{\Omega}|u|^{p-2}uv\,dx.
$$
and the (1) follows. The second property follows immediately from the definition of the first eigenvalue. For the third statement, let $u\in X^p_0(\Omega)$ be given by (1) and let $v_r:\R^N\to\R$ be given by $v_r(x)=r^{-N/p}u(x/r)$. Then $v_r$ is nonnegative, $L^p$-normalized and by Lemma \ref{weak scaling}, we have $v_r\in X^p_0(r\Omega)$, and in weak sense, for $x\in r\Omega$
 \begin{align*}
 L_{\Delta_p}v_r(x)&=r^{-\frac{N}{p}(p-1)}\Big(\lambda_{L,p}^1(\Omega)-p\ln(r)\Big)(u(x/r))^{p-1}\\  
 &=\Big(\lambda_{L,p}^1(\Omega)-p\ln(r)\Big)(v_r(x))^{p-1}.
 \end{align*}
 The claim follows.
\end{proof}

\begin{thm}[Strong maximum principle]\label{smp}
Let $\Omega\subset \R^N$ be an open set and let $c\in L^{\infty}(\Omega)$. Let $u\in V(\Omega,\R^N)\cap L^{\infty}(\R^N)$ be a nonnegative supersolution of $L_{\Delta_p}u=c(x)g(u)$ in $\Omega$, such that that $\supp\,u$ is compact in $\R^N$. Then, either $u=0$ in $\Omega$ or $u>0$ in $\Omega$ in the sense that
$$
\essinf_{K}u>0\quad\text{for all compact sets $K\subset \Omega$}.
$$
\end{thm}
\begin{proof}
Assume $u\not\equiv 0$ in $\Omega$. Then, there is compact set $K\subset \Omega$ with $|K|>0$ and such that 
$$
\essinf_{K}u=\delta>0.
$$
Let $x_0\in \Omega\setminus K$, $r>0$ such that $B:=B_r(x_0)\subset \Omega\setminus K$. Let $f\in C^{\infty}_c(\R^N)$ be a nonnegative function with $\supp f\subset B$, $0\leq f\leq 1$, and $f\equiv 1$ in $B_{r/2}(x_0)$. In the following, we may suppose that $\lambda_{L,p}^1(B)>\|c^+\|_{L^{\infty}(\Omega)}$ by making $r$ smaller, if necessary, and applying Proposition \ref{prop of lambda1}. Then, by Lemma \ref{maximum principle-pre}, $L_{\Delta_p}-c(x)$ satisfies the maximum principle in $B$. Consider next the function
$$
w_a:=u_a-\delta 1_K\quad\text{with}\quad u_a:=u-\frac{1}{a}f,
$$
for $a>0$. Then $w_a\in V(B,\R^N)$,  $w_a\geq 0$ in $\R^N\setminus B$ by construction, and we claim that there is $a>0$ such that $w_a$ is a supersolution of $L_{\Delta_p}w=0$ in $B$. Indeed, let $\phi\in X^p_0(B)$, $\phi\geq 0$. 
Then, by Proposition \ref{alternative2}
\begin{align*}
\cE_{L,p}(w_a,\phi)&=\frac{C_{N,p}}{2}\iint_{B\times B}\frac{g(w_a(x)-w_a(y))(\phi(x)-\phi(y))}{|x-y|^N}\,dxdy+\int_{B}\Big(\rho_N(p)+h_{\Omega}(x)\Big)g(w_a(x))\phi(x)\,dx\\
&\qquad +C_{N,p}\int_{B}\phi(x)\int_{\R^N\setminus B}\frac{g(w_a(x)-w_a(y))-g(w_a(x))}{|x-y|^N}\,dydx\\
&=\frac{C_{N,p}}{2}\iint_{B\times B}\frac{g(u_a(x)-u_a(y))(\phi(x)-\phi(y))}{|x-y|^N}\,dxdy+\int_{B}\Big(\rho_N(p)+h_{\Omega}(x)\Big)g(u_a(x))\phi(x)\,dx\\
&\qquad+C_{N,p}\int_{B}\phi(x)\int_{\R^N\setminus B}\frac{g(u_a(x)-u_a(y)+\delta 1_K(y))-g(u_a(x))}{|x-y|^N}\,dydx.
\end{align*}
First, we consider the case $p\geq 2$. Then, by Lemma \ref{estimate-general-2} 
$$
g(u_a(x)-u_a(y)+\delta1_K(y))-g(u_a(x)-u_a(y))\geq c\delta^{p-1}1_K(y)\quad\text{for all $x\in B$, $y\in \R^N\setminus B$},
$$
where $c=2^{2-p}$. Thus, with Lemma \ref{perturbation lemma}, we have for some $C>0$
\begin{align*}
\cE_{L,p}(w_a,\phi)&=\cE_{L,p}(u_a,\phi)+\int_B\phi(x)\int_{K}\frac{c\delta^{p-1} C_{N,p}}{|x-y|^{N}}\,dy\,dx\\
&\geq \int_{\Omega}\phi(x)\left(\int_{K}\frac{c\delta^{p-1} C_{N,p}}{|x-y|^{N}}\,dy-\frac{C}{\min\{a,a^{p-1}\}}\right)\,dx.
\end{align*} 
Since $|K|>0$, we may thus choose $a>0$ large enough such that
$$
\inf_{x\in \Omega}\int_{K}\frac{c\delta^{p-1} C_{N,p}}{|x-y|^{N}}\,dy-\frac{C}{\min\{a,a^{p-1}\}}\geq \|c^+\|_{L^{\infty}(\Omega)}\|u+f\|_{L^{\infty}(\Omega)}^{p-1}.
$$
It thus follows that in weak sense
$$
L_{\Delta_p}w_a\geq c(x)g(w_a)\quad\text{in $B$}, \qquad w_a\geq 0\quad\text{in $\R^N\setminus B$.}
$$
By Lemma \ref{maximum principle-pre}, it follows that $w_a\geq 0$ a.e. in $B$ and thus, in particular, $u\geq \frac1{a}$ a.e. in $B_{r/2}(x_0)$. The claim follows.\\
For the case $p\in(1,2)$, we use Lemma \ref{estimate-general-3} with $M=2(\|u\|_{L^{\infty}(\R^N)}+\|f\|_{L^{\infty}(\R^N)})$. Then, for some $c=c(M,p)>0$
$$
g(u_a(x)-u_a(y)+\delta1_K(y))-g(u_a(x)-u_a(y))\geq c\delta1_K(y)\quad\text{for all $x\in B$, $y\in \R^N\setminus B$}.
$$
Similar to the case $p\geq 2$, we find 
\begin{align*}
\cE_{L,p}(w_a,\phi)&\geq \int_{\Omega}\phi(x)\Big(\int_{K}\frac{c\delta C_{N,p}}{|x-y|^{N}}\,dy-\frac{C}{a^{p-1}}\Big)\,dx.
\end{align*} 
Choosing $a>0$ large such that
$$
\inf_{x\in \Omega}\int_{K}\frac{c\delta  C_{N,p}}{|x-y|^{N}}\,dy-\frac{C}{a^{p-1}}\geq 0.
$$
The claim follows analogously. 
\end{proof}

\begin{lemma}[Weak comparison principle]\label{wcp}
Let $\Omega\subset U\subset \R^N$ be an open bounded sets and let $c\in L^{\infty}(\Omega)$. Suppose 
$$
c(x)\leq \rho_{N}(p)+h_{U}(x)\quad\text{for a.e. $x\in \Omega$}.
$$
Let $u,v\in V(\Omega,\R^N)$ be such that in weak sense
$$
L_{\Delta_p}u- c(x)g(u)\geq L_{\Delta_p}v- c(x)g(v)\quad\text{in $\Omega$ with}\quad u\geq v\quad\text{in $\R^N\setminus \Omega$},
$$
and $u,v=0$ on $\R^N\setminus U$. Then $u\geq v$ a.e. in $\R^N$.
\end{lemma}
\begin{proof}
First, note that $\phi=(u-v)^-$ belongs to $X^p_0(\Omega)\subset X^p_0(U)$. Moreover, we have with Lemma \ref{alternative2}
\begin{align*}
0&\leq \cE_{L,p}(u,\phi)-\cE_{L,p}(v,\phi)-\int_{\Omega}c(x)\Big(g(u(x))-g(v(x))\Big)\phi(x)\,dx\\
&=\frac{C_{N,p}}{2}\iint_{U\times U}\frac{\Big(g(u(x)-u(y))-g(v(x)-v(y))\Big)(\phi(x)-\phi(y))}{|x-y|^N}\,dxdy\\
&\quad+\int_{\Omega}\Big(\rho_{N}(p)+h_{U}(x)-c(x)\Big)\Big(g(u(x))-g(v(x))\Big)\phi(x)\,dx.
\end{align*}
In the following, note that for $a,b\in \R$ we have
$$
g(a)-g(b)=Q(a,b)(a-b)\quad\text{with}\quad Q(a,b):=(p-1)\int_0^1|b+t(a-b)|^{p-2}\,dt\geq0.
$$
Note that $Q(a,b)=Q(b,a)$, by the substitution $t=1-\tau$, and we have $Q(-a,-b)=Q(a,b)$ as also $Q(a,-b)=Q(-a,b)$. With this notation and putting $w:=u-v$, we have
\begin{align*}
0&\leq \cE_{L,p}(u,\phi)-\cE_{L,p}(v,\phi)-\int_{\Omega}c(x)\Big(g(u(x))-g(v(x))\Big)\phi(x)\,dx\\
&=\frac{C_{N,p}}{2}\iint_{U\times U}\frac{Q(u(x)-u(y),v(x)-v(y))(w(x)-w(y))(\phi(x)-\phi(y))}{|x-y|^N}\,dxdy\\
&\quad +\int_{\Omega}\Big(\rho_{N}(p)+h_{U}(x)-c(x)\Big)Q(u(x),v(x))w(x)\phi(x)\,dx\leq0,
\end{align*}
where we used $w(x)\phi(x)=-\phi^2(x)\leq 0$ and
$$
(w(x)-w(y))(\phi(x)-\phi(y))=-(\phi(x)-\phi(y))^2-w^+(x)\phi(y)-w^+(y)\phi(x)\leq 0.
$$
This entails that the above integrands are nonpositive, while the integrals are zero. Thus $\phi\equiv 0$.
\end{proof}

\begin{prop}[Strong comparison principle]\label{scp}
  Let $\Omega\subset U\subset \R^N$ be an open bounded sets and let $c\in L^{\infty}(\Omega)$. Suppose 
$$
c(x)\leq \rho_{N}(p)+h_{U}(x)\quad\text{for a.e. $x\in \Omega$}.
$$
Let $u,v\in V(\Omega,\R^N)$ be such that
$$
L_{\Delta_p}u- c(x)g(u)\geq L_{\Delta_p}v- c(x)g(v)\quad\text{in $\Omega$ with}\quad u\geq v\quad\text{in $\R^N\setminus \Omega$},
$$
and $u,v=0$ on $\R^N\setminus U$. Suppose further that either $u\in L^{\infty}(\Omega)$ or $v\in L^{\infty}(\Omega)$. Then, either $u\equiv v$ in $\Omega$ or $u>v$ in $\Omega$ in the sense that 
$$
\essinf_K(u-v)>0\quad\text{for all compact sets $K\subset \Omega$}.
$$
\end{prop}
\begin{proof}
Without loss of generality, we may assume $u\in L^{\infty}(\Omega)$. Assume further $u\not \equiv v$ in $\Omega$. Let $K\subset \Omega$ be a compact set with positive measure and such that 
$$
\essinf_K (u-v)=\delta>0.
$$
Proceeding as in the proof of Theorem \ref{smp}, we fix some ball $B\subset \Omega\setminus K$ and a nonnegative function $f\in C^{\infty}_c(B)$ and replace $u$ with the function $w_a=u-\frac{1}{a}f+\delta 1_K$. Then we can find $a>0$ such that
$$
L_{\Delta_p}w_a\geq c(x)w_a\quad\text{in $B$}.
$$
Then $w_a\geq v$ in $\R^N\setminus B$ and $w_a$ and $v$ are respectively super- and subsolution of $L_{\Delta_p}w=c(x)w$ in $B$. By Lemma \ref{wcp}, the claim follows.
\end{proof}

\begin{proof}[Proof of Theorem \ref{intro:comparison principle}]
This follows immediately from Proposition \ref{scp}.
\end{proof}

\begin{remark}\label{discussion on scp}
As the comparison principle for nonlinear problems involving the $p$-Laplace or the fractional $p$-Laplace are quite interesting and more involved than in the linear case (see e.g. \cite{PS04,RS07,J18,IMP23,CFGQ23}), let us state some remarks concerning the weak and strong comparison principle stated in Lemma \ref{wcp} and Proposition \ref{scp}.
\begin{enumerate}
    \item Both statements are in particular of interest for $U=\Omega$. Note that $h_{U}$ is not necessarily positive and might be negative for some $x\in \Omega$ if $U$ is large.
    \item It is tempting to only assume $u\geq v$ in $\R^N\setminus \Omega$ instead of assuming $u=v=0$ in $\R^N\setminus U$. In view of Lemma \ref{restriction}, however, this is quite delicate as the the values in the exterior have an influence on the interior. In the particular case, where $u\geq 0\geq v$ in $\R^N\setminus U$, however, it holds
    $$
    -\int_{\R^N\setminus U}\frac{g(u(x)-u(y))-g(u(x))}{|x-y|^N}\,dy\geq 0\geq  -\int_{\R^N\setminus U}\frac{g(v(x)-v(y))-g(v(x))}{|x-y|^N}\,dy
    $$
    for any $x\in \Omega$, using the monotonicity of $g$. Thus, Lemma \ref{wcp} easily also holds if one assumes $u\geq 0\geq v$ in $\R^N\setminus U$ in place of $u=v=0$ in $\R^N\setminus U$. An analogous assumption can be used in Proposition \ref{scp}.
\end{enumerate} 
\end{remark}

\section{The Dirichlet eigenvalue problem}\label{eigenfunction}

Consider the following nonlinear Dirichlet eigenvalue problem on a bounded open set $\Omega$ in $\rz^N$:
\begin{equation}\label{ev problem}
  \begin{cases}
      L_{\Delta_p}u=\lambda|u|^{p-2}u\,\,\text{ in }\Omega,\\
      u=0\,\,\text{ in }\rz^N\setminus\Omega. 
  \end{cases}
\end{equation}
By Corollary \ref{bounded2}, we already know that any solution of \eqref{ev problem} is bounded. Moreover, we have collected some simple preliminary properties of $\lambda^1_{L,p}(\Omega)$ as defined in Lemma \ref{maximum principle-pre}. We emphasize that the validity of maximum principles is strongly entwined with the positivity of $\lambda^1_{L,p}(\Omega)$ as discussed in the previous section. Here, we will now investigate in detail properties of the first eigenvalue and the corresponding first eigenfunction and their relation to the respective parts for $s>0$. We start with the following.

\begin{thm}\label{thm:first eigenfunction main}
    Let $\Omega$ be a bounded open set in $\rz^N$ and $1<p<\infty$. Consider the following minimization problem
    \begin{equation}\label{first ev def}
       \lambda^1_{L,p}:=\lambda^1_{L,p}(\Omega)=\inf\left\{\cE_{L,p}(u,u):u\in X^p_0(\Omega)\text{ and }\|u\|_{L^p(\Omega)}=1\right\}.
    \end{equation}
Then the following hold:
\begin{enumerate}
    \item[i)] The quantity $\lambda^1_{L,p}$ is the eigenvalue and the extremal function $u$ of \eqref{first ev def} is the eigenfunction of \eqref{ev problem} corresponding to $\lambda^1_{L,p}$.
    \smallskip
    \item[ii)] The eigenfunction $u$ corresponding to $\lambda^1_{L,p}$ is strictly positive in $\Omega$. Moreover, $\lambda^1_{L,p}$ is simple in the sense that if $u,\,v\in X^p_0(\Omega)$ are the two eigenfunctions corresponding to $\lambda^1_{L,p}$ then $u=c\,v$ for some $c\in\re.$ 
\end{enumerate}
\end{thm}
\begin{proof}
The first part follows immediately from Lemma \ref{prop of lambda1}(i). For (ii), let $u\in X^p_0(\Omega)$ be any $L^p$-normalized minimizer for $\lambda_{L,p}^1$. Then, similarly to the proof of Lemma \ref{prop of lambda1}(i), $u$ satisfies 
\begin{align*}
\lambda^1_{L,p}&=\cE_{L,p}(|u|,|u|)=\cE_{L,p}(u,u)\quad\text{and, in weak sense,}\\
L_{\Delta_p}u&=\lambda^1_{L,p}g(u)\quad\text{in $\Omega$,}\quad u=0\quad\text{in $\R^N\setminus \Omega$.}
\end{align*}
Thus, $u$ can be assumed to be nonnegative (see also Lemma \ref{lower bound absolut value}). But then $u>0$ in $\Omega$ by Theorem \ref{smp}, using that $u\in L^{\infty}(\Omega)$ by Corollary \ref{bounded2} (for the case $p\in(1,2)$). Hence we have:
\begin{equation}\label{sign for eigenfunction}
\text{Any eigenfunction corresponding to $\lambda^1_{L,p}$ is either positive or negative in $\Omega$.}
\end{equation}
Suppose next, $u,\,v\in X^p_0(\Omega)$ are the eigenfunctions of \eqref{ev problem} corresponding to the eigenvalue $\lambda^1_{L,p}$ then we may assume $u,\,v>0$ in $\Omega$ by \eqref{sign for eigenfunction}. For each $n\in\N$, define $\phi_n=\frac{u^p}{v_n^{p-1}}$, where $v_n=v+1/n$ and $\phi=\frac{u^p}{v^{p-1}}$. By Corollary \ref{bounded2}, we have the eigenfunction $u\in L^{\infty}(\Omega)$ then it is easy to see that $\phi_n\in X^p_0(\Omega)$ for all $n.$
Then by discrete Picone's inequality (see \cite[Proposition~4.2]{BrFr}), we have 
$$
0\leq|u(x)-u(y)|^p-g(v_n(x)-v_n(y))(\phi_n(x)-\phi_n(y))=:L(u,v_n)(x,y),
$$
and consequently this yields
\begin{equation*}
\begin{split}
0&\leq\frac{C_{N,p}}{2}\iint_{\Omega\times\Omega}\frac{L(u,v_n)(x,y)}{|x-y|^N}dxdy
\\
&=\frac{C_{N,p}}{2}\iint_{\Omega\times\Omega}\frac{|u(x)-u(y)|^p}{|x-y|^N}dxdy-\frac{C_{N,p}}{2}\iint_{\Omega\times\Omega}\frac{g(v(x)-v(y))}{|x-y|^N}(\phi_n(x)-\phi_n(y))dxdy
\\
&=\cE_{L,p}(u,u)-\int_{\Omega}\left(h_{\Omega}(x)+\rho_{N}(p)\right)|u(x)|^pdx-\cE_{L,p}(v,\phi_n)+\int_{\Omega}\left(h_{\Omega}(x)+\rho_{N}(p)\right)g(v(x))\phi_n(x)dx
\\
&=\lambda^1_{L,p}\int_{\Omega}|u(x)|^pdx-\lambda^1_{L,p}\int_{\Omega}g(v(x))\phi_n(x)dx+\int_{\Omega}\left(h_{\Omega}(x)+\rho_{N}(p)\right)\left[g(v(x))\phi_n(x)-|u(x)|^p\right]dx,
\end{split}
\end{equation*}
where in the above we used Proposition \ref{alternative} and definitions of $u,v$. Then, by Fatou's lemma and the dominated convergence theorem, we obtain
$$
\iint_{\Omega\times\Omega}\frac{L(u,v)(x,y)}{|x-y|^N}dxdy=0.
$$
Therefore, we have $L(u,v)(x,y)=0$ a.e. in $\Omega\times\Omega$. Hence, again by discrete Picone's inequality we get $u=c\,v$ for some $c>0$.
\end{proof}

\smallskip

In order to state the next result, let us first recall the structure of the Dirichlet eigenvalue problem for the fractional $p$-Laplace operator. Let $\Omega$ be an open set in $\rn$ and $0<s<1,\,p\in(1,\infty)$ and recall the definition of fractional Sobolev space $\cW^{s,p}_0(\Omega)$ in Section \ref{function spaces} with zero \textit{nonlocal exterior data}. Note that by definition $\cW^{s,p}_0(\Omega)$ is a closed subspace of $W^{s,p}(\rn)$ and if $\Omega\subset\rn$ is a bounded set with Lipschitz boundary, then $C_c^\infty(\Omega)$ is a dense subset of $\cW^{s,p}_0(\Omega)$. A non-zero function $u\in \cW^{s,p}_0(\Omega)$ is called a weak solution of the nonlocal Dirichlet eigenvalue problem
\begin{equation}\label{frac p ev problem}
        (-\Delta_p)^su=\lambda|u|^{p-2}u\quad\text{ in }\Omega,\quad
        u=0\quad \text{ in }\rn\setminus\Omega,
\end{equation}
if for all $v\in \cW^{s,p}_0(\Omega)$ we have
\begin{equation*}
\cE(u,v)=\lambda\left\langle g(u),v\right\rangle:=\lambda\int_{\Omega}|u|^{p-2}u(x)v(x)dx,
\end{equation*}
where 
$$
\cE(u,v):=\left\langle(-\Delta_p)^su,v\right\rangle=\frac{C_{N,s,p}}{2}\iint_{\rn\times\rn}\frac{g(u(x)-u(y))(v(x)-v(y))}{|x-y|^{N+sp}}dxdy,
$$
and $\langle\cdot,\cdot\rangle$ denotes the duality action. Any such function $u$ is called the eigenfunction corresponding to the eigenvalue $\lambda$ of \eqref{frac p ev problem}. The first eigenvalue $\lambda^1_{s,p}(\Omega)$ of \eqref{frac p ev problem} can be characterized as \eqref{intro:lambda-s-p}.

Next, we need the following $\Gamma$-convergence type result which is useful in the proof of Theorem \ref{relation between evs of log p Lap and frac p Lap}.
\begin{lemma}\label{conv btwn functional}
Let $\Omega$ be a bounded open set in $\rn$, $1<p<\infty$. Suppose $\{u_{s_n}\}$ is a sequence in $\cW^{s_n,\,p}_0(\Omega)$ which is bounded in $X^p_0(\Omega)$.
Then, there exist $u\in X^p_0(\Omega)$ and a subsequence $\{u_{s_{n_k}}\}$ of $\{u_{s_n}\}$ such that 
$$
\lim_{k\to\infty} u_{s_{n_{k}}}=u\quad\text{in $L^p(\Omega)$},
$$
and
$$
\lim_{k\to\infty}\frac{1}{s_{n_k}}\left[\cE(u_{s_{n_k}},v)-\int_{\Omega}g(u_{s_{n_k}})v\,dx\right]=\cE_{L,p}(u,v)\,\,\,\text{ for all }v\in C_c^\infty(\Omega). 
$$
\end{lemma}

\begin{remark}
Note that, similar to \eqref{simple expansion}, the above lemma can be viewed as the asymptotic expansion of the Gagliardo seminorm at $s=0$. For other results about the asymptotic expansion of the Gagliardo seminorm for $p=2$ in the sense of $\Gamma$-convergence as $s\to0^+$ and $s\to1^-$, see \cite{CrLuKuNiPo, KuPaTr}, and for the pointwise convergence of the Gagliardo seminorm as $s\to0^+$, see \cite{MaSh}  .
\end{remark}
\begin{proof}[Proof of Lemma \ref{conv btwn functional}]
Since $\{u_{s_n}\}$ is bounded in $X^p_0(\Omega)$, by reflexivity of $X^p_0(\Omega)$ there is $u\in X^p_0(\Omega)$ such that $u_{s_n}\rightharpoonup u$ in $X^p_0(\Omega)$ for $n\to \infty$, after passing to a subsequence (which we keep on denoting by $\{u_{s_{n}}\}$). Moreover, by compact embedding and passing to a further subsequence (also still denoted by $\{u_{s_{n}}\}$), we have $u_{s_n}\to u$ in $L^p(\Omega)$ and also $u_{s_n}(x)\to u(x)$ a.e. in $\rn$. Now by definition, we have
\begin{equation}\label{derivative of bilinear}
\begin{split}
\frac{1}{s_n}\left[\cE(u_{s_n},\,v)-\int_{\Omega}g(u_{s_n})v\,dx\right]
&=\frac{C_{N,s_n,p}}{2s_n}\iint_{\rn\times\rn}\frac{g(u_{s_n}(x)-u_{s_n}(y))(v(x)-v(y))}{|x-y|^{N+s_np}}dxdy-\frac{1}{s_n}\int_{\Omega}g(u_{s_n})v\,dx
\\
&=I_{1,n}+I_{2,n}+I_{3,n},
\end{split}
\end{equation}
where with
\begin{align*}
   E_n(x,y)&:= \frac{C_{N,s_n,p}}{2s_n}\frac{g(u_{s_n}(x)-u_{s_n}(y))(v(x)-v(y))}{|x-y|^{N+s_np}}\quad\text{and}\\
   F_n(x,y)&:=\frac{C_{N,s_n,p}}{2s_n}\frac{g(u_{s_n}(x)-u_{s_n}(y))(v(x)-v(y))-g(u_{s_n}(x))v(x)-g(u_{s_n}(y))v(y)}{|x-y|^{N+s_np}},
\end{align*}
for $x,y\in \R^N$, $x\neq y$, we let
\begin{equation*}
\begin{split}
I_{1,n}&:=\iint_{ |x-y|<1}E_{n}(x,y)\,dxdy,\\
I_{2,n}&:=\iint_{|x-y|\geq1}F_n(x,y)\,dxdy,\\
 I_{3,n}&:=\frac{C_{N,s_n,p}}{s_n}\iint_{|x-y\geq1}\frac{g(u_{s_n}(x))v(x)}{|x-y|^{N+s_np}}dxdy-\frac{1}{s_n}\int_{\Omega}g(u_{s_n})v\,dx.
\end{split}
\end{equation*}
Note that, by the pointwise a.e. convergence of $u_{s_n}$ to $u$, we have for a.e. $(x,y)\in \R^N\times \R^N$
\begin{equation}\label{pointwise limits}
\begin{split}
\lim_{n\to\infty}E_n(x,y)&=E(x,y):=\frac{C_{N,p}}{2}\frac{g(u(x)-u(y))(v(x)-v(y))}{|x-y|^{N}}\quad\text{and}\\
\lim_{n\to\infty}F_n(x,y)&=F(x,y):=\frac{C_{N,p}}{2}\frac{g(u(x)-u(y))(v(x)-v(y))-g(u(x))v(x)-g(u(y))v(y)}{|x-y|^{N}}.
\end{split}
\end{equation}
\textbf{Convergence of $I_{1,n}$ :}
We prove that
\begin{equation}\label{conv of I1n}
\lim_{n\to\infty}\iint_{ |x-y|<1}E_{n}(x,y)dxdy=
\iint_{|x-y|<1}E(x,y)dxdy.
\end{equation}
Let $R>0$ such that $B_1(\Omega)\subset B_R(0)$ and note that
$$
\iint_{ |x-y|<1}E_{n}(x,y)dxdy=\iint_{\substack{B_R(0)\times B_R(0) \\|x-y|<1}}E_{n}(x,y)\,dxdy.
$$
Indeed, this holds since for points in $x,y\in \R^N\setminus B_R(0)$ we have $u_{s_n}(x)=u_{s_n}(y)=v(x)=v(y)=0$ and if $x\in B_R(0)$, $y\in \R^N\setminus B_R(0)$ (or vice versa), we can consider two cases: If $x\in \Omega$, then $|x-y|>1$ since $y\notin B_1(\Omega)$, and if $x\in B_R(0)\setminus \Omega$, then, again, $u_{s_n}(x)=u_{s_n}(y)=v(x)=v(y)=0$. This implies that only the integral in $B_R(0)\times B_R(0)$ remains. Next, let $\alpha\in \R$ such that
$$
 \frac{N(p-1)}{p}-1 <\alpha <\frac{N(p-1)}{p},
$$
then we can fix $s_1\in(0,1)$ such that 
$$
 N+\alpha p+p-Np-s_1p^2>0.
$$
Since $s_n\to0$ as $n\to\infty$, we may assume $s_n<s_1$ for all $n\in \N$. Let $x,\,y\in B_R(0)$ with $|x-y|<1$, then we have with Young's inequality
\begin{align}
|E_n(x,y)|&\leq C(N,p)|u_{s_n}(x)-u_{s_n}(y)|^{p-1}\frac{|v(x)-v(y)|}{|x-y|^{N+s_1p}}\notag\\
&\leq C(N,p)\Bigg( \frac{|u_{s_n}(x)|^p+|u_{s_n}(y)|^p}{|x-y|^{\alpha \frac{p}{p-1}}}+ \frac{|v(x)-v(y)|^p}{|x-y|^{Np+s_1p^2-\alpha p}}\Bigg),\label{estimate part 1 gagliardo conv}
\end{align}
where we used the fact $s\mapsto C_{N,s,p}$ is bounded in [0,1]. Since $v\in C_c^{\infty}(\Omega)$, there is $C>0$ such that
$$
\frac{|v(x)-v(y)|^p}{|x-y|^{Np+s_1p^2-\alpha p}}\leq C|x-y|^{\alpha p+p-Np-s_1p^2},
$$
and thus this function belongs to $L^1(B_R(0)\times B_R(0))$, since
\begin{align*}
\iint_{B_R(0)\times B_R(0)}|x-y|^{p-Np-s_1p^2}\, dxdy\leq |B_R(0)|\omega_N\int_0^{2R} t^{N+\alpha p+p-Np-s_1p^2-1}\, dt<\infty,
\end{align*}
by the choices of $s_1$ and $\alpha$. Now for the first term in \eqref{estimate part 1 gagliardo conv}, by Young's convolution inequality we have
\begin{align*}
\iint_{\substack{B_R(0)\times B_R(0) \\|x-y|<1}}\frac{|u_{s_n}(x)|^p+|u_{s_n}(y)|^p}{|x-y|^{\alpha \frac{p}{p-1}}}\,dxdy&\leq 2\int_{B_R(0)}|u_{s_n}|^p\ast |\cdot|^{-\alpha\frac{p}{p-1}}\,dx\leq 2\|u_{s_n}\|_{L^p(\Omega)}\||\cdot|^{-\alpha\frac{p}{p-1}}\|_{L^1(B_{2R}(0))}.
\end{align*}
Thus, by continuity of the convolution, using $\alpha \frac{p}{p-1}<N$, and since $u_{s_n}\to u$ in $L^p(\Omega)$, it follows that 
$$
\frac{|u_{s_n}(x)|^p+|u_{s_n}(y)|^p}{|x-y|^{\alpha \frac{p}{p-1}}}\to \frac{|u(x)|^p+|u(y)|^p}{|x-y|^{\alpha \frac{p}{p-1}}}\quad\text{in $L^1(B_R(0)\times B_{R}(0))$ for $n\to\infty$.}
$$
Thus using \eqref{pointwise limits}, \eqref{estimate part 1 gagliardo conv}, and applying the generalized dominated convergence theorem, we conclude \eqref{conv of I1n}.
\smallskip

\noindent\textbf{Convergence of $I_{2,n}$ :}
Since $u_{s_n}=0=v$ in $\rn\setminus\Omega$, then we have
\begin{equation*}
I_{2,n}:=\iint_{|x-y|\geq1}F_n(x,y)\,dxdy=\iint_{\substack{\Omega\times\Omega\\|x-y|\geq1}}F_n(x,y)\,dxdy.
\end{equation*}
We claim that
\begin{equation}\label{conv of I2n}
\lim_{n\to\infty}\iint_{\substack{\Omega\times\Omega\\|x-y|\geq1}}F_{n}(x,y)\,dxdy=
\iint_{\substack{\Omega\times\Omega\\|x-y|\geq1}}F(x,y)\,dxdy=\iint_{\substack{\R^N\times\R^N\\|x-y|\geq1}}F(x,y)\,dxdy.
\end{equation}
Let $x,y\in\Omega$ with $|x-y|\geq1$, then we have with Young's inequality
\begin{equation*}
\begin{split}
|F_n(x,y)|&\leq C(N,p)\frac{|g(u_{s_n}(x)-u_{s_n}(y))(v(x)-v(y))-g(u_{s_n}(x))v(x)-g(u_{s_n}(y))v(y)|}{|x-y|^{N+s_np}}
\\
&\leq C(N,p)\Big(|u_{s_n}(x)-u_{s_n}(y)|^{p-1}|v(x)-v(y)|+|u_{s_n}(x)|^{p-1}|v(x)|+|u_{s_n}(y)|^{p-1}|v(y)|\Big)
\\
&\leq C(N,p)\Big(|u_{s_n}(x)|^p+|u_{s_n}(y)|^p\Big)\Big(|v(x)|^p+|v(y)|^p\Big),
\end{split}
\end{equation*}
using that $u_{s_n}\to u$ in $L^p(\Omega)$, the claim follows from the generalized dominated convergence theorem with \eqref{pointwise limits}.
\smallskip

\noindent\textbf{Convergence of $I_{3,n}$ :}
Note that
\begin{equation*}
\begin{split}
I_{3,n}&:=\frac{C_{N,s_n,p}}{s_n}\iint_{|x-y\geq1}\frac{g(u_{s_n}(x))v(x)}{|x-y|^{N+s_np}}dxdy-\frac{1}{s_n}\int_{\Omega}g(u_{s_n}(x))v(x)\,dx
\\
&=\left(\frac{C_{N,s_n,p}\,\omega_N}{s^2_np}-\frac{1}{s_n}\right)\int_{\Omega}g(u_{s_n}(x))v(x)dx.
\end{split}
\end{equation*}
Recall, $C_{N,s_n,p}=s_n\,d_{N,p}(s_n)$ and since $$d_{N,p}(s_n)\xrightarrow{s\to0^+}C_{N,p}=\frac{p\Gamma\left(\frac{N}{2}\right)}{2\pi^{N/2}}=\frac{p}{\omega_N}.$$
Using this, we have $$\frac{1}{s_n}\left(\frac{C_{N,s_n,p}\omega_N}{s_np}-1\right)\to\rho_{N}(p)\text{ as }s_n\to0.$$ 
Again, applying the generalized dominated convergence theorem, we obtain
$$
\int_{\Omega}g(u_{s_n})v\,dx\to\int_{\Omega}g(u)v\,dx\text{ as }n\to\infty.
$$
This implies that 
\begin{equation}\label{conv of I3n}
\lim_{n\to\infty}I_{3,n}=\rho_N(p)\int_{\Omega}g(u)v\,dx.
\end{equation}
Therefore, letting $n\to\infty$ in \eqref{derivative of bilinear} and using \eqref{conv of I1n}, \eqref{conv of I2n}, \eqref{conv of I3n}, we conclude that 
$$
\lim_{n\to\infty}\frac{1}{s_n}\left[\cE(u_{s_n},v)-\int_{\Omega}g(u_{s_n})v\,dx\right]=\cE_{L,p}(u,v). 
$$
Since the above can be done for any subsequence of $\{u_{s_n}\}$, and thus this completes the proof of the lemma.
\end{proof}

\begin{lemma}\label{limit eigenvalue to 1}
Let $\Omega$ be an open bounded Lipschitz subset of $\rn$ and $p\in(1,\infty)$. Then 
$$
\lim_{s\to0^+}\lambda^1_{s,p}(\Omega)=1.
$$
\end{lemma}
\begin{proof}
Since $\Omega$ has a Lipschitz boundary, we have from \eqref{intro:lambda-s-p}
\begin{equation*}
\lambda^1_{s,p}(\Omega)=\inf\left\{[\phi]_{W^{s,p}(\rn)}^p:\phi\in C_c^\infty(\Omega),\,\,\int_{\Omega}|\phi(x)|^p\,dx=1\right\}.
\end{equation*}
This implies
\begin{equation}\label{limsup for frac ev s to 0}
\limsup_{s\to0^+}\lambda^1_{s,p}(\Omega)\leq\limsup_{s\to0^+}[\phi]_{W^{s,p}(\rn)}^p=\limsup_{s\to0^+}\langle(-\Delta_p)^s\phi,\phi\rangle=1,  
\end{equation}
where we used the fact $(-\Delta_p)^s\phi\to g(\phi)$ for $\phi\in C_c^\infty(\Omega)$ as $s\to0^+$. To bound the limit from below, let $\phi_s\in \cW^{s,p}_0(\Omega)$ be the $L^p$-normalized eigenfunction corresponding to $\lambda^1_{s,p}(\Omega)$. 
Then
\begin{align}\label{lower bdd of frac ev}
\lambda^1_{s,p}(\Omega)&=\frac{C_{N,s,p}}{2}\iint_{\rn\times\rn}\frac{|\phi_s(x)-\phi_s(y)|^p}{|x-y|^{N+sp}}dxdy
\geq C_{N,s,p}\int_{\Omega
}|\phi_s(x)|^p\int_{\R^N\setminus\Omega}\frac{dy}{|x-y|^{N+sp}}\,dx.
\end{align}
Now for $x\in\Omega$, let $B_R(x)$ be an open ball such that $|\Omega|=|B_R(x)|$ that is $R=\left(\frac{|\Omega|}{|B_1|}\right)^{1/N}$, then following as in \cite[Lemma~6.1]{NPV12}, we obtain
\begin{equation}\label{lower bdd of usual kernel}
\int_{\R^N\setminus\Omega}\frac{dy}{|x-y|^{N+sp}}\geq\int_{\R^N\setminus B_R(x)}\frac{dy}{|x-y|^{N+sp}}=\frac{\omega_N}{sp}R^{-sp}.
\end{equation}
Plugging the estimate \eqref{lower bdd of usual kernel} into \eqref{lower bdd of frac ev}, we have
\begin{equation}\label{lower bdd of frac ev final}
\lambda^1_{s,p}(\Omega)\geq\frac{C_{N,s,p}\,\omega_N}{sp\,R^{sp}}=\frac{C_{N,s,p}\,2\pi^{\frac{N}{2}}}{sp\,R^{sp}\,\Gamma(\frac{N}{2})}.
\end{equation}
Thus, taking limit as $s\to0^+$ in \eqref{lower bdd of frac ev final} and by definition of the constant $C_{N,s,p}$ we get
\begin{equation}\label{liminf for frac ev s to 0}
\liminf_{s\to0^+}\lambda^1_{s,p}(\Omega)\geq \lim_{s\to0}\frac{C_{N,s,p}\,2\pi^{\frac{N}{2}}}{sp\,R^{sp}\,\Gamma(\frac{N}{2})}=1.
\end{equation}
Combining \eqref{limsup for frac ev s to 0} and \eqref{liminf for frac ev s to 0} to get the desired result.
\end{proof}

\begin{proof}[Proof of Theorem \ref{relation between evs of log p Lap and frac p Lap}]
We divide our proof into four steps.
\smallskip 

\noindent\textbf{Step 1:} By Lemma \ref{limit eigenvalue to 1} we obtain
$$
\lim_{s\to0^+}\lambda^1_{s,p}(\Omega)=1.
$$
Next, for $v\in C_c^\infty(\Omega)$ with $\|v\|_{L^p(\Omega)}=1$, using the second part of Theorem \ref{derivative}, we obtain
$$
\limsup_{s\to0^+}\frac{\lambda^1_{s,p}(\Omega)-1}{s}\leq\limsup_{s\to0^+}\frac{[v]^p_{W^{s,p}(\rn)}-\|v\|^p_{L^p(\Omega)}}{s}=\lim_{s\to0^+}\left\langle\frac{(-\Delta_p)^sv-g(v)}{s},v\right\rangle=\left\langle L_{\Delta_p}v,v\right\rangle.
$$
This entails
$$
\limsup_{s\to0^+}\frac{\lambda^1_{s,p}(\Omega)-1}{s}\leq\inf_{\substack{v\in C_c^\infty(\Omega)\\\|v\|_{L^p(\Omega)=1}}}\left<L_{\Delta_p}v,v\right>.
$$
By using \eqref{reltn btwn operator and bilinear} and the density property of $C_c^{\infty}(\Omega)$, Proposition \ref{dense}, together with \eqref{first ev def}, we obtain
\begin{equation}\label{limsup for log ev}
\limsup_{s\to0^+}\frac{\lambda^1_{s,p}(\Omega)-1}{s}\leq\lambda^1_{L,p}(\Omega).
\end{equation}
\textbf{Step 2:} We claim that the sequence $\{\phi_s\}$ of functions with $\|\phi_s\|_{L^p(\Omega)}=1$ is bounded uniformly in $X^p_0(\Omega)$. For this by \eqref{limsup for log ev}, we have as $s\to0^+$
\begin{multline}\label{lower bdd for log ev:eq1}
\lambda^1_{L,p}(\Omega)+o(1)
      \geq\frac{\lambda^1_{s,p}(\Omega)-1}{s}
      =\frac{[\phi_s]^p_{W^{s,p}(\rn)}-1}{s}\\
      =\frac{C_{N,s,p}}{2s}\iint_{\substack{x,\,y\in\rn\\|x-y|\leq1}}\frac{|\phi_s(x)-\phi_s(y)|^p}{|x-y|^{N+sp}}\,dxdy+\frac{C_{N,s,p}}{2s}\iint_{\substack{x,\,y\in\rn\\|x-y|>1}}\frac{|\phi_s(x)-\phi_s(y)|^p}{|x-y|^{N+sp}}\,dxdy-\frac{1}{s}. 
\end{multline}
Note that,
$$
\int_{\Omega}|\phi_s(z)|^p\int_{\rn\setminus B_1(w)}\frac{dw}{|z-w|^{N+sp}}dz=\frac{\omega_N}{sp}\|\phi_s\|_{L^p(\Omega)}^p=\frac{\omega_N}{sp}.
$$
Thus, from \eqref{lower bdd for log ev:eq1} we obtain, for $s\to 0^+$
\begin{equation}\label{ineq:xyz}
\begin{split}
 \lambda^1_{L,p}(\Omega)+o(1)
 &\geq\frac{C_{N,s,p}}{2s}\iint_{\substack{x,\,y\in\rn\\|x-y|\leq1}}\frac{|\phi_s(x)-\phi_s(y)|^p}{|x-y|^{N+sp}}\,dxdy\\
 &+\frac{C_{N,s,p}}{2s}\iint_{\substack{x,\,y\in\rn\\|x-y|>1}}\frac{|\phi_s(x)-\phi_s(y)|^p-\left(|\phi_s(x)|^p+|\phi_s(y)|^p\right)}{|x-y|^{N+sp}}\,dxdy+f_{N,p}(s),
\end{split}
\end{equation}
where 
$$
f_{N,p}(s)=\frac{C_{N,s,p}\omega_N}{s^2p}-\frac{1}{s}.
$$ 
Note that since also $\supp\,\phi_s=\overline{\Omega}$, we have as $s\to0^+$
\begin{align*}
&\iint_{\substack{x,\,y\in\rn\\|x-y|>1}}\frac{|\phi_s(x)-\phi_s(y)|^p-\left(|\phi_s(x)|^p+|\phi_s(y)|^p\right)}{|x-y|^{N+sp}}\,dxdy\\
&=\frac{2}{C_{N,p}}\cF_p(\phi_s,\phi_s)+\iint_{\substack{x,\,y\in\rn\\|x-y|>1}}\frac{|\phi_s(x)-\phi_s(y)|^p-\left(|\phi_s(x)|^p+|\phi_s(y)|^p\right)}{|x-y|^{N}}\left(\frac{1}{|x-y|^{sp}}-1\right)\,dxdy\\
&=\frac{2}{C_{N,p}}\cF_p(\phi_s,\phi_s)+\iint_{\substack{x,\,y\in\Omega\\|x-y|>1}}\frac{|\phi_s(x)-\phi_s(y)|^p-\left(|\phi_s(x)|^p+|\phi_s(y)|^p\right)}{|x-y|^{N}}\left(\frac{1}{|x-y|^{sp}}-1\right)\,dxdy\\
&=\frac{2}{C_{N,p}}\cF_p(\phi_s,\phi_s)+\iint_{\substack{x,\,y\in\Omega\\|x-y|>1}}\frac{|\phi_s(x)-\phi_s(y)|^p-\left(|\phi_s(x)|^p+|\phi_s(y)|^p\right)}{|x-y|^{N}}\Big(-sp\ln(|x-y|)+o(s)\Big)\,dxdy.
\end{align*}
Note here, that with $m(t)=c+p\ln|t|$ for some $c>0$ we have for $s\to 0^+$, for a constant $C_p$ depending only on $p$
\begin{align*}
\iint_{\substack{x,\,y\in\Omega\\|x-y|>1}}&\frac{|\phi_s(x)-\phi_s(y)|^p-\left(|\phi_s(x)|^p+|\phi_s(y)|^p\right)}{|x-y|^{N}}\Big(-sp\ln(|x-y|)+o(s)\Big)\,dxdy\\
&\leq s\,C_p\int_{\Omega}\int_{\Omega\setminus B_1(x)}\frac{|\phi_s(x)|^p+|\phi_s(y)|^p}{|x-y|^{N}}m(|x-y|)\,dxdy
\leq s\,2\,C_p\int_{\Omega}|\phi_s(x)|^p\int_{B_R(x)\setminus B_1(x)}\frac{m(|x-y|)}{|x-y|^{N}}\,dydx,
\end{align*}
where $R=\diam(\Omega)+1$. Since
$$
\int_{B_R(x)\setminus B_1(x)}\frac{m(|x-y|)}{|x-y|^{N}}\,dy=\omega_N\int_1^R \frac{m(r)}{r}\,dr=c\ln(R)+\int_1^R\frac{\ln(r)}{r}\,dr=c\ln(R)+\frac{\ln^2(R)}{2}<\infty.
$$
Thus we have, for $s\to0^+$, from \eqref{ineq:xyz}, since $C_{N,s,p}\to 0$ for $s\to 0^+$
\begin{equation}\label{lower bdd of log ev: eq2}
\begin{split}
 \lambda^1_{L,p}(\Omega)+o(1) &\geq \frac{C_{N,s,p}}{2s}\iint_{\substack{x,\,y\in\rn\\|x-y|\leq1}}\frac{|\phi_s(x)-\phi_s(y)|^p}{|x-y|^{N+sp}}\,dxdy+\frac{C_{N,s,p}}{sC_{N,p}}\cF_p(\phi_s,\phi_s)+f_{N,p}(s)\\
  &\geq \frac{C_{N,s,p}}{sC_{N,p}}\Big( \cE_p(\phi_s,\phi_s)+ \cF_p(\phi_s,\phi_s)\Big)+f_{N,p}(s).
\end{split}
\end{equation}
Note here, that with a similar calculation as above, we have
$$
\cF_p(\phi_s,\phi_s)\leq C_{N,p}\int_{\Omega}|\phi_s(x)|^p\int_{B_R(x)\setminus B_1(x)}|x-y|^{-N}\,dydx=C_{N,p}\omega_N\ln(R),
$$
as $\phi_s$ is $L^p$-normalized. Recalling, $C_{N,s,p}=sd_{N,p}(s)$ and since $$d_{N,p}(s)\xrightarrow{s\to0^+}C_{N,p}=\frac{p\Gamma\left(\frac{N}{2}\right)}{2\pi^{N/2}}=\frac{p}{\omega_N}.$$
Using this we have $$f_{N,p}(s)=\frac{1}{s}\left(\frac{C_{N,s,p}\omega_N}{sp}-1\right)\to\rho_{N}(p)\text{ as }s\to0^+.$$ Therefore, from the estimate \eqref{lower bdd of log ev: eq2} and above, we obtain
\begin{equation*}
    \begin{split}
        \cE_{p}(\phi_s,\phi_s)\leq\frac{s\,C_{N,p}}{C_{N,s,p}}\left[\lambda^1_{L,p}+o(1)-f_{N,p}(s)\right]-\cF_{p}(\phi_s,\phi_s).
    \end{split}
\end{equation*}
Hence for $s\to0^+$, we have
\begin{equation*}
   |\cE_{p}(\phi_s,\phi_s)|\leq(1+o(1))|\lambda^1_{L,p}+o(1)-\rho_N(p)|+C_{N,p}\omega_N\ln(R). 
\end{equation*}
From this it follows that $\{\phi_s\}$ is bounded in $X^p_0(\Omega)$, thanks to the fractional Poincar\'e inequality.
\smallskip

\noindent\textbf{Step 3:} Let $u_1$ be the first positive $L^p$-normalized eigenfunction of $L_{\Delta_p}$ in $\Omega$ given by Theorem \ref{thm:first eigenfunction main}. To show \eqref{eigenfunction conv}, we use the method of contradiction. Suppose \eqref{eigenfunction conv} is not true that is there exists $\epsilon>0$ and a sequence $\{s_n\}$ of real numbers such that $s_n\to0$ and for any $n$ we have
\begin{equation}\label{ef not conv}
    \|\phi_{s_{n}}-u_1\|_{L^p(\Omega)}\geq\epsilon.
\end{equation}
By Step 2, we have the sequence $\{\phi_{s_n}\}$ is bounded in $X^p_0(\Omega)$. Thus up to a subsequence, we obtain
\begin{equation}\label{properties of usual ef}
\phi_{s_n}\rightharpoonup u_0\text{ in }X^p_0(\Omega),\,\,\phi_{s_n}\to u_0\text{ in }L^p(\Omega),\text{ and }\frac{\lambda^1_{s_n}-1}{s_n}\to\lambda^*\in\left[-\infty,\lambda^1_{L,p}\right]\,\,\,\text{ as }n\to\infty. 
\end{equation}
We claim that $u_0$ is an eigenfunction of $L_{\Delta_p}$ corresponding to the eigenvalue $\lambda^*$. Let $v\in C_c^\infty(\Omega)$.  Since $\{\phi_{s_n}\}$ is uniformly bounded and by \eqref{properties of usual ef} together with the dominated convergence theorem, we obtain
\begin{equation}\label{dual conv of ef}
\int_{\Omega}g(\phi_{s_n})v\,dx\to\int_{\Omega}g(u_0)v\,dx\,\,\text{ as }n\to\infty.
\end{equation}
Therefore, by \eqref{properties of usual ef}, \eqref{dual conv of ef}, and Lemma \ref{conv btwn functional} ---passing to a further subsequence which is still denoted by $\{\phi_{s_n}\}$---
we have for $v\in C_c^\infty(\Omega)$
\begin{equation}\label{lambda* is ev}
\begin{split}
\lambda^*\int_{\Omega}g(u_0)v\,dx=\lim_{n\to\infty}\frac{\lambda^1_{s_n}-1}{s_n}\left<g(\phi_{s_n}),v\right>
&=\lim_{n\to\infty}\frac{\cE(\phi_{s_n},v)-\left<g(\phi_{s_n}),v\right>}{s_n}
=\cE_{L,p}(u_0,v),
\end{split}
\end{equation}
noting that the limit given by Lemma \ref{conv btwn functional} must coincide with $u_0$ due to the $L^p$ convergence.
Since, we may choose $v\in C_c^\infty(\Omega)$ such that $\int_{\Omega}g(u_0)v\,dx>0$, we conclude from \eqref{lambda* is ev} and by density,  that $\lambda^*>-\infty$ and
$$
\cE_{L,p}(u_0,v)=\lambda^*\int_{\Omega}g(u_0)v\,dx\,\,\text{ for all }v\in X^p_0(\Omega).
$$
Therefore, we get $(\lambda^*,\,u_0)$ is an eigenpair for $L_{\Delta_p}$. Again, by \eqref{properties of usual ef} we have $\lambda^*\leq\lambda^1_{L,p}$ and thus by definition of $\lambda^1_{L,p}$, we have $\lambda^1_{L,p}=\lambda^*.$ Further, $\|u_0\|_{L^p(\Omega)}=1$ and $u_0\geq0$, hence $u_0=u_1$ is the unique positive eigenfunction of $L_{\Delta_p}$ in $\Omega$. This gives a contradiction to \eqref{ef not conv} and therefore, we proved \eqref{eigenfunction conv}.
\smallskip

\noindent\textbf{Step 4: }It remains to prove the reverse inequality of \eqref{limsup for log ev}. For this, let $\lambda_*:=\liminf_{s\to0^+}\frac{\lambda^1_{s,p}(\Omega)-1}{s}$ and consider a sequence $\{s_n\}\subset(0,1)$ with $s_n\to0$ such that 
$$
\frac{\lambda^1_{s_n,\,p}(\Omega)-1}{s_n}\to\lambda_*\text{ as }n\to\infty.
$$
Then by Step 3, we have $\phi_{s_n}\to u_1$ and by the similar argument as in Step 3, we obtain
$$
\lambda_*>-\infty, \text{ and } \cE_{L,p}(u_1,v)=\lambda_*\int_{\Omega}g(u_1)v\,dx\,\,\text{ for all }v\in X^p_0(\Omega).
$$
This gives that $\lambda_*=\lambda^1_{L,p}$, and thus combining with \eqref{limsup for log ev} gives the desired result. This completes the proof of the theorem.
\end{proof}

\begin{proof}[Proof of Corollary \ref{faber krahn}]
By the Faber-Krahn inequality for the fractional $p$-Laplacian (see for example \cite[Theorem~3.5]{BrLiPa}), we have 
$$
\lambda^1_{s,p}(B^{(m)})\leq\lambda^1_{s,p}(\Omega)\text{ for all }s\in(0,1).
$$
Therefore, using this and Theorem \ref{relation between evs of log p Lap and frac p Lap}, we obtain
$$
\lambda^1_{L,p}(B^{(m)})=\lim_{s\to0^+}\frac{\lambda^1_{s,p}(B^{(m)})-1}{s}\leq\lim_{s\to0^+}\frac{\lambda^1_{s,p}(\Omega)-1}{s}=\lambda^1_{L,p}(\Omega),
$$
and this gives the desired result.
\end{proof}

\begin{thm}\label{maximum principle}
Let $\Omega\subset \R^N$ be an open bounded set. $L_{\Delta_p}$ satisfies the maximum principle in $\Omega$  if and only if $\lambda_{L,p}^1(\Omega)>0$.
\end{thm}
\begin{proof}
Let $\phi_1\in X^p_0(\Omega)$ be the uniquely determined $L^p$-normalized first eigenfunction of $L_{\Delta_p}$ in $\Omega$, which is nonnegative in $\R^N$. If $\lambda_{L,p}^1(\Omega)\leq 0$, then $u=-\phi_1$ satisfies $u=0$ in $\R^N\setminus \Omega$, $u\lneq 0$ in $\Omega$ and $L_{\Delta_p}u=\lambda_{L,p}^1(\Omega)u\geq 0$ in $\Omega$. Thus, the maximum principle does not hold.\\
If otherwise $\lambda_{L,p}^1(\Omega)> 0$, then the maximum principle holds by Lemma \ref{maximum principle-pre} (with $c\equiv 0$).
\end{proof}

\begin{proof}[Proof of Theorem \ref{intro:maximum principle}]
This follows immediately from Theorem \ref{maximum principle} and its proof combined with Theorem \ref{smp}.
\end{proof}

In the following result, we collect some useful properties of $h_{\Omega}$ defined in \eqref{h-omega function} (see also Lemma \ref{some properties}(3)). The results are very slight adjustments from the case $p=2$ proven in \cite{K22} as $h_{\Omega}$ varies in $p$ only through the constant in front. The proofs have been made available to us through personal communication, and we include them for the reader's convenience.
\begin{lemma}\label{properties homega}
Let $\Omega\subset \R^N$ be an open and bounded set, and let $x\in \Omega$. Then, the following are true for $h_{\Omega}$.
\begin{enumerate}
    \item For any $\epsilon\in(0,\delta(x)]$ with $\delta(x)=\dist(x,\R^N\setminus \Omega)$, we have
    $$
    h_{\Omega}(x)=p\ln(\epsilon^{-1})-C_{N,p}\int_{\Omega\setminus B_{\epsilon}(x)}|x-y|^{-N}\,dy.
    $$
    \item It holds
    $$
     h_{\Omega}(x)\geq \frac{p}{N}\ln\left(\frac{|B_1|}{|\Omega|}\right),
    $$
    in particular, $h_{\Omega}\to\infty$ for $|\Omega|\to 0$.
    \item For $r>0$, it holds
    $$
    h_{r\Omega}(rx)=h_{\Omega}(x)-p\ln(r).
    $$
\end{enumerate}
\end{lemma}
\begin{proof}
Denote by $\omega_N$ the $(N-1)$-dimensional volume of $\partial B_1$. Then $C_{N,p}\omega_N=p$. Moreover, for $\epsilon\in(0,\delta(x)]$ we have
\begin{align*}
h_{\Omega}(x)&=C_{N,p}\Bigg(\,\int_{[B_1(x)\setminus B_\epsilon(x)]\setminus \Omega}|x-y|^{-N}\,dy-\int_{\Omega\setminus B_\epsilon(x)}|x-y|^{-N}\,dy+\int_{[B_1(x)\setminus B_\epsilon(x)]\cap \Omega}|x-y|^{-N}\,dy\Bigg)\\
&=p\int_{\epsilon}^{1}t^{-1}\,dt-\int_{\Omega\setminus B_\epsilon(x)}|x-y|^{-N}\,dy=p\ln(\epsilon^{-1})-\int_{\Omega\setminus B_\epsilon(x)}|x-y|^{-N}\,dy.
\end{align*}
Thus (1) follows. Next, let $r>0$ such that $|\Omega|=|B_r|$, that is, 
$$
r=\Big(\frac{|\Omega|}{|B_1|}\Big)^\frac{1}{N}.
$$
Notice that $r\geq \delta(x)=:\epsilon$, and thus $B_{\epsilon}(x)\subset B_r(x)\cap \Omega$. Moreover, since $|\Omega\setminus B_r(x)|=|B_r(x)\setminus \Omega|$ it follows that
\begin{align*}
\int_{\Omega\setminus B_\epsilon(x)}|x-y|^{-N}\,dy&=\int_{B_r(x)\setminus B_{\epsilon}(x)}|x-y|^{-N}\,dy-\int_{[B_r(x)\setminus\Omega] \setminus B_{\epsilon}(x)}|x-y|^{-N}\,dy+\int_{[\Omega\setminus B_r(x)]\setminus B_{\epsilon}(x)}|x-y|^{-N}\,dy\\
&=\int_{B_r(x)\setminus B_{\epsilon}(x)}|x-y|^{-N}\,dy-\int_{B_r(x)\setminus\Omega}|x-y|^{-N}\,dy+\int_{\Omega\setminus B_r(x)}|x-y|^{-N}\,dy\\
&\leq \omega_N\int_{\epsilon}^rt^{-1}\,dt +r^{-N}\Big(- |B_r(x)\setminus\Omega|+|\Omega\setminus B_r(x)|\Big)=\omega_N\Big(\ln(r)+\ln(\epsilon^{-1})\Big).
\end{align*}
With (1) it follows that
$$
h_{\Omega}(x)\geq p\ln(\epsilon^{-1}) -p\Big(\ln(r)+\ln(\epsilon^{-1})\Big)=p\ln(r^{-1}),
$$
and (2) follows by the explicit representation of $r$. Finally, (3) follows by a simple direct computation.
\end{proof}

In the next lemma, we estimate $\lambda_{L,p}^1(\Omega)$ using the properties of $h_{\Omega}$. In particular, these imply the positivity of $\lambda_{L,p}^1(\Omega)$ if $|\Omega|$ is small enough.

\begin{lemma}\label{how to ensure lambda1 positive}
For $\Omega\subset \R^N$ open and bounded set, it holds
$$
\frac{p}{N}\ln\left(\frac{|B_1|}{|\Omega|}\right)+\rho_{N}(p)\leq \lambda_{L,p}^1(\Omega)\leq \frac{1}{|\Omega|}\int_{\Omega}h_{\Omega}(x)\,dx+\rho_{N}(p).
$$
 Moreover, if $h_{\Omega}+\rho_{N}(p)\gneq 0$ in $\Omega$, then $\lambda_{L,p}^1(\Omega)>0$.
\end{lemma}
\begin{proof}
For the first statement, note that from Proposition \ref{alternative} we have with $u=\frac{1}{|\Omega|^{1/p}}1_{\Omega}\in X^p_0(\Omega)$ that $\|u\|_{L^p(\Omega)}=1$ and thus
$$
\lambda_{L,p}^1(\Omega)\leq \cE_{L,p}(u,u)=\frac{1}{|\Omega|}\int_{\Omega}h_{\Omega}(x)\,dx+\rho_{N}(p).
$$
While with $u_1$ being the $L^p$-normalized extremal for $\lambda_{L,p}^1(\Omega)$, which is positive in $\Omega$
\begin{align*}
  \lambda_{L,p}^1(\Omega)&\geq \int_{\Omega}\Big(h_{\Omega}(x)+\rho_{N}(p)\Big)u_1(x)^{p}\,dx=\int_{\Omega}h_{\Omega}(x)u_1(x)^{p}\,dx+\rho_{N}(p)
  \geq \frac{p}{N}\ln\left(\frac{|B_1|}{|\Omega|}\right)+\rho_{N}(p),
\end{align*}
by Lemma \ref{properties homega}(3). The second statement follows immediately from the definition of $\lambda_{L,p}^1(\Omega)$ with Proposition \ref{alternative} and the strict positivity of the first eigenfunction by Theorem \ref{thm:first eigenfunction main}. 
\end{proof}

\begin{remark}\label{remark mp} 
Let us mention that Theorem \ref{maximum principle} for solutions in $X^p_0(\Omega)$ can be reformulated using Theorem \ref{smp} and Corollary \ref{bounded2}. Indeed, it holds:\\
Let $\Omega\subset \R^N$ be an open set and $f,c\in L^{\infty}(\Omega)$. If $f\gneq0$ and $\lambda_{L,1}^p(\Omega)>\|c^+\|_{L^{\infty}(\Omega)}$, then any supersolution $u\in X^p_0(\Omega)$ of $L_{\Delta_p}u=c(x)g(u)+f$ in $\Omega$, $u=0$ in $\R^N\setminus \Omega$ is positive.
\end{remark}

\begin{cor}[Small volume maximum principle]\label{small volume}
Let $k>0$. Then there is $\delta>0$ with the following property. For any open bounded set $\Omega\subset \R^N$ and $c\in L^{\infty}(\Omega)$ with $\|c^+\|_{L^{\infty}(\Omega)}\leq k$ the following holds: If $u\in X^p_0(\Omega)\cap L^{\infty}(\Omega)$ satisfies weakly
$$
L_{\Delta_p}u\geq c(x)g(u)\quad\text{in $\Omega$},\quad u=0\quad\text{in $\R^N\setminus \Omega$},
$$
then, either $u\equiv 0$ in $\Omega$ or $u>0$ in $\Omega$ in the sense that
$$
\essinf_{K}u>0\quad\text{for all compact sets $K\subset \Omega$}.
$$
\end{cor}
\begin{proof}
This follows immediately from Lemma \ref{maximum principle-pre} and Theorem \ref{smp} (see, also Remark \ref{remark mp}) and Lemma \ref{properties homega}, noting that we have  $\lambda_{L,1}^p(\Omega)>\|c^+\|_{L^{\infty}(\Omega)}$, if 
$$
\frac{p}{N}\ln\left(\frac{|B_1|}{|\Omega|}\right)+\rho_{N}(p)\geq k\geq \|c^+\|_{L^{\infty}(\Omega)}.
$$
\end{proof}

\subsection*{Acknowledgments} The authors would like to thank Tobias Weth for the fruitful discussions on this subject, Adimurthi and Tomasz Grzywny for pointing out some references, and an anonymous referee for numerous comments that allowed us to improve the manuscript.  F. Sk is supported by the Alexander von Humboldt foundation, Germany.

\end{document}